\makeatletter

\documentclass[
    11pt,
    a4paper,
    oneside,
    openright,
    center,
    chapterbib,
    crosshair,
    reqno,
    headcount,
    headline,
    indent,
    indentfirst=false,
    portrait,
    phonetic,
    oldernstyle,
    onecolumn,
    sfbold,
    upper,
]{amsart}

\let\th@plain\relax

\PassOptionsToPackage{T2A}{fontenc} 
\PassOptionsToPackage{utf8}{inputenc} 
\PassOptionsToPackage{cyrpart}{cyrillic}
\PassOptionsToPackage{british,ngerman,russian}{babel}
\PassOptionsToPackage{
    foot,
}{amsaddr}
\PassOptionsToPackage{
    style=numeric-comp,
    citestyle=authoryear,
}{biblatex}
\PassOptionsToPackage{
    english,
    ngerman,
    russian,
    capitalise,
}{cleveref}
\PassOptionsToPackage{
    margin=10pt,
    position=bottom,
    justification=centering,
    font=small,
    labelformat=simple,
    labelfont={sc},
    labelsep={period},
    textfont={it},
}{caption}
\PassOptionsToPackage{
    margin=10pt,
    position=bottom,
    justification=centering,
    font=small,
    labelformat=parens,
    labelfont={bf},
    labelsep={space},
    textfont={it},
}{subcaption}
\PassOptionsToPackage{framemethod=TikZ}{mdframed}
\PassOptionsToPackage{
    bookmarks=true,
    bookmarksopen=false,
    bookmarksopenlevel=0,
    bookmarkstype=toc,
    raiselinks=true,
    colorlinks=true,
    hyperfigures=true,
    anchorcolor=red,
    citecolor=green,
    linkcolor=blue,
    urlcolor=blue,
    hypertexnames=false, 
}{hyperref} 
\PassOptionsToPackage{normalem}{ulem}
\PassOptionsToPackage{
    amsmath,
    thmmarks,
}{ntheorem}
\PassOptionsToPackage{
}{thmtools}
\PassOptionsToPackage{
    table,
}{xcolor}
\PassOptionsToPackage{
    all,
    color,
    curve,
    frame,
    import,
    knot,
    line,
    movie,
    rotate,
    textures,
    tile,
    tips,
    web,
    xdvi,
}{xy}
\PassOptionsToPackage{
    reset,
    left=1in,
    right=1in,
    top=20mm,
    bottom=20mm,
    heightrounded,
}{geometry}
\PassOptionsToPackage{
    symbol*, 
    multiple, 
}{footmisc}
\PassOptionsToPackage{overload}{textcase}
\PassOptionsToPackage{indentafter}{titlesec}
\PassOptionsToPackage{
    shortcuts, 
}{extdash}

\usepackage{amsaddr}
\usepackage{amsfonts}
\usepackage{amsmath}
\usepackage{amssymb}
\usepackage{ntheorem} 
\usepackage{thmtools} 
\usepackage{array}
\usepackage{babel}
\usepackage{braket}
\usepackage{bbding}
\usepackage{bm}
\usepackage{bbm}
\usepackage{bibentry}
\usepackage{booktabs}
\usepackage{bold-extra} 
\usepackage{calc}
\usepackage{cancel}
\usepackage{caption} 
\usepackage{changepage}
\usepackage{cjhebrew}
\usepackage{cmlgc}
\usepackage{colonequals}
\usepackage{color}
\usepackage{comment}
\usepackage{datetime}
\usepackage{dsfont}
\usepackage{etex}
\usepackage{etoolbox}
\usepackage{eurosym}
\usepackage{extdash}
\usepackage{fancybox}
\usepackage{fancyhdr}
\usepackage{float}
\usepackage{fontenc}
\usepackage{footmisc}
\usepackage{fp}
\usepackage{geometry}
\usepackage{graphicx}
\usepackage{ifpdf}
\usepackage{ifthen}
\usepackage{ifoddpage}
\usepackage{ifnextok}  
\usepackage{inputenc}
\usepackage{latexsym}
\usepackage{lineno}
\usepackage{listings}
\usepackage{longtable}
\usepackage{lscape}
\usepackage{mathtools} 
\usepackage{mathrsfs}
\usepackage{multicol}
\usepackage{multirow}
\usepackage{nameref}
\usepackage{nowtoaux}
\usepackage{paralist}
\usepackage{enumerate} 
\usepackage{enumitem} 
\usepackage{pgf}
\usepackage{pgfplots}
\usepackage{phonetic}
\usepackage{proof}
\usepackage{qtree}
\usepackage{refcount}
\usepackage{savesym}
\usepackage{stmaryrd}
\usepackage{synttree}
\usepackage{subcaption}
\usepackage{suffix}
\usepackage{yfonts} 
\usepackage{textcase} 
\usepackage{tikz}
\usepackage{xy}
\usepackage{ulem} 
\usepackage{wrapfig}
\usepackage{xcolor}
\usepackage{xspace}
\usepackage{xstring}
\usepackage{arydshln}
\usepackage{hyperref}
\usepackage{cleveref} 

\pgfplotsset{compat=newest}
\usetikzlibrary{math}

\usetikzlibrary{
    angles,
    arrows,
    automata,
    calc,
    decorations,
    decorations.pathmorphing,
    decorations.pathreplacing,
    positioning,
    patterns,
    patterns.meta,
    quotes,
}

\savesymbol{corresponds}
\savesymbol{Diamond}
\savesymbol{emptyset}
\savesymbol{ggg}
\savesymbol{int}
\savesymbol{lll}
\savesymbol{RectangleBold}
\savesymbol{langle}
\savesymbol{rangle}
\savesymbol{hookrightarrow}
\savesymbol{hookleftarrow}
\savesymbol{Asterisk}
\usepackage{mathabx}
\usepackage{wasysym}

\restoresymbol{x}{corresponds}
\restoresymbol{x}{Diamond}
\restoresymbol{x}{emptyset}
\restoresymbol{x}{ggg}
\restoresymbol{x}{int}
\restoresymbol{x}{lll}
\restoresymbol{x}{RectangleBold}
\restoresymbol{x}{langle}
\restoresymbol{x}{rangle}
\restoresymbol{x}{hookrightarrow}
\restoresymbol{x}{hookleftarrow}
\restoresymbol{x}{Asterisk}

\ifpdf
    \usepackage{pdfcolmk}
\fi

\usepackage{mdframed}

\DeclareFontFamily{U}{MnSymbolA}{}
\DeclareFontShape{U}{MnSymbolA}{m}{n}{
    <-6> MnSymbolA5
    <6-7> MnSymbolA6
    <7-8> MnSymbolA7
    <8-9> MnSymbolA8
    <9-10> MnSymbolA9
    <10-12> MnSymbolA10
    <12-> MnSymbolA12
}{}
\DeclareFontShape{U}{MnSymbolA}{b}{n}{
    <-6> MnSymbolA-Bold5
    <6-7> MnSymbolA-Bold6
    <7-8> MnSymbolA-Bold7
    <8-9> MnSymbolA-Bold8
    <9-10> MnSymbolA-Bold9
    <10-12> MnSymbolA-Bold10
    <12-> MnSymbolA-Bold12
}{}
\DeclareSymbolFont{MnSymA}{U}{MnSymbolA}{m}{n}
\DeclareMathSymbol{\lcirclearrowright}{\mathrel}{MnSymA}{252}
\DeclareMathSymbol{\lcirclearrowdown}{\mathrel}{MnSymA}{255}
\DeclareMathSymbol{\rcirclearrowleft}{\mathrel}{MnSymA}{250}
\DeclareMathSymbol{\rcirclearrowdown}{\mathrel}{MnSymA}{251}
\DeclareFontFamily{U}{MnSymbolC}{}
\DeclareSymbolFont{MnSyC}{U}{MnSymbolC}{m}{n}
\DeclareFontShape{U}{MnSymbolC}{m}{n}{
    <-6>  MnSymbolC5
    <6-7>  MnSymbolC6
    <7-8>  MnSymbolC7
    <8-9>  MnSymbolC8
    <9-10> MnSymbolC9
    <10-12> MnSymbolC10
    <12->   MnSymbolC12%
}{}
\DeclareMathSymbol{\powerset}{\mathord}{MnSyC}{180}
\DeclareMathSymbol{\righthalfcap}{\mathbin}{MnSyC}{186}

\DeclareMathAlphabet{\mathpzc}{OT1}{pzc}{m}{it}
\DeclareMathAlphabet{\blackboardfont}{U}{BOONDOX-ds}{m}{n}

\def\boolwahr{true}
\def\boolfalsch{false}
\def\boolleer{}

\let\boolinappendix\boolfalsch
\let\boolinmdframed\boolfalsch

\newcount\bufferctr
\newcount\bufferreplace

\newlength\rtab
\newlength\gesamtlinkerRand
\newlength\gesamtrechterRand
\newlength\ownspaceabovethm
\newlength\ownspacebelowthm
\newlength\aboveequation
\newlength\belowequation
\def\secnumberingpt{.}
\def\secnumberingseppt{.}
\def\subsecnumberingseppt{}
\def\thmnumberingpt{.}
\def\thmnumberingseppt{}
\def\thmForceSepPt{.}

\definecolor{leer}{gray}{1}
\definecolor{boxgrau}{gray}{0.85}
\definecolor{dunkelgrau}{gray}{0.5}
\definecolor{maroon}{rgb}{0.6901961,0.1882353,0.3764706}
\definecolor{dunkelgruen}{rgb}{0.015625,0.363281,0.109375}
\definecolor{dunkelrot}{rgb}{0.5450980392,0,0}
\definecolor{dunkelblau}{rgb}{0,0,0.5450980392}
\definecolor{blau}{rgb}{0,0,1}
\definecolor{newresult}{rgb}{0.6,0.6,0.6}
\definecolor{improvedresult}{rgb}{0.9,0.9,0.9}
\definecolor{hervorheben}{rgb}{0,0.9,0.7}
\definecolor{starkesblau}{rgb}{0.1019607843,0.3176470588,0.8156862745}
\definecolor{achtung}{rgb}{1,0.5,0.5}
\definecolor{frage}{rgb}{0.5,1,0.5}
\definecolor{schreibweise}{rgb}{0,0.7,0.9}
\definecolor{axiom}{rgb}{0,0.3,0.3}
\definecolor{drawing_light_grey}{gray}{0.85}
\definecolor{background_light_grey}{gray}{0.95}

\def\let@name#1#2{
    \expandafter\let\csname #1\expandafter\endcsname\csname #2\endcsname\relax
}
\DeclareRobustCommand\crfamily{\fontfamily{ccr}\selectfont}
\DeclareTextFontCommand{\textcr}{\crfamily}

\def\ifthenelseleer#1#2#3{\ifthenelse{\equal{#1}{}}{#2}{#1#3}}
\def\bedingtesspaceexpand#1#2#3{\ifthenelseleer{\csname #1\endcsname}{#3}{#2#3}}

\def\hraum{\null\hfill\null}

\def\nvraum{\@ifnextchar\bgroup{\nvraum@c}{\nvraum@bes}}
    \def\nvraum@c#1{\vspace*{-#1\baselineskip}}
    \def\nvraum@bes{\vspace*{-\baselineskip}}
\def\forceaddspace{\relax\ifmmode\else\@\xspace\fi}
\def\forceremovespace{\relax\ifmmode\else\expandafter\@gobble\fi}

\def\send@toaux#1{\@bsphack\protected@write\@auxout{}{\string#1}\@esphack}

\def\rlabel#1[#2]#3#4#5{#5\rlabel@aux{#1}[#2]{#3}{#4}{#5}}
    \def\rlabel@aux#1[#2]#3#4#5{%
        \send@toaux{\newlabel{#1}{{\@currentlabel}{\thepage}{{\unexpanded{#5}}}{#2.\csname the#2\endcsname}{}}}\relax%
    }

\def\tag@rawscheme#1#2[#3]#4#5{\@ifnextchar[{\tag@rawscheme@{#1}{#2}[#3]{#4}{#5}}{\tag@rawscheme@{#1}{#2}[#3]{#4}{#5}[*]}}
    \def\tag@rawscheme@#1#2[#3]#4#5[#6]{\@ifnextchar\bgroup{\tag@rawscheme@@{#1}{#2}[#3]{#4}{#5}[#6]}{\tag@rawscheme@@{#1}{#2}[#3]{#4}{#5}[#6]{}}}
    \def\tag@rawscheme@@#1#2[#3]#4#5[#6]#7{%
        \ifthenelse{\equal{#6}{*}}{%
            \ifthenelse{\equal{#7}{\boolleer}}{\refstepcounter{#3}#4\csname the#3\endcsname#5}{#4#7#5}%
        }{%
            \refstepcounter{#3}#4%
            \ifthenelse{\equal{#7}{\boolleer}}{\rlabel{#6}[#3]{#1}{#2}{\csname the#3\endcsname}}{\rlabel{#6}[#3]{#1}{#2}{#7}}%
            #5%
        }%
    }
\def\tag@scheme#1#2[#3]{\tag@rawscheme{#1}{#2}[#3]{\upshape(}{\upshape)}}

\def\eqtag@post#1{\makebox[0pt][r]{#1}}
\def\eqtag@pre{\tag@scheme{Eq}{Equation}[equation]}
\def\eqtag{\@ifnextchar[{\eqtag@}{\eqtag@[*]}}
    \def\eqtag@[#1]{\@ifnextchar\bgroup{\eqtag@@[#1]}{\eqtag@@[#1]{}}}
    \def\eqtag@@[#1]#2{\eqtag@post{\eqtag@pre[#1]{#2}}}

\def\eqcref#1{\text{(\ref{#1})}}
\def\punktlabel#1{\label{it:#1:\beweislabel}}
\def\punktcref#1{\eqcref{it:#1:\beweislabel}}

\def\opfromto[#1]_#2^#3{\underset{#2}{\overset{#3}{#1}}}
\def\textoverset#1#2{\overset{\text{#1}}{#2}}

\def\eqcrefoverset#1#2{\textoverset{\eqcref{#1}}{#2}}

\def\mathclap#1{#1}
\def\oberunterset#1{\@ifnextchar^{\oberunterset@oben{#1}}{\oberunterset@unten{#1}}}
    \def\oberunterset@oben#1^#2_#3{\underset{\mathclap{#3}}{\overset{\mathclap{#2}}{#1}}}
    \def\oberunterset@unten#1_#2^#3{\underset{\mathclap{#2}}{\overset{\mathclap{#3}}{#1}}}
    \def\breitunderbrace#1_#2{\underbrace{#1}_{\mathclap{#2}}}
    \def\breitoverbrace#1^#2{\overbrace{#1}^{\mathclap{#2}}}
    \def\breitunderbracket#1_#2{\underbracket{#1}_{\mathclap{#2}}}
    \def\breitoverbracket#1^#2{\overbracket{#1}^{\mathclap{#2}}}

\def\generatenestedsecnumbering#1#2#3{%
    \expandafter\gdef\csname thelong#3\endcsname{%
        \expandafter\csname the#2\endcsname%
        \secnumberingpt%
        \expandafter\csname #1\endcsname{#3}%
    }%
    \expandafter\gdef\csname theshort#3\endcsname{%
        \expandafter\csname #1\endcsname{#3}%
    }%
}
\def\generatenestedthmnumbering#1#2#3{%
    \expandafter\gdef\csname the#3\endcsname{%
        \expandafter\csname the#2\endcsname%
        \thmnumberingpt%
        \expandafter\csname #1\endcsname{#3}%
    }%
    \expandafter\gdef\csname theshort#3\endcsname{%
        \expandafter\csname #1\endcsname{#3}%
    }%
}

\providecommand{\setcounternach}{}
\renewcommand{\setcounternach}[2]{\setcounter{#1}{#2}\addtocounter{#1}{-1}}
\providecommand{\textsubscript}{}
\renewcommand{\textsubscript}[1]{${}_{\textup{#1}}$}

\def\forcepunkt#1{#1\IfEndWith{#1}{.}{}{.}}

\def\matrix#1{\left(\begin{array}{#1}}
    \def\endmatrix{\end{array}\right)}
\def\smatrix{\left(\begin{smallmatrix}}
    \def\endsmatrix{\end{smallmatrix}\right)}

\def\multiargrekursiverbefehl#1#2#3#4#5#6#7#8{%
    \expandafter\gdef\csname#1\endcsname #2##1#4{\csname #1@anfang\endcsname##1#3\egroup}
    \expandafter\def\csname #1@anfang\endcsname##1#3{#5##1\@ifnextchar\egroup{\csname #1@ende\endcsname}{#7\csname #1@mitte\endcsname}}
    \expandafter\def\csname #1@mitte\endcsname##1#3{#6##1\@ifnextchar\egroup{\csname #1@ende\endcsname}{#7\csname #1@mitte\endcsname}}
    \expandafter\def\csname #1@ende\endcsname##1{#8}
}
\multiargrekursiverbefehl{svektor}{[}{;}{]}{\begin{smatrix}}{}{\\}{\\\end{smatrix}}
\multiargrekursiverbefehl{vektor}{[}{;}{]}{\begin{matrix}{c}}{}{\\}{\\\end{matrix}}
\multiargrekursiverbefehl{vektorzeile}{}{,}{;}{}{&}{}{}
\multiargrekursiverbefehl{matlabmatrix}{[}{;}{]}{\begin{smatrix}\vektorzeile}{\vektorzeile}{;\\}{;\end{smatrix}}

\def\BeweisRichtung[#1]{\@ifnextchar\bgroup{\@BeweisRichtung@c[#1]}{\@BeweisRichtung@bes[#1]}}
    \def\@BeweisRichtung@bes[#1]{{\bfseries (#1)}}
    \def\@BeweisRichtung@c[#1]#2#3{#2~#1~#3}
\def\erzeugeBeweisRichtungBefehle#1#2{
    \expandafter\gdef\csname #1text\endcsname##1##2{\BeweisRichtung[#2]{##1}{##2}}
    \expandafter\gdef\csname #1\endcsname{%
        \@ifnextchar\bgroup{\csname #1@\endcsname}{\csname #1text\endcsname{}{}}%
    }
    \expandafter\gdef\csname #1@\endcsname##1##2{%
        \csname #1text\endcsname{\punktcref{##1}}{\punktcref{##2}}%
    }
}
\erzeugeBeweisRichtungBefehle{hinRichtung}{$\Rightarrow$}
\erzeugeBeweisRichtungBefehle{herRichtung}{$\Leftarrow$}
\erzeugeBeweisRichtungBefehle{hinherRichtung}{$\Leftrightarrow$}

\def\cal#1{\mathcal{#1}}
\def\mathfrak#1{\mbox{\usefont{U}{euf}{m}{n}#1}}

\def\rectangleblack{\text{\RectangleBold}}

\def\squareblack{\blacksquare}

\def\create@abbreviation#1#2{
    \expandafter\gdef\csname #1\endcsname{%
        #2\@ifnextchar.{%
            \relax\ifmmode\else\expandafter\@gobble\fi%
        }{%
            \relax\ifmmode\else\@\xspace\fi%
        }%
    }
}

\create@abbreviation{cf}{\emph{cf.}}
\create@abbreviation{Cf}{\emph{Cf.}}
\create@abbreviation{thatis}{\emph{i.e.}}
\create@abbreviation{idest}{\emph{i.e.}}
\create@abbreviation{Idest}{\emph{I.e.}}
\create@abbreviation{exempli}{\emph{e.g.}}
\create@abbreviation{Exempli}{\emph{E.g.}}
\create@abbreviation{etcetera}{\emph{etc.}}
\create@abbreviation{etAlia}{\emph{et al.}}
\create@abbreviation{andothers}{\emph{et al.}}
\create@abbreviation{initself}{\emph{per se}}
\create@abbreviation{perse}{\emph{per se}}
\create@abbreviation{namely}{\emph{viz.}}
\create@abbreviation{viz}{\emph{viz.}}

\create@abbreviation{Tfae}{T.\,f.\,a.\,e.}
\create@abbreviation{tfae}{t.\,f.\,a.\,e.}
\create@abbreviation{wrt}{wrt.}
\create@abbreviation{randomvar}{r.\,v.}
\create@abbreviation{iid}{i.\,i.\,d.}
\create@abbreviation{almostEverywhere}{a.\,e.}
\create@abbreviation{almostSurely}{a.\,s.}
\create@abbreviation{withoutlog}{w.\,l.\,o.\,g.}
\create@abbreviation{Withoutlog}{W.\,l.\,o.\,g.}
\create@abbreviation{resp}{resp.}
\create@abbreviation{posdef}{p.d.}
\create@abbreviation{akin}{\`{a}~la}

\def\crefname@full#1#2#3#4#5{%
    \crefname{#1}{#2}{#3}
    \Crefname{#1}{#4}{#5}
}
\def\crefname@fullmod#1#2#3#4#5{%
    \crefname@full{#1}{#2}{#3}{#4}{#5}
    \crefname@full{#1@basic}{#2}{#3}{#4}{#5}
    \crefname@full{#1@withName}{#2}{#3}{#4}{#5}
}
\crefname@full{chapter}{chapter}{chapters}{Chapter}{Chapters}
\crefname@full{appendix}{appendix}{appendices}{Appendix}{Appendices}
\crefname@full{section}{section}{sections}{Section}{Sections}
\crefname@full{subsection}{section}{sections}{Section}{Sections}
\crefname@full{subsubsection}{section}{sections}{Section}{Sections}
\crefname@full{subsubsubsection}{section}{sections}{Section}{Sections}
\crefname@full{table}{table}{tables}{Table}{Tables}
\crefname@full{figure}{figure}{figures}{Figure}{Figures}
\crefname@full{subfigure}{figure}{figures}{Figure}{Figures}

\crefname@fullmod{thm}{theorem}{theorems}{Theorem}{Theorems}
\crefname@fullmod{thmStar}{theorem}{theorems}{Theorem}{Theorems}
\crefname@fullmod{conj}{conjecture}{conjectures}{Conjecture}{Conjectures}
\crefname@fullmod{cor}{corollary}{corollaries}{Corollary}{Corollaries}
\crefname@fullmod{defn}{definition}{definitions}{Definition}{Definitions}
\crefname@fullmod{conv}{convention}{conventions}{Convention}{Conventions}
\crefname@fullmod{e.g.}{example}{examples}{Example}{Examples}
\crefname@fullmod{prop}{proposition}{propositions}{Proposition}{Propositions}
\crefname@fullmod{proof}{proof}{proofs}{Proof}{Proofs}
\crefname@fullmod{lemm}{lemma}{lemmata}{Lemma}{Lemmata}
\crefname@fullmod{problem}{problem}{problems}{Problem}{Problems}
\crefname@fullmod{qstn}{question}{questions}{Question}{Questions}
\crefname@fullmod{rem}{remark}{remarks}{Remark}{Remarks}

\def\qedEIGEN#1{\@ifnextchar[{\qedEIGEN@c{#1}}{\qedEIGEN@bes{#1}}}
\def\qedEIGEN@bes#1{%
    \parfillskip=0pt
    \widowpenalty=10000
    \displaywidowpenalty=10000
    \finalhyphendemerits=0
    \leavevmode
    \unskip
    \nobreak
    \hfil
    \penalty50
    \hskip.2em
    \null
    \hfill
    #1
    \par%
}
\def\qedEIGEN@c#1[#2]{%
    \parfillskip=0pt
    \widowpenalty=10000
    \displaywidowpenalty=10000
    \finalhyphendemerits=0
    \leavevmode
    \unskip
    \nobreak
    \hfil
    \penalty50
    \hskip.2em
    \null
    \hfill
    {#1~{\small\bfseries\upshape (#2)}}%
    \par%
}
\def\qedVARIANT#1#2{
    \expandafter\def\csname ennde#1Sign\endcsname{#2}
    \expandafter\def\csname ennde#1\endcsname{\@ifnextchar[{\qedEIGEN@c{#2}}{\qedEIGEN@bes{#2}}} 
}
\qedVARIANT{OfProof}{$\squareblack$}
\qedVARIANT{OfWork}{\rectangleblack}
\qedVARIANT{OfSomething}{$\lrcorner$}
\qedVARIANT{OnNeutral}{} 

\def\ra@pretheoremwork{
    \setlength{\theorempreskipamount}{\ownspaceabovethm}
    \setlength{\theorempostskipamount}{\ownspacebelowthm}
}
\def\rathmtransfer#1#2{
    \expandafter\def\csname #2\endcsname{\csname #1\endcsname}
    \expandafter\def\csname end#2\endcsname{\csname end#1\endcsname}
}

\def\ranewthm#1#2#3[#4]{
    \theoremstyle{\current@theoremstyle}
    \theoremseparator{\current@theoremseparator}
    \theoremprework{\ra@pretheoremwork}
    \@ifundefined{#1@basic}{\newtheorem{#1@basic}[#4]{#2}}{\renewtheorem{#1@basic}[#4]{#2}}
    \theoremstyle{\current@theoremstyle}
    \theoremseparator{\thmForceSepPt}
    \theoremprework{\ra@pretheoremwork}
    \@ifundefined{#1@withName}{\newtheorem{#1@withName}[#4]{#2}}{\renewtheorem{#1@withName}[#4]{#2}}
    \theoremstyle{nonumberplain}
    \theoremseparator{\thmForceSepPt}
    \theoremprework{\ra@pretheoremwork}
    \@ifundefined{#1@star@basic}{\newtheorem{#1@star@basic}[#4]{#2}}{\renewtheorem{#1@star@basic}[#4]{#2}}
    \theoremstyle{nonumberplain}
    \theoremseparator{\thmForceSepPt}
    \theoremprework{\ra@pretheoremwork}
    \@ifundefined{#1@star@withName}{\newtheorem{#1@star@withName}[#4]{#2}}{\renewtheorem{#1@star@withName}[#4]{#2}}
    \umbauenenv{#1}{#3}[#4]
    \umbauenenv{#1@star}{#3}[#4]
    \rathmtransfer{#1@star}{#1*}
}

\def\umbauenenv#1#2[#3]{%
    \expandafter\def\csname #1\endcsname{\relax%
        \@ifnextchar[{\csname #1@\endcsname}{\csname #1@\endcsname[*]}%
    }
    \expandafter\def\csname #1@\endcsname[##1]{\relax%
        \@ifnextchar[{\csname #1@@\endcsname[##1]}{\csname #1@@\endcsname[##1][*]}%
    }
    \expandafter\def\csname #1@@\endcsname[##1][##2]{%
        \ifx*##1%
            \def\enndeOfBlock{\csname end#1@basic\endcsname}
            \csname #1@basic\endcsname%
        \else%
            \def\enndeOfBlock{\csname end#1@withName\endcsname}
            \csname #1@withName\endcsname[##1]%
        \fi%
        \def\makelabel####1{%
            \gdef\beweislabel{####1}%
            \label{\beweislabel}%
        }%
        \ifx*##2%
            \def\enndeSymbol{\qedEIGEN{#2}}
        \else%
            \def\enndeSymbol{\qedEIGEN{#2}[##2]}
        \fi
    }
    \expandafter\gdef\csname end#1\endcsname{\enndeSymbol\enndeOfBlock}
}

    \def\current@theoremstyle{plain}
    \def\current@theoremseparator{\thmnumberingseppt}
    \theoremstyle{\current@theoremstyle}
    \theoremseparator{\current@theoremseparator}
    \theoremsymbol{}

\generatenestedthmnumbering{arabic}{section}{equation}

\generatenestedthmnumbering{arabic}{section}{X}
\generatenestedthmnumbering{Roman}{section}{Xsp}

    \theoremheaderfont{\upshape\bfseries}
    \theorembodyfont{\slshape}

\ranewthm{thm}{Theorem}{\enndeOfSomethingSign}[X]
\ranewthm{lemm}{Lemma}{\enndeOfSomethingSign}[X]
\ranewthm{cor}{Corollary}{\enndeOfSomethingSign}[X]
\ranewthm{prop}{Proposition}{\enndeOfSomethingSign}[X]

    \theorembodyfont{\upshape}

\ranewthm{defn}{Definition}{\enndeOfSomethingSign}[X]
\ranewthm{conv}{Convention}{\enndeOfSomethingSign}[X]
\ranewthm{e.g.}{Example}{\enndeOfSomethingSign}[X]
\ranewthm{fact}{Fact}{\enndeOfSomethingSign}[X]
\ranewthm{problem}{Problem}{\enndeOfSomethingSign}[X]
\ranewthm{qstn}{Question}{\enndeOfSomethingSign}[X]
\ranewthm{rem}{Remark}{\enndeOfSomethingSign}[X]

    \theoremheaderfont{\itshape}
    \theorembodyfont{\upshape}

\ranewthm{proof@tmp}{Proof}{\enndeOfProofSign}[Xdisplaynone]
\rathmtransfer{proof@tmp*}{proof}

\def\shortclaim@claim{%
    \iflanguage{british}{Claim}{%
    \iflanguage{english}{Claim}{%
    \iflanguage{ngerman}{Behauptung}{%
    \iflanguage{russian}{Утверждение}{%
    Claim%
    }}}}%
}
\def\shortclaim@pf@kurz{%
    \iflanguage{british}{Pf}{%
    \iflanguage{english}{Pf}{%
    \iflanguage{ngerman}{Bew}{%
    \iflanguage{russian}{Доказательство}{%
    Pf%
    }}}}%
}
\def\shortclaim{\@ifnextchar\bgroup{\shortclaim@c}{\shortclaim@bes}}
    \def\shortclaim@c#1{\item[{\bfseries \shortclaim@claim\forceaddspace #1.}]}
    \def\shortclaim@bes{\item[{\bfseries \shortclaim@claim.}]}
\def\proofofshortclaim{\item[{\bfseries\itshape\shortclaim@pf@kurz.}]}

\newdateformat{standardshort}{\oldstylenums{\THEYEAR}.\oldstylenums{\THEMONTH}.\oldstylenums{\THEDAY}}
\newdateformat{standardcompact}{\THEYEAR\twodigit{\THEMONTH}\twodigit{\THEDAY}}
\newdateformat{standardlong}{\THEYEAR\ \monthname\ \THEDAY}
\newcolumntype{\RECHTS}[1]{>{\raggedleft}p{#1}}
\newcolumntype{\LINKS}[1]{>{\raggedright}p{#1}}
\newcolumntype{m}{>{$}l<{$}}
\newcolumntype{C}{>{$}c<{$}}
\newcolumntype{L}{>{$}l<{$}}
\newcolumntype{R}{>{$}r<{$}}
\newcolumntype{0}{@{\hspace{0pt}}}
\newcolumntype{\LINKSRAND}{@{\hspace{\@totalleftmargin}}}
\newcolumntype{h}{@{\extracolsep{\fill}}}
\newcolumntype{i}{>{\itshape}}
\newcolumntype{t}{@{\hspace{\tabcolsep}}}
\newcolumntype{q}{@{\hspace{1em}}}
\newcolumntype{n}{@{\hspace{-\tabcolsep}}}
\newcolumntype{M}[2]{%
    >{\begin{minipage}{#2}\begin{math}}%
    {#1}%
    <{\end{math}\end{minipage}}%
}
\newcolumntype{T}[2]{%
    >{\begin{minipage}{#2}}%
    {#1}%
    <{\end{minipage}}%
}
\def\punkteumgebung@genbefehl#1#2#3{
    \punkteumgebung@genbefehl@{#1}{#2}{#3}{}{}
    \punkteumgebung@genbefehl@{multi#1}{#2}{#3}{
        \setlength{\columnsep}{10pt}%
        \setlength{\columnseprule}{0pt}%
        \begin{multicols}{\thecolumnanzahl}%
    }{\end{multicols}\nvraum{1}}
}
\def\punkteumgebung@genbefehl@#1#2#3#4#5{
    \expandafter\gdef\csname #1\endcsname{
        \@ifnextchar\bgroup{\csname #1@c\endcsname}{\csname #1@bes\endcsname}
    }
        \expandafter\def\csname #1@c\endcsname##1{
            \@ifnextchar[{\csname #1@c@\endcsname{##1}}{\csname #1@c@\endcsname{##1}[\z@]}
        }
        \expandafter\def\csname #1@c@\endcsname##1[##2]{
            \@ifnextchar[{\csname #1@c@@\endcsname{##1}[##2]}{\csname #1@c@@\endcsname{##1}[##2][\z@]}
        }
        \expandafter\def\csname #1@c@@\endcsname##1[##2][##3]{
            \let\alterlinkerRand\gesamtlinkerRand
            \let\alterrechterRand\gesamtrechterRand
            \addtolength{\gesamtlinkerRand}{##2}
            \addtolength{\gesamtrechterRand}{##3}
            \advance\linewidth -##2%
            \advance\linewidth -##3%
            \advance\@totalleftmargin ##2%
            \parshape\@ne \@totalleftmargin\linewidth%
            #4
            \begin{#2}[\upshape ##1]%
                \setlength{\parskip}{0.5\baselineskip}\relax%
                \setlength{\topsep}{\z@}\relax%
                \setlength{\partopsep}{\z@}\relax%
                \setlength{\parsep}{\parskip}\relax%
                \setlength{\itemsep}{#3}\relax%
                \setlength{\listparindent}{\z@}\relax%
                \setlength{\itemindent}{\z@}\relax%
        }
        \expandafter\def\csname #1@bes\endcsname{
            \@ifnextchar[{\csname #1@bes@\endcsname}{\csname #1@bes@\endcsname[\z@]}
        }
        \expandafter\def\csname #1@bes@\endcsname[##1]{
            \@ifnextchar[{\csname #1@bes@@\endcsname[##1]}{\csname #1@bes@@\endcsname[##1][\z@]}
        }
        \expandafter\def\csname #1@bes@@\endcsname[##1][##2]{
            \let\alterlinkerRand\gesamtlinkerRand
            \let\alterrechterRand\gesamtrechterRand
            \addtolength{\gesamtlinkerRand}{##1}
            \addtolength{\gesamtrechterRand}{##2}
            \advance\linewidth -##1%
            \advance\linewidth -##2%
            \advance\@totalleftmargin ##1%
            \parshape\@ne \@totalleftmargin\linewidth%
            #4
            \begin{#2}%
                \setlength{\parskip}{0.5\baselineskip}\relax%
                \setlength{\topsep}{\z@}\relax%
                \setlength{\partopsep}{\z@}\relax%
                \setlength{\parsep}{\parskip}\relax%
                \setlength{\itemsep}{#3}\relax%
                \setlength{\listparindent}{\z@}\relax%
                \setlength{\itemindent}{\z@}\relax%
        }
    \expandafter\gdef\csname end#1\endcsname{%
        \end{#2}#5
        \setlength{\gesamtlinkerRand}{\alterlinkerRand}
        \setlength{\gesamtlinkerRand}{\alterrechterRand}
    }
}

\def\ritempunkt{{\Large \textbullet}} 
\setdefaultitem{\ritempunkt}{\ritempunkt}{\ritempunkt}{\ritempunkt}
\punkteumgebung@genbefehl{itemise}{compactitem}{\parskip}{}{}
\punkteumgebung@genbefehl{kompaktitem}{compactitem}{\z@}{}{}
\punkteumgebung@genbefehl{kompaktenum}{compactenum}{\z@}{}{}

\def\enumerate{%
    \@ifnextchar\bgroup{%
        \enumerate@legacyarg%
    }{%
        \@ifnextchar[{\enumerate@args}{\enumerate@noargs}
    }%
}
    \def\enumerate@spacing{
        \setlength{\parskip}{0.5\baselineskip}\relax%
        \setlength{\topsep}{\z@}\relax%
        \setlength{\partopsep}{\z@}\relax%
        \setlength{\parsep}{\parskip}\relax%
        \setlength{\itemsep}{\z@}\relax%
        \setlength{\listparindent}{\z@}\relax%
        \setlength{\itemindent}{\z@}\relax%
    }
    \def\enumerate@noargs{
        \begin{oldenumerate}
        \enumerate@spacing
    }
    \def\enumerate@args[#1]{
        \begin{oldenumerate}[#1]
        \enumerate@spacing
    }
    \def\enumerate@legacyarg#1{
        \begin{oldenumerate}[label=#1]
        \enumerate@spacing
    }
    \def\endenumerate{%
        \end{oldenumerate}
    }

\def\shorteqnarray{%
    \bgroup
    \setlength{\abovedisplayshortskip}{\aboveequation}%
    \setlength{\belowdisplayshortskip}{\belowequation}%
    \setlength{\abovedisplayskip}{\aboveequation - \baselineskip}%
    \setlength{\belowdisplayskip}{\belowequation}%
    \begin{eqnarray*}%
}
\def\endshorteqnarray{%
    \end{eqnarray*}%
    \egroup
}
\def\longeqnarray{%
    \bgroup%
    \allowdisplaybreaks%
    \setlength{\abovedisplayshortskip}{\aboveequation}%
    \setlength{\belowdisplayshortskip}{\belowequation}%
    \setlength{\abovedisplayskip}{\aboveequation - \baselineskip}%
    \setlength{\belowdisplayskip}{\belowequation}%
    \begin{eqnarray*}
}
\def\endlongeqnarray{%
    \end{eqnarray*}%
    \egroup%
}
\def\displayarray[#1]#2{
    \bgroup
    \everymath={\displaystyle}
    \begin{array}[#1]{#2}
}
\def\enddisplayarray{
    \end{array}
    \egroup
}

\def\matrix#1{\left(\begin{array}[mc]{#1}}
    \def\endmatrix{\end{array}\right)}
\def\smatrix{\left(\begin{smallmatrix}}
    \def\endsmatrix{\end{smallmatrix}\right)}

\def\multiargrekursiverbefehl#1#2#3#4#5#6#7#8{%
    \expandafter\gdef\csname#1\endcsname #2##1#4{\csname #1@anfang\endcsname##1#3\egroup}
    \expandafter\def\csname #1@anfang\endcsname##1#3{#5##1\@ifnextchar\egroup{\csname #1@ende\endcsname}{#7\csname #1@mitte\endcsname}}
    \expandafter\def\csname #1@mitte\endcsname##1#3{#6##1\@ifnextchar\egroup{\csname #1@ende\endcsname}{#7\csname #1@mitte\endcsname}}
    \expandafter\def\csname #1@ende\endcsname##1{#8}
}
\multiargrekursiverbefehl{svektor}{[}{;}{]}{\begin{smatrix}}{}{\\}{\\\end{smatrix}}
\multiargrekursiverbefehl{vektor}{[}{;}{]}{\begin{matrix}{c}}{}{\\}{\\\end{matrix}}
\multiargrekursiverbefehl{vektorzeile}{}{,}{;}{}{&}{}{}
\multiargrekursiverbefehl{matlabmatrix}{[}{;}{]}{\begin{smatrix}\vektorzeile}{\vektorzeile}{;\\}{;\end{smatrix}}

\def\underbracenodisplay#1{%
    \mathop{\vtop{\m@th\ialign{##\crcr
    $\hfil\displaystyle{#1}\hfil$\crcr
    \noalign{\kern3\p@\nointerlineskip}%
    \upbracefill\crcr\noalign{\kern3\p@}}}}\limits%
}

\def\changemargins{\@ifnextchar[{\indents@}{\indents@[\z@]}}
\def\indents@[#1]{\@ifnextchar[{\indents@@[#1]}{\indents@@[#1][\z@]}}
\def\indents@@[#1][#2]{%
    \begin{list}{}{\relax
        \setlength{\topsep}{\z@}\relax
        \setlength{\partopsep}{\z@}\relax
        \setlength{\parsep}{\parskip}\relax
        \setlength{\listparindent}{\z@}\relax
        \setlength{\itemindent}{\z@}\relax
        \setlength{\leftmargin}{#1}\relax
        \setlength{\rightmargin}{#2}\relax
        \let\alterlinkerRand\gesamtlinkerRand
        \let\alterrechterRand\gesamtrechterRand
        \addtolength{\gesamtlinkerRand}{#1}
        \addtolength{\gesamtrechterRand}{#2}
    }\relax
        \item[]\relax
}
    \def\endchangemargins{%
        \setlength{\gesamtlinkerRand}{\alterlinkerRand}
        \setlength{\gesamtlinkerRand}{\alterrechterRand}
        \end{list}%
    }

\def\indentonce{\begin{changemargins}[\rtab][\rtab]}
    \def\endindentonce{\end{changemargins}}

\def\restoremargins{\begin{changemargins}[-\gesamtlinkerRand][-\gesamtrechterRand]}
    \def\endrestoremargins{\end{changemargins}}

\def\programmiercode{
    \modulolinenumbers[1]
    \begin{changemargins}[\rtab][\rtab]%
    \begin{linenumbers}%
        \fontfamily{cmtt}\fontseries{m}\fontshape{u}\selectfont%
        \setlength{\parskip}{1\baselineskip}%
        \setlength{\parindent}{0pt}%
}
    \def\endprogrammiercode{
        \end{linenumbers}
        \end{changemargins}
    }

\def\schattiertebox@genbefehl#1#2#3{
    \expandafter\gdef\csname #1\endcsname{%
        \@ifnextchar[{\csname #1@args\endcsname}{\csname #1@args\endcsname[#3]}
    }
        \expandafter\def\csname #1@args\endcsname[##1]{%
            \@ifnextchar[{\csname #1@args@l\endcsname[##1]}{\csname #1@args@n\endcsname[##1]}
        }
        \expandafter\def\csname #1@args@l\endcsname[##1][##2]{%
            \@ifnextchar[{\csname #1@args@l@r\endcsname[##1][##2]}{\csname #1@args@l@n\endcsname[##1][##2]}
        }
        \expandafter\def\csname #1@args@n\endcsname[##1]{%
            \let\boolinmdframed\boolwahr
            \begin{mdframed}[#2leftmargin=0,rightmargin=0,outermargin=0,innermargin=0,##1]
        }
        \expandafter\def\csname #1@args@l@n\endcsname[##1][##2]{%
            \let\boolinmdframed\boolwahr
            \begin{mdframed}[#2leftmargin=##2/2,rightmargin=##2/2,outermargin=##2/2,innermargin=##2/2,##1]
        }
        \expandafter\def\csname #1@args@l@r\endcsname[##1][##2][##3]{%
            \let\boolinmdframed\boolwahr
            \begin{mdframed}[#2leftmargin=##2,rightmargin=##3,outermargin=##2,innermargin=##3,##1]
        }
    \expandafter\gdef\csname end#1\endcsname{%
        \end{mdframed}
        \let\boolinmdframed\boolfalsch
    }
}
    \schattiertebox@genbefehl{schattiertebox}{
        splittopskip=0,%
        splitbottomskip=0,%
        frametitleaboveskip=0,%
        frametitlebelowskip=0,%
        skipabove=1\baselineskip,%
        skipbelow=1\baselineskip,%
        linewidth=2pt,%
        linecolor=black,%
        roundcorner=4pt,%
    }{
        backgroundcolor=leer,%
        nobreak=true,%
    }

    \schattiertebox@genbefehl{schattierteboxdunn}{
        splittopskip=0,%
        splitbottomskip=0,%
        frametitleaboveskip=0,%
        frametitlebelowskip=0,%
        skipabove=1\baselineskip,%
        skipbelow=1\baselineskip,%
        linewidth=1pt,%
        linecolor=black,%
        roundcorner=2pt,%
    }{
        backgroundcolor=leer,%
        nobreak=true,%
    }

    \schattiertebox@genbefehl{highlightbox}{
        splittopskip=0,%
        splitbottomskip=0,%
        frametitleaboveskip=0,%
        frametitlebelowskip=0,%
        skipabove=1\baselineskip,%
        skipbelow=1\baselineskip,%
        linewidth=0pt,%
        linecolor=black,%
        roundcorner=2pt,%
    }{
        backgroundcolor=background_light_grey,%
        nobreak=true,%
    }

    \schattiertebox@genbefehl{highlightboxWithBreaks}{
        splittopskip=0,%
        splitbottomskip=0,%
        frametitleaboveskip=0,%
        frametitlebelowskip=0,%
        skipabove=1\baselineskip,%
        skipbelow=1\baselineskip,%
        linewidth=0pt,%
        linecolor=black,%
        roundcorner=2pt,%
    }{
        backgroundcolor=background_light_grey,%
    }

\def\tikzsetzepfeil#1{%
    \begin{tikzpicture}[remember picture,overlay,>=latex]%
        \draw #1;%
    \end{tikzpicture}%
}

\def\tikzsetzekreise[#1]#2#3{%
    \tikzsetzepfeil{%
    [rounded corners,#1]%
        ([shift={(-\tabcolsep,0.75\baselineskip)}]#2)%
        rectangle%
        ([shift={(\tabcolsep,-0.5\baselineskip)}]#3)
    }%
}

\tikzset{
    >=stealth,
    auto,
    node distance=1cm,
    thick,
    main node/.style={
        circle,draw,font=\sffamily\Large\bfseries,minimum size=0pt
    },
    state/.style={minimum size=0pt}
    loop above right/.style={loop,out=30,in=60,distance=0.5cm},
    loop above left/.style={above left,out=150,in=120,loop},
    loop below right/.style={below right,out=330,in=300,loop},
    loop below left/.style={below left,out=240,in=210,loop},
    itria/.style={
        draw,dashed,shape border uses incircle,
        isosceles triangle,shape border rotate=90,yshift=-1.45cm
    },
    rtria/.style={
        draw,dashed,shape border uses incircle,
        isosceles triangle,isosceles triangle apex angle=90,
        shape border rotate=-45,yshift=0.2cm,xshift=0.5cm
    },
    ritria/.style={
        draw,dashed,shape border uses incircle,
        isosceles triangle,isosceles triangle apex angle=110,
        shape border rotate=-55,yshift=0.1cm
    },
    litria/.style={
        draw,dashed,shape border uses incircle,
        isosceles triangle,isosceles triangle apex angle=110,
        shape border rotate=235,yshift=0.1cm
    }
}

\renewenvironment{cases}[0]{\left\{\begin{array}[c]{0lcl}}{\end{array}\right.}

\renewenvironment{displaycases}[0]{%
    \left\{\relax%
    \bgroup\relax%
    \everymath={\displaystyle}\relax%
    \begin{array}[c]{0lcl}\relax%
}{%
    \end{array}\relax%
    \egroup\relax%
    \right.\relax%
}

\providecommand{\usesinglequotes}{}
\renewcommand{\usesinglequotes}[1]{`#1'}

\providecommand{\zeroone}{}
\renewcommand{\zeroone}[0]{\textup{0\=/1}\relax\ifmmode\else\@\xspace\fi}
\providecommand{\onetoone}{}
\renewcommand{\onetoone}[0]{\ensuremath{1\!\!:\!\!1}\relax\ifmmode\else\@\xspace\fi}
\providecommand{\First}{}
\renewcommand{\First}[0]{\text{I\textsuperscript{st}}\relax\ifmmode\else\@\xspace\fi}
\providecommand{\Second}{}
\renewcommand{\Second}[0]{\text{II\textsuperscript{nd}}\relax\ifmmode\else\@\xspace\fi}
\providecommand{\Third}{}
\renewcommand{\Third}[0]{\text{III\textsuperscript{rd}}\relax\ifmmode\else\@\xspace\fi}

\providecommand{\TextCStarAlg}{}
\renewcommand{\TextCStarAlg}[0]{\text{C\textsuperscript{\ensuremath{\ast}}\=/algebra}\relax\ifmmode\else\@\xspace\fi}
\providecommand{\TextCStarSubAlg}{}
\renewcommand{\TextCStarSubAlg}[0]{\text{C\textsuperscript{\ensuremath{\ast}}\=/subalgebra}\relax\ifmmode\else\@\xspace\fi}
\providecommand{\TextCStarAlgs}{}
\renewcommand{\TextCStarAlgs}[0]{\text{C\textsuperscript{\ensuremath{\ast}}\=/algebras}\relax\ifmmode\else\@\xspace\fi}
\providecommand{\TextCStarSubAlgs}{}
\renewcommand{\TextCStarSubAlgs}[0]{\text{C\textsuperscript{\ensuremath{\ast}}\=/subalgebras}\relax\ifmmode\else\@\xspace\fi}

\providecommand{\envPreMathsLong}{}
\renewcommand{\envPreMathsLong}[0]{%
    \bgroup\relax%
    \let\old@arraystretch\arraystretch\relax%
    \renewcommand\arraystretch{1.2}\relax\relax%
}
\providecommand{\envPostMathsLong}{}
\renewcommand{\envPostMathsLong}[0]{%
    \renewcommand\arraystretch{\old@arraystretch}\relax%
    \egroup\relax%
}
\providecommand{\id}{}
\renewcommand{\id}[0]{\mathrm{\textit{id}}}
\providecommand{\mod}{}
\renewcommand{\mod}[0]{\:\mathbin{\mathrm{mod}}\:}

\providecommand{\complex}{}
\renewcommand{\complex}[0]{\mathbb{C}}
\providecommand{\Torus}{}
\renewcommand{\Torus}[0]{\mathbb{T}}

\providecommand{\reals}{}
\renewcommand{\reals}[0]{\mathbb{R}}
\providecommand{\realsNonNeg}{}
\renewcommand{\realsNonNeg}[0]{\reals_{\geq 0}}

\providecommand{\rationals}{}
\renewcommand{\rationals}[0]{\mathbb{Q}}
\providecommand{\rationalsNonNeg}{}
\renewcommand{\rationalsNonNeg}[0]{\rationals_{\geq 0}}

\providecommand{\integers}{}
\renewcommand{\integers}[0]{\mathbb{Z}}
\providecommand{\naturals}{}
\renewcommand{\naturals}[0]{\mathbb{N}}
\providecommand{\naturalsZero}{}
\renewcommand{\naturalsZero}[0]{\mathbb{N}_{0}}

\providecommand{\HilbertRaum}{}
\renewcommand{\HilbertRaum}[0]{\mathcal{H}}
\providecommand{\BanachRaum}{}
\renewcommand{\BanachRaum}[0]{\mathcal{E}}

\providecommand{\RaumX}{}
\renewcommand{\RaumX}[0]{X}
\providecommand{\RaumY}{}
\renewcommand{\RaumY}[0]{Y}

\providecommand{\Proj}{}
\renewcommand{\Proj}[0]{\mathrm{Proj}}

\providecommand{\oBall}{}
\renewcommand{\oBall}[2]{\cal{B}_{#2}(#1)}

\providecommand{\topSOT}{}
\renewcommand{\topSOT}[0]{\text{\upshape\scshape sot}}
\providecommand{\topSOTstar}{}
\renewcommand{\topSOTstar}[0]{\text{\upshape\scshape sot}^{\star}}
\providecommand{\topWOT}{}
\renewcommand{\topWOT}[0]{\text{\upshape\scshape wot}}

\providecommand{\topPW}{}
\renewcommand{\topPW}[0]{\text{\upshape\scshape pw}}

\providecommand{\FnRm}{}
\renewcommand{\FnRm}[2]{\mathrm{Fct}(#1,#2)}
\providecommand{\KmpRm}{}
\renewcommand{\KmpRm}[1]{\mathcal{K}(#1)}
\providecommand{\card}{}
\renewcommand{\card}[1]{\lvert #1 \rvert}

\newcommand{\Pot}[0]{\mathop{\powerset}}

\providecommand{\einser}{}
\renewcommand{\einser}[0]{1\!\!1}

\providecommand{\ConstOne}{}
\renewcommand{\ConstOne}[0]{\blackboardfont{1}}

\providecommand{\Gph}{}
\renewcommand{\Gph}[1]{\mathcal{G}\mathrm{ph}(#1)}
\providecommand{\floor}{}
\renewcommand{\floor}[1]{{\lfloor #1 \rfloor}}

\providecommand{\iunit}{}
\renewcommand{\iunit}[0]{\imath}
\providecommand{\abs}{}
\renewcommand{\abs}[1]{\lvert #1 \rvert}
\providecommand{\absLong}{}
\renewcommand{\absLong}[1]{\Big| #1 \Big|}
\providecommand{\revProd}{}
\renewcommand{\revProd}[0]{\text{\upshape rev-}\prod}

\providecommand{\linspann}{}
\renewcommand{\linspann}[0]{\textup{lin}}

\providecommand{\onematrix}{}
\renewcommand{\onematrix}[0]{\text{\upshape\bfseries I}}
\providecommand{\zeromatrix}{}
\renewcommand{\zeromatrix}[0]{\mathbf{0}}

\providecommand{\zerovector}{}
\renewcommand{\zerovector}[0]{\mathbf{0}}

\providecommand{\brkt}{}
\renewcommand{\brkt}[2]{\langle{}#1,\:#2{}\rangle}

\providecommand{\brktLong}{}
\renewcommand{\brktLong}[2]{\Big\langle{}#1,\:#2{}\Big\rangle}

\providecommand{\norm}{}
\renewcommand{\norm}[1]{\lVert #1 \rVert}
\providecommand{\normLong}{}
\renewcommand{\normLong}[1]{\Big\| #1 \Big\|}
\providecommand{\normLarge}{}
\renewcommand{\normLarge}[1]{\left\| #1 \right\|}

\providecommand{\opDomain}{}
\renewcommand{\opDomain}[1]{\mathcal{D}(#1)}

\providecommand{\opResolvent}{}
\renewcommand{\opResolvent}[2]{R(#2, #1)}
\providecommand{\BoundedOpsSymbol}{}
\renewcommand{\BoundedOpsSymbol}[0]{\mathfrak{L}}
\providecommand{\BaseVector}{}
\renewcommand{\BaseVector}[1]{\mathbf{e}_{#1}}
\providecommand{\ElementaryMatrix}{}
\renewcommand{\ElementaryMatrix}[2]{\mathbf{E}_{#1,#2}}
\providecommand{\Cnought}{}
\renewcommand{\Cnought}[0]{\mathcal{C}_{0}}

\providecommand{\Repr}{}
\renewcommand{\Repr}[2]{\mathrm{Repr}(#1 \!:\! #2)}
\providecommand{\OpSpaceC}{}
\renewcommand{\OpSpaceC}[1]{\textup{\textbf{C}}(#1)}
\providecommand{\OpSpaceU}{}
\renewcommand{\OpSpaceU}[1]{\textup{\textbf{U}}(#1)}

\providecommand{\ShiftRight}{}
\renewcommand{\ShiftRight}[0]{\mathrm{Sh}_{\to}}
\providecommand{\HalmosOp}{}
\renewcommand{\HalmosOp}[0]{\mathfrak{D}}
\providecommand{\NagyIso}{}
\renewcommand{\NagyIso}[0]{\mathcal{V}}
\providecommand{\NagyUtr}{}
\renewcommand{\NagyUtr}[0]{\mathcal{U}}

\providecommand{\dee}{}
\renewcommand{\dee}[0]{\mathop{\textup{d}}\!}

\providecommand{\freepr}{}
\renewcommand{\freepr}[0]{\oasterisk}
\providecommand{\freeprBig}{}
\renewcommand{\freeprBig}[0]{\bigoasterisk}
\providecommand{\FreeProduct}{}
\renewcommand{\FreeProduct}[2]{\prod_{#1}^{\freepr}#2}
\providecommand{\FreePower}{}
\renewcommand{\FreePower}[2]{#2^{\freepr #1}}

\providecommand{\DilatableMonoids}{}
\renewcommand{\DilatableMonoids}[0]{\mathbb{M}_{1}}
\providecommand{\DilatableGroupMonoids}{}
\renewcommand{\DilatableGroupMonoids}[0]{\mathbb{M}_{2}}

\def\Cts{\@ifnextchar_{\Cts@tief}{\Cts@tief_{}}}
    \def\Cts@tief_#1#2{\@ifnextchar\bgroup{\Cts@two_{#1}{#2}}{\Cts@one_{#1}{#2}}}
    \def\Cts@one_#1#2{C_{#1}\big(#2\big)}
    \def\Cts@two_#1#2#3{C_{#1}\big(#2,~#3\big)}

\def\BoundedOps#1{\@ifnextchar\bgroup{\BoundedOps@two{#1}}{\mathop{\BoundedOpsSymbol}(#1)}}
    \def\BoundedOps@two#1#2{\mathop{\BoundedOpsSymbol}(#1,#2)}
\def\BoundedOpsInv#1{\@ifnextchar\bgroup{\BoundedOps@two{#1}}{\mathop{\BoundedOpsSymbol}(#1)^{\times}}}
    \def\BoundedOpsInv@two#1#2{\mathop{\BoundedOpsSymbol}(#1,#2)^{\times}}

\def\restr#1{\vert_{#1}}
\def\without{\mathbin{\setminus}}

\def\da{\partial}
\def\eps{\varepsilon}
\let\altphi\phi
\let\altvarphi\varphi
    \def\phi{\altvarphi}
    \def\varphi{\altphi}
\def\quer#1{\overline{#1}}

\def\lim{\mathop{\ell\mathrm{im}}}
\def\supp{\mathop{\textup{supp}}}
\def\dim{\mathop{\textup{dim}}}

\def\ran{\mathop{\textup{ran}}}

\def\Re{\mathop{\mathfrak{R}\mathrm{e}}}

\def\tinytopWOT{\text{\scriptsize\upshape \scshape wot}}

\def\toplocWOT{\ensuremath{\mathpzc{k}}_{\text{\tiny\upshape \scshape wot}}}

\def\tinytoplocWOT{\scriptsize{\ensuremath{\mathpzc{k}}}\text{-{\upshape \scshape wot}}}
\def\tinytopSOT{\text{\scriptsize\upshape \scshape sot}}

\def\toplocSOT{\ensuremath{\mathpzc{k}}_{\text{\tiny\upshape \scshape sot}}}

\def\tinytoplocSOT{\scriptsize{\ensuremath{\mathpzc{k}}}\text{-{\upshape \scshape sot}}}

\def\toplocSOTstar{\ensuremath{\mathpzc{k}}_{\text{\tiny\upshape \scshape sot\ensuremath{\star}}}}

\def\tinytoplocSOTstar{\scriptsize{\ensuremath{\mathpzc{k}}}\text{-{\upshape \scshape sot\ensuremath{\star}}}}

\def\tinytopPW{\text{\scriptsize\upshape \scshape pw}}

\def\interval{\mathcal{J}}
\def\filterunit{\mathcal{F}_{[0,\:1]}}
\def\emb{\boldsymbol{\iota}}

\@ifundefined{setcitestyle}{%
}{%
    \setcitestyle{numeric-comp,open={[},close={]}}
}

\allowdisplaybreaks 
\renewcommand{\arraystretch}{1}
\def\firstparagraph{\noindent}
\def\continueparagraph{\noindent}

\generatenestedsecnumbering{arabic}{section}{subsection}
\generatenestedsecnumbering{arabic}{subsection}{subsubsection}
\def\theunitnamesection{\thesection}

\def\sectionname{}

\let\appendix@orig\appendix
\def\appendix{%
    \appendix@orig%
    \let\boolinappendix\boolwahr
    \addcontentsline{toc}{part}{\appendixname}%
    \addtocontents{toc}{\protect\setcounter{tocdepth}{0}}
    \def\sectionname{Appendix}%
    \def\theunitnamesection{\Alph{section}}%
}
\def\notappendix{%
    \let\boolinappendix\boolfalse
    \addtocontents{toc}{\protect\setcounter{tocdepth}{1 }}
    \def\sectionname{}%
    \def\theunitnamesection{\arabic{section}}%
}

\def\@settitle{%
    \bgroup
    \LARGE
    \scshape
    \@title
    \egroup
}

\def\@seccntformat#1{%
    \protect\textup{%
        \protect\@secnumfont
        \expandafter\protect\csname format#1\endcsname%
        \csname the#1\endcsname
        \expandafter\protect\csname format#1@pt\endcsname%
        \space
    }%
}

\def\formatsection@text{\centering\Large\scshape}
\def\formatsection@pt{\secnumberingseppt}
\def\section{\@startsection{section}{1}{\z@}{.7\linespacing\@plus\linespacing}{.5\linespacing}{\formatsection@text}}

\def\formatsubsection@text{\flushleft\bfseries\scshape}
\def\formatsubsection@pt{\subsecnumberingseppt}
\def\subsection{\@startsection{subsection}{2}{\z@}{\z@}{\z@\hspace{1em}}{\formatsubsection@text}}

\renewcommand{\paragraph}[1]{%
    {\itshape #1}\:%
}

\DefineFNsymbols*{customAlphabet}{abcdefghijklmnopqrstuvwxyz}
\setfnsymbol{customAlphabet}

\def\footnotemark[#1]{\text{\textsuperscript{\getrefnumber{#1}}}}

\def\footnote@custom@period{24}
\providecommand{\footnote@ctr@prebump}{}
\renewcommand{\footnote@ctr@prebump}[1]{%
    \ifnum\value{#1}<\footnote@custom@period%
    \else\relax
        \setcounter{#1}{0}%
    \fi%
}
\providecommand{\footnoteref}{}
\renewcommand{\footnoteref}[1]{\protected@xdef\@thefnmark{\ref{#1}}\@footnotemark}
\let\@old@footnotetext\footnotetext
\def\footnotetext[#1]#2{%
    \footnote@ctr@prebump{footnote}%
    \addtocounter{footnote}{1}%
    \@old@footnotetext[\value{footnote}]{\label{#1}#2}%
}

\let\@old@footnote\footnote
\renewcommand{\footnote}[1]{%
    \footnote@ctr@prebump{footnote}%
    \@old@footnote{#1}%
}

\def\kopfzeiledefault{
    \lhead[]{}
    \lhead[]{}
    \chead[]{}
    \rhead[]{}
    \lfoot[]{}
    \cfoot{\footnotesize\thepage}
    \rfoot[]{}
}

\def\aktuellesfont{\csname lmodern\endcsname}
\def\documentfont{%
    \gdef\aktuellesfont{\csname lmodern\endcsname}%
    \fontfamily{lmr}\fontseries{m}\selectfont%
    \renewcommand{\sfdefault}{phv}%
    \renewcommand{\ttdefault}{pcr}%
    \renewcommand{\rmdefault}{cmr}
    \renewcommand{\bfdefault}{bx}%
    \renewcommand{\itdefault}{it}%
    \renewcommand{\sldefault}{sl}%
    \renewcommand{\scdefault}{sc}%
    \renewcommand{\updefault}{n}%
}

\allowdisplaybreaks

\def\startdocumentlayoutoptions{
    \selectlanguage{british}
    \setlength{\parskip}{0.25\baselineskip}
    \setlength{\parindent}{2em}
    \kopfzeiledefault
    \documentfont
    \normalsize
}

\providecommand{\highlightTerm}{}
\renewcommand{\highlightTerm}[1]{\emph{#1}}

\renewenvironment{highlightForReview}[0]{\color{blue}}{}

\def\addresseshere{%
  \bgroup
  \setlength{\parindent}{0pt}
  \enddoc@text
  \egroup
  \let\enddoc@text\relax
}

\makeatother

\begin{document}
\startdocumentlayoutoptions

\thispagestyle{plain}



\def\abstractname{Abstract}
\begin{abstract}
    Commuting families of contractions
    or contractive $\Cnought$\=/semigroups
    on Hilbert spaces
    often fail to admit power dilations \resp simultaneous unitary dilations
    which are themselves commutative
    (see
        \cite{%
            Parrott1970counterExamplesDilation,%
            Dahya2023dilation,%
            Dahya2024interpolation%
        }%
    ).
    In the \emph{non-commutative} setting,
    Sz.-Nagy \cite{Nagy1960Article}
    and
    Bo\.{z}ejko \cite{Bozejko1989Article}
    provided means to dilate arbitrary families of contractions.
    The present work extends these discrete-time results
    to families
        $\{T_{i}\}_{i \in I}$
    of contractive $\Cnought$\=/semigroups.
    We refer to these dilations as continuous-time \emph{free unitary dilations}
    and present three distinct approaches to obtain them:
        1) An explicit derivation applicable to semigroups that arise as interpolations;
        2) A full proof with an explicit construction,
            via the theory of co-generators
            \akin S\l{}oci\'{n}ski
            \cite{%
                Slocinski1974,%
                Slocinski1982%
            };
        and
        3) A second full proof based on the abstract structure of semigroups,
        which admits a natural reformulation
        to semigroups defined over topological free products of $\realsNonNeg$
        and leads to various residuality results.
    In 2) a \Second free dilation theorem
    for topologised index sets is developed
    via a reformulation of the Trotter--Kato theorem for co-generators.
    As an application of this we demonstrate how evolution families
    can be reduced to continuously monitored processes subject to temporal change,
    \akin the quantum Zeno effect
        \cite{%
            ExnerIchinose2005Article,%
            ExnerIchinose2006Inproceedings,%
            FacchiLigabo2010Article,%
            HermanShaydulinSun2023Article,%
            Kitano1997Article%
        }.
\end{abstract}



\title[Free dilations]{
    \hraum Free dilations of families of $\mathcal{C}_{0}$\=/semigroups\hraum%
    \large
    \newline%
    \hraum and applications to evolution families\hraum%
}

\author{Raj Dahya}
\address{Fakult\"at f\"ur Mathematik und Informatik\newline
Universit\"at Leipzig, Augustusplatz 10, D-04109 Leipzig, Germany}
\email{raj\,[\!\![dot]\!\!]\,dahya\:[\!\![at]\!\!]\:web\,[\!\![dot]\!\!]\,de}

\def\subjclassname{Mathematics Subject Classification (2020)} 
\subjclass{47A20, 47D06}

\maketitle



\setcounternach{section}{1}



\section[Background and motivation]{Background and motivation}
\label{sec:introduction:sig:article-free-raj-dahya}


\firstparagraph
The existence of dilations
for commuting families of contractions
(discrete-time setting)
and contractive $\Cnought$\=/semigroups
(continuous-time setting)
on Hilbert spaces
has been largely determined.
Let $d\in\naturals$ and consider $d$-tuples
    $\{S_{i}\}_{i=1}^{d}$
of commuting contractions
\resp
families
    $\{T_{i}\}_{i=1}^{d}$
of commuting contractive $\Cnought$\=/semigroups
on some Hilbert space $\HilbertRaum$.
By Sz.-Nagy, And\^{o}, and S\l{}oci\'{n}ski
(see
    \cite[Theorems~I.4.2 and I.8.1]{Nagy1970},
    \cite{Ando1963pairContractions}
    \cite{Slocinski1974},
    and
    \cite[Theorem~2]{Slocinski1982}%
),
we know for $d \in \{1,2\}$
that power dilations \resp simultaneous unitary dilations
always exist in the discrete- \resp continuous-time setting.
For $d \geq 3$,
Parrott, Varopoulos, and Kaijser
(see
    \cite[\S{}3]{Parrott1970counterExamplesDilation}
    and
    \cite[Theorem~1]{Varopoulos1974counterexamples}%
)
showed that power dilations of $d$-tuples of contractions
do not always exist.
Continuous-time analogues of these results
were recently presented in
\cite[Theorem~1.5 and Corollary~1.7~b)]{Dahya2024interpolation},
in which generic commuting families
of contractive $\Cnought$\=/semigroups
were demonstrated to possess no simultaneous unitary dilations.
This failure can be amended, if we discard the requirement
that the dilations themselves be commuting families of unitaries
(\resp unitary representations).
We thus turn our attention
to perhaps the simplest non-commutative setting,
\viz one in which
no assumptions are made in terms of commutation relations
or algebraic simplifications.

In the discrete-time setting,
Sz.-Nagy and later Bo\.{z}ejko demonstrated
that non-commutative dilations are always possible.

\begin{thm}[Sz.-Nagy, 1960/1]
\makelabel{thm:free-dilation-discrete-time:sig:article-free-raj-dahya}
    Let $I$ be a non-empty index set
    and $\{S_{i}\}_{i \in I}$ be a family of (not necessarily commuting)
    contractions on a Hilbert space $\HilbertRaum$.
    Then there exists a family of unitaries
        $\{V_{i}\}_{i \in I}$
    on a Hilbert space $\HilbertRaum^{\prime}$
    and an isometry ${r : \HilbertRaum \to \HilbertRaum^{\prime}}$
    satisfying

    \begin{restoremargins}
    \begin{equation}
    \label{eq:free-dilation:discrete-time:sig:article-free-raj-dahya}
        \prod_{k=1}^{N}
            S_{i_{k}}^{n_{k}}
        = r^{\ast}
            \:\Big(
                \prod_{k=1}^{N}
                    V_{i_{k}}^{n_{k}}
            \Big)
            \:r
    \end{equation}
    \end{restoremargins}

    \continueparagraph
    for all $N \in \naturals$,
    all sequences $(i_{k})_{k=1}^{N} \subseteq I$,
    and all
        $(n_{k})_{k=1}^{N} \subseteq \naturalsZero$.
\end{thm}

Since we shall make use of the constructions that occur in this result,
a proof of this will be sketched
in \S{}\ref{sec:result-concrete:discrete-dilation:sig:article-free-raj-dahya}
(see also \cite[\S{}I.5,~(5.5)]{Nagy1970}).
Bo\.{z}ejko strengthened this result as follows.

\begin{defn}
\makelabel{defn:bubble-swap-free:sig:article-free-raj-dahya}
    For $N \in \naturals$ say that a sequence
        $\{a_{k}\}_{k=1}^{N}$
    of arbitrary elements
    is \highlightTerm{bubble-swap free}
    if $a_{k'} \neq a_{k}$
    for each $k,k' \in \{1,2,\ldots,N\}$ with $\abs{k' - k} = 1$.
\end{defn}

\begin{thm}[Bo\.{z}ejko, 1989]
\makelabel{thm:free-regular-dilation-discrete-time:sig:article-free-raj-dahya}
    Let $I$ be a non-empty index set
    and $\{S_{i}\}_{i \in I}$ be a family of (not necessarily commuting)
    contractions on a Hilbert space $\HilbertRaum$.
    Then there exists a family of unitaries
    $\{V_{i}\}_{i \in I}$
    on a Hilbert space $\HilbertRaum^{\prime}$
    and an isometry ${r : \HilbertRaum \to \HilbertRaum^{\prime}}$
    satisfying

    \begin{restoremargins}
    \begin{equation}
    \label{eq:free-regular-dilation:discrete-time:sig:article-free-raj-dahya}
        \prod_{k=1}^{N}
            \begin{cases}
                S_{i_{k}}^{n_{k}}
                    &: &n_{k} > 0\\
                (S_{i_{k}}^{\ast})^{-n_{k}}
                    &: &n_{k} < 0\\
            \end{cases}
        = r^{\ast}
            \:\Big(
            \prod_{k=1}^{N}V_{i_{k}}^{n_{k}}
            \Big)
            \:r
    \end{equation}
    \end{restoremargins}

    \continueparagraph
    for all $N \in \naturals$,
    all bubble-swap free $(i_{k})_{k=1}^{N} \subseteq I$,
    and all
        $(n_{k})_{k=1}^{N} \subseteq \integers \without \{0\}$.
\end{thm}

For a proof, see \cite[Theorem~8.1]{Bozejko1989Article}.



We shall refer to the dilation of Sz.-Nagy
as a (discrete-time) \highlightTerm{free unitary dilation},
and to that of Bo\.{z}ejko
as a (discrete-time) \highlightTerm{free regular unitary dilation}.
Such dilations can be instrumentalised to obtain polynomial bounds
(see \exempli %
    \cite[Lemma~2.8]{SampatShalit2025ArticleNcOpBalls}%
).
We shall investigate the existence of continuous-time counterparts.
Our first main result is as follows:

\begin{highlightboxWithBreaks}
\begin{thm}[\First Free dilation theorem]
\makelabel{thm:free-dilation:1st:sig:article-free-raj-dahya}
    Let $I$ be a non-empty index set
    and $\{T_{i}\}_{i \in I}$ be a family of (not necessarily commuting)
    contractive $\Cnought$\=/semigroups on a Hilbert space $\HilbertRaum$.
    There exists a Hilbert space $\HilbertRaum^{\prime}$,
    a family of $\topSOT$\=/continuous unitary representations
        $\{U_{i}\}_{i \in I} \subseteq \Repr{\reals}{\HilbertRaum^{\prime}}$,
    and an isometry ${r : \HilbertRaum \to \HilbertRaum^{\prime}}$
    satisfying

    \begin{restoremargins}
    \begin{equation}
    \label{eq:free-dilation:cts-time:sig:article-free-raj-dahya}
        \prod_{k=1}^{N}
            T_{i_{k}}(t_{k})
        = r^{\ast}
            \:\Big(
                \prod_{k=1}^{N}
                    U_{i_{k}}(t_{k})
            \Big)
            \:r
    \end{equation}
    \end{restoremargins}

    \continueparagraph
    for all $N \in \naturals$,
    all sequences $(i_{k})_{k=1}^{N} \subseteq I$,
    and all
        $(t_{k})_{k=1}^{N} \subseteq \realsNonNeg$.
\end{thm}
\end{highlightboxWithBreaks}

We shall refer to
$(\HilbertRaum^{\prime},r,\{U_{i}\}_{i \in I})$
in \Cref{thm:free-dilation:1st:sig:article-free-raj-dahya}
as a (continuous-time) \highlightTerm{free unitary dilation}.

\begin{rem}[Na\"{i}ve approach]
\makelabel{rem:naive-approach:sig:article-free-raj-dahya}
    Since the dilation is not required to preserve any algebraic relations,
    it is tempting to attempt to prove \Cref{thm:free-dilation:1st:sig:article-free-raj-dahya}
    by simply taking $1$\=/parameter unitary dilations
    and stitching these together.
    This approach would be as follows:
    For each $i \in I$ we let $(H_{i},r_{i},U_{i})$
    be a minimal unitary dilation of $(\HilbertRaum,T_{i})$
    (see \cite[Theorem~I.8.1]{Nagy1970}).
    Using direct sums, and unitary adjustments,
    we can assume that $H_{i} = H$ and $r_{i} = r$
    for all $i \in I$
    and some Hilbert space $H$
    and isometry $r \in \BoundedOps{\HilbertRaum}{H}$.
    Let $N \in \naturals$,
    $(i_{k})_{k=1}^{N} \subseteq I$,
    and $(t_{k})_{k=1}^{N} \subseteq \realsNonNeg$.
    Then

    \begin{shorteqnarray}
        \prod_{k=1}^{N}
            T_{i_{k}}(t_{k})
        = \prod_{k=1}^{N}
            (r^{\ast}\:U_{i_{k}}(t_{k})\:r)
        = r^{\ast}
            \:
            \Big(
                \prod_{k=1}^{N}
                    p\:U_{i_{k}}(t_{k})
            \Big)
            \:r
    \end{shorteqnarray}

    \continueparagraph
    where $p \coloneqq rr^{\ast}$,
    which is the projection in $\HilbertRaum^{\prime}$
    to the subspace $\ran(r) = r\HilbertRaum$.
    The issue is that it is unclear whether
    the $p$'s in the above can be eliminated
    to obtain \eqcref{eq:free-dilation:cts-time:sig:article-free-raj-dahya}.
    This requires further special knowledge of the structure of the $1$\=/parameter dilations.
\end{rem}

\begin{rem}[Non-topological free dilations]
    In \cite[Theorem~4.1]{Popescu1999Article},
    Popescu worked with the notion of
    \highlightTerm{positive-definite} kernels
    to obtain similar dilations
    for operator semigroups
    satisfying certain restrictive conditions,
    and defined on \highlightTerm{commensurable} submonoids
    $P \subseteq \realsNonNeg$.%
    \footnote{%
        A set $A \subseteq \realsNonNeg$
        is called commensurable if for each finite set $F \subseteq A \without \{0\}$,
        there exists $t \in A$,
        such that $F \subseteq \naturals \cdot t$.
        \Exempli
        $\rationalsNonNeg$,
        $\{\tfrac{m}{2^{n}} \mid m,n\in\naturalsZero\}$
        are commensurable,
        but $\realsNonNeg$ is not.
    }
    However, these results make no claims about the \emph{continuity}
    of either the semigroups or their dilations,
    as the underlying space of time points is endowed with the discrete topology.
\end{rem}

It is also a straightforward exercise to
extend Bo\.{z}ejko's techniques in \cite[\S{}8]{Bozejko1989Article}
to families of $\Cnought$\=/semigroups.
This approach however fails to produce dilations that are continuous.

\begin{rem}[No free regular dilations in continuous-time]
    Given a non-empty index set $I$
    and family $\{T_{i}\}_{i \in I}$ (not necessarily commuting)
    of contractive $\Cnought$\=/semigroups on a Hilbert space $\HilbertRaum$,
    a natural definition for a
    (continuous-time) \highlightTerm{free regular unitary dilation}
    would consist of
    a family of $\topSOT$\=/continuous unitary representations
        $\{U_{i}\}_{i \in I} \subseteq \Repr{\reals}{\HilbertRaum^{\prime}}$,
    and an isometry
        ${r : \HilbertRaum \to \HilbertRaum^{\prime}}$
    satisfying

        \begin{restoremargins}
        \begin{equation}
        \label{eq:free-regular-dilation:cts-time:sig:article-free-raj-dahya}
            \prod_{k=1}^{N}
                \begin{cases}
                    T_{i_{k}}(t_{k})
                        &: &t_{k} > 0\\
                    T_{i_{k}}(-t_{k})^{\ast}
                        &: &t_{k} < 0\\
                \end{cases}
            = r^{\ast}
                \:\Big(
                    \prod_{k=1}^{N}
                        U_{i_{k}}(t_{k})
                \Big)
                \:r
        \end{equation}
        \end{restoremargins}

    \continueparagraph
    for all $N \in \naturals$,
    all bubble-swap free $(i_{k})_{k=1}^{N} \subseteq I$,
    and all
        $(t_{k})_{k=1}^{N} \subseteq \reals \without \{0\}$.
    Suppose that such a dilation existed and that $\card{I} \geq 2$.
    Consider an arbitrary index $i \in I$ and $t > 0$.
    Taking any $j \in I \setminus \{i\}$
    the free regular dilation would yield

        \begin{shorteqnarray}
            T_{i}(t)^{\ast}T_{j}(\varepsilon)T_{i}(t)
            &\eqcrefoverset{eq:free-regular-dilation:cts-time:sig:article-free-raj-dahya}{=}
                &r^{\ast}
                U_{i}(-t)U_{j}(\varepsilon)U_{i}(t)
                \:r,
                ~\text{and}\\
            T_{i}(t)T_{j}(\varepsilon)T_{i}(t)^{\ast}
            &\eqcrefoverset{eq:free-regular-dilation:cts-time:sig:article-free-raj-dahya}{=}
                &r^{\ast}
                U_{i}(t)U_{j}(\varepsilon)U_{i}(-t)
                \:r
        \end{shorteqnarray}

    \continueparagraph
    for all $\varepsilon > 0$.
    Since by continuity
    ${U_{j}(\varepsilon) \overset{\tinytopSOT}{\longrightarrow} I}$
    and
    ${T_{j}(\varepsilon) \overset{\tinytopSOT}{\longrightarrow} I}$
    for ${(0,\:\infty)\ni\varepsilon \longrightarrow 0}$,
    taking limits of either side of the above expressions yields
    $
        T_{i}(t)^{\ast}T_{i}(t)
        = r^{\ast}\:U_{i}(-t)U_{i}(t)\:r
        = I
    $
    and similarly
    $
        T_{i}(t)T_{i}(t)^{\ast}
        = r^{\ast}\:U_{i}(t)U_{i}(-t)\:r
        = I
    $.
    Since this must hold for all $i \in I$ and all $t > 0$,
    we conclude that the existence of a free regular unitary dilation
    is only possible if either $\card{I} = 1$
    or the semigroups $\{T_{i}\}_{i \in I}$
    are all already unitary.
    That is, there are no non-trivial continuous-time free regular unitary dilations.
\end{rem}



We shall also derive a version of free dilations for topologised index sets.
To achieve this, we recall and make use of the following notions.
Let
    $\FnRm{\RaumX}{\RaumY}$
    denote the set of all functions
    from a set $\RaumX$ to a set $\RaumY$
and let
    $\KmpRm{\RaumX}$
    denote the compact subsets of a topological space $\RaumX$.

\begin{defn}
\makelabel{defn:k-sot-top:sig:article-free-raj-dahya}
    Let $\RaumX$ be a topological space and $\BanachRaum$ a Banach space.
    The topology ($\toplocSOT$)
    of \highlightTerm{uniform strong convergence on compact subsets of $\RaumX$}
    on $\FnRm{\RaumX}{\BoundedOps{\BanachRaum}}$
    is defined by the convergence condition
    $
        f^{(\alpha)}
        \underset{\alpha}{\overset{\tinytoplocSOT}{\longrightarrow}}
        f
    $
    if and only if

    \begin{shorteqnarray}
        \forall{K \in \KmpRm{\RaumX}:~}
        \forall{\xi \in \BanachRaum:~}
            \sup_{x \in K}
                \norm{(f^{(\alpha)}(x) - f(x))\xi}
            \underset{\alpha}{\longrightarrow}
            0
    \end{shorteqnarray}

    \continueparagraph
    for nets
    $
        (f^{(\alpha)})_{\alpha \in \Lambda}
        \subseteq
        \FnRm{\RaumX}{\BoundedOps{\BanachRaum}}
    $
    and $f\in\FnRm{\RaumX}{\BoundedOps{\BanachRaum}}$.
\end{defn}

\begin{defn}
\makelabel{defn:k-sot-star-top:sig:article-free-raj-dahya}
    Let $\RaumX$ be a topological space and $\HilbertRaum$ a Hilbert space.
    We let $\toplocSOTstar$ denote the topology
    defined by the convergence condition
    ${
        f^{(\alpha)}
        \underset{\alpha}{\overset{\tinytoplocSOTstar}{\longrightarrow}}
        f
    }$
    if and only if
    ${
        f^{(\alpha)}
        \underset{\alpha}{\overset{\tinytoplocSOT}{\longrightarrow}}
        f
    }$
    and
    ${
        f^{(\alpha)}(\cdot)^{\ast}
        \underset{\alpha}{\overset{\tinytoplocSOT}{\longrightarrow}}
        f(\cdot)^{\ast}
    }$
    for nets
    $
        (f^{(\alpha)})_{\alpha \in \Lambda}
        \subseteq
        \FnRm{\RaumX}{\BoundedOps{\HilbertRaum}}
    $
    and $f\in\FnRm{\RaumX}{\BoundedOps{\HilbertRaum}}$.
\end{defn}

\begin{defn}
\makelabel{defn:k-wot-top:sig:article-free-raj-dahya}
    Let $\RaumX$ be a topological space
    and $\HilbertRaum$ a Hilbert space.
    The topology ($\toplocWOT$)
    of \highlightTerm{uniform weak convergence on compact subsets of $\RaumX$}
    on $\FnRm{\RaumX}{\BoundedOps{\HilbertRaum}}$
    is defined by the convergence condition
    $
        f^{(\alpha)}
        \underset{\alpha}{\overset{\tinytoplocWOT}{\longrightarrow}}
        f
    $
    if and only if

    \begin{shorteqnarray}
        \forall{K \in \KmpRm{\RaumX}:~}
        \forall{\xi, \eta \in \HilbertRaum:~}
            \sup_{x \in K}
                \abs{\brkt{(f^{(\alpha)}(x) - f(x))\xi}{\eta}}
            \underset{\alpha}{\longrightarrow} 0
    \end{shorteqnarray}

    \continueparagraph
    for nets
    $
        (f^{(\alpha)})_{\alpha \in \Lambda}
        \subseteq
        \FnRm{\RaumX}{\BoundedOps{\HilbertRaum}}
    $
    and $f\in\FnRm{\RaumX}{\BoundedOps{\HilbertRaum}}$.
    This topology is similarly defined with appropriate adjustments for Banach spaces.
\end{defn}

Consider a family
    $\{T_{\omega}\}_{\omega \in \Omega}$
of either $\Cnought$\=/semigroups
(defined over $\RaumX = \realsNonNeg$)
on a Hilbert or Banach space $E$,
or else continuous representations of $\RaumX = \reals$
on a Hilbert space.
In latter parts of this paper,
we shall focus on such families
which are
\highlightTerm{%
    $\toplocWOT$-/ $\toplocSOT$-/$\toplocSOTstar$\=/continuous
    in the index set
}
(or simply: \highlightTerm{in $\Omega$}),
\idest
for which the map
    ${
        \Omega
        \ni \omega
        \mapsto
        T_{\omega}
        \in
        \FnRm{\RaumX}{\BoundedOps{E}}
    }$
is
    $\toplocSOT$- \resp $\toplocSOTstar$- \resp $\toplocWOT$\=/continuous.

\begin{rem}
\makelabel{rem:equivalence-of-topologies-in-unitary-case:sig:article-free-raj-dahya}
    Consider the respective topologies in
    \Cref{%
        defn:k-sot-top:sig:article-free-raj-dahya,%
        defn:k-sot-star-top:sig:article-free-raj-dahya,%
        defn:k-wot-top:sig:article-free-raj-dahya%
    }
    but restricted to spaces of continuous contraction-valued functions,
    where the space of contractions is endowed
    with the strong \resp strong\textsuperscript{$\star$} \resp weak
    operator topologies.
    It is straightforward to observe that
    the $\toplocSOT$-/$\toplocSOTstar$-/$\toplocWOT$\=/topologies
    are completely determined
    by the topologies on the underlying operator spaces,
    \idest
        $(\BoundedOps{E},\tau)$,
    where $E$ is a Hilbert or Banach space
    and $\tau$ is the $\topSOT$-/$\topSOTstar$-/$\topSOT$\=/topology.%
    \footnote{%
        In fact, they coincide with the respective \highlightTerm{compact-open} topologies,
        \cf \cite[Theorem~7.11]{Kelley1955Book},
        \cite[Remark~1.12]{Dahya2022weakproblem}.
    }
    And since the strong, strong\textsuperscript{$\star$}, and weak
    topologies coincide for the subspace of unitary operators,
    the three topologies mentioned above
    \Cref{rem:equivalence-of-topologies-in-unitary-case:sig:article-free-raj-dahya}
    coincide when considering families of unitary $\Cnought$\=/semigroups
    (\cf \cite[Remark~1.17]{Dahya2022weakproblem}).
    From this, it is a simple exercise to arrive at the fact that
    for a family
        $
            \{U_{\omega}\}_{\omega \in \Omega}
            \subseteq
            \Repr{\reals}{\HilbertRaum}
        $
    of continuous unitary representations,
    the map
        ${
            \Omega
            \ni
            \omega
            \mapsto
            U_{\omega}
            \in
            \FnRm{\reals}{\BoundedOps{\HilbertRaum}}
        }$
        is $\toplocSOT$\=/continuous
    if and only if
        ${
            \Omega
            \ni
            \omega
            \mapsto
            U_{\omega}\restr{\realsNonNeg}
            \in
            \FnRm{\realsNonNeg}{\BoundedOps{\HilbertRaum}}
        }$
        is $\toplocSOT$\=/continuous.
    Hence, in the unitary case,
    one can switch between representations and their corresponding semigroups
    without affecting the topology.
\end{rem}

With these definitions we formulate the second main result:

\begin{highlightboxWithBreaks}
\begin{thm}[\Second Free dilation theorem]
\makelabel{thm:free-dilation:2nd:sig:article-free-raj-dahya}
    Let $\Omega$ be a compact topological space
    and
        $\{T_{\omega}\}_{\omega \in \Omega}$
    a family of contractive $\Cnought$\=/semigroups on a Hilbert space $\HilbertRaum$.
    There exists a free unitary dilation
        $(\HilbertRaum^{\prime},r,\{U_{\omega}\}_{\omega \in \Omega})$,
    such that
        $\{U_{\omega}\}_{\omega \in \Omega}$
    is $\toplocSOT$\=/continuous in $\Omega$
    if and only if
        $\{T_{\omega}\}_{\omega \in \Omega}$
    is $\toplocSOTstar$\=/continuous in $\Omega$.
\end{thm}
\end{highlightboxWithBreaks}



\subsection[Further motivation]{Further motivation}
\label{sec:introduction:motivation:sig:article-free-raj-dahya}

\firstparagraph
The primary motivation of this paper is to obtain non-commutative dilation results.
The underlying systems of evolution are further worthy of investigation.

Product expressions of non-commuting families of semigroups
as in \eqcref{eq:free-dilation:cts-time:sig:article-free-raj-dahya}
arise when considering systems
whose states are given by a vector $\xi$ in a Banach space $\BanachRaum$,
and whose evolutions can change in the course of time.
Note that such constructions occur frequently in semigroup theory,
\exempli
    in the Lie--Trotter product formula
    (which is applied in the celebrated the Feynman path integral)
    or more generally Chernoff approximations,
    in the study of random evolutions
    and evolution families,
    \etcetera,
    see
        \cite{%
            Chernoff1968article,%
            SmolyanovWeizsaeckerWittich2000chernoff,%
            SmolyanovVWeizsackerWittich2003Incollection,%
            SmolyanovWeizsaeckerWittich2007chernoff,%
            Butko2020chernoff,%
            GriegoHersh1971Article,%
            Goldstein1985semigroups,%
            Pazy1983Book,%
            Faris1967Article%
        }.
More concretely, given a family $\{T_{i}\}_{i \in I}$
of (not necessarily commuting)
$\Cnought$\=/semigroups indexed
by a set $I$
and a time interval $[0,\:t)$
together with a finite partition
    ${0 = t_{0} < t_{1} < t_{2} < \ldots < t_{N} = t}$,
we may envisage that the system evolves
    according to
    $T_{i_{1}}$ on $[t_{0},\:t_{1})$,
    according to
    $T_{i_{2}}$ on $[t_{1},\:t_{2})$,
    \textellipsis,
    and according to
    $T_{i_{N}}$ on $[t_{N-1},\:t_{N})$,
where
    $(i_{k})_{k=2}^{N} \subseteq I$.
The resulting state is then
    $\Big(\prod_{k=1}^{N}T_{i_{k}}(\tau_{k})\Big)\:\xi$,
where
    $\tau_{k} = t_{k} - t_{k-1}$
for each $k \in \{1,2,\ldots,N\}$.

Under suitable conditions, such product constructions
give rise to \emph{evolution families},
discussed at greater length
in \S{}\ref{sec:applications:sig:article-free-raj-dahya}.
Instrumental in the proof of the \Second free dilation theorem
is a reformulation of the Trotter--Kato theorem,
from which a \emph{diagonalisation construction}
precipitates as a by-product
(see \S{}\ref{sec:result-concrete:trotter-kato:sig:article-free-raj-dahya}),
which which shall apply to evolution families.



\subsection[Notation]{Notation}
\label{sec:introduction:notation:sig:article-free-raj-dahya}

\firstparagraph
Throughout this paper we fix the following conventions:

\begin{itemize}
\item
        $\naturals = \{1,2,\ldots\}$,
        $\naturalsZero = \{0,1,2,\ldots\}$,
        $\realsNonNeg = \{r\in\reals \mid r\geq 0\}$,
        and
        $\Torus = \{z\in\complex \mid \abs{z} = 1\}$.
    To distinguish from indices $i$ we use $\iunit$ for the imaginary unit $\sqrt{-1}$.

\item
    We write elements of product spaces in bold
    and denote their components in light face fonts with appropriate indices,
    \exempli the $i$\textsuperscript{th} components of
        $\mathbf{t} \in \reals^{d}$
        and
        $\mathbf{n} \in \naturalsZero^{d}$
        are denoted
            $t_{i}$ and $n_{i}$
        respectively.

\item
    In the case of concrete Hilbert spaces like
        $\HilbertRaum \in \{
            \ell^{2}(\integers),
            \ell^{2}(\naturalsZero),
            \complex^{N}
            \mid
            N \in \naturals
        \}$,
    the vector
        $\BaseVector{k} \in \HilbertRaum$
    denotes the $k$\textsuperscript{th}
    canonical unit vector of the standard orthonormal basis
    and
        $\ElementaryMatrix{i}{j} \in \BoundedOps{\HilbertRaum}$
    denotes the operator
    which satisfies
        $
            \ElementaryMatrix{i}{j}\xi
            =\brkt{\xi}{\BaseVector{j}}\:\BaseVector{i}
        $.
    This allows us, \exempli on $\ell^{2}(\naturalsZero)$
    to denote the discrete-time forwards-shift operator
    as
        $
            \ShiftRight
            = \sum_{n=0}^{\infty}
                \ElementaryMatrix{n+1}{n}
        $
    (with convergence in the $\topSOT$\=/sense).

\item
    For any Banach space $\BanachRaum$,
        $\onematrix$ denotes the identity operator
    and
        $
            \BoundedOps{\BanachRaum}
            \supseteq
            \OpSpaceC{\BanachRaum}
            \supseteq
            \OpSpaceU{\BanachRaum}
        $
    denote the sets of
        bounded linear operators%
        /contractions%
        /surjective isometries
    on $\BanachRaum$.
    In the Hilbert space setting,
    the latter coincides with the set of unitaries.

\item
    For any operator $S$ on a Hilbert space,
    we adopt the convention $S^{0} \coloneqq \onematrix$.

\item
    Empty products (\resp sums)
    shall always be assumed to equal the multiplicative
    (\resp additive) identity.

\item
    For arbitrary groups $G$, we let $e$ denote the neutral element.

\item
    For a Hilbert space $\HilbertRaum$
    and a group $G$,
    the set $\Repr{G}{\HilbertRaum}$
    shall denote the set of unitary representations
        ${U : G \to \OpSpaceU{\HilbertRaum}}$
    of $G$ on $\HilbertRaum$.
    And for a Banach space $\BanachRaum$,
    $\Repr{G}{\BanachRaum}$
    shall similarly denote the set of
    group homomorphisms between
    $G$ and $\OpSpaceU{\BanachRaum}$.
    In the case of topological groups,
    representations in this paper
    shall \emph{not} be taken to be (strongly) continuous,
    unless explicitly stated.
\end{itemize}

Letting $I$ be a non-empty index set,
we shall use the notation
    $\iota_{i}$ for $i \in I$
to denote certain canonical inclusions
dependent on the context:

\begin{itemize}
\item
    Given Hilbert spaces $H_{i}$ for $i \in I$
    and the direct sum $H \coloneqq \bigoplus_{i=1}^{n}H_{i}$,
    we let
        ${\iota_{i} : H_{i} \to H}$
    denote the canonical isometric embedding
    into the $i$\textsuperscript{th} component
    for each $i \in I$.

\item
    We habitually include in the presentation
    of the unique upto (topological) isomorphism
    free (topological) product $G = \FreeProduct{i \in I}G_{i}$
    of a family of (topological) groups $\{G_{i}\}_{i \in I}$,
    a family of (continuous) embeddings
        ${\iota_{i} : G_{i} \to G}$
    (\cf \S{}\ref{sec:algebra:sig:article-free-raj-dahya}).
\end{itemize}

There are two slightly different presentations of dilation we consider in this paper:

\begin{itemize}
\item
    Let $\HilbertRaum$ be a Hilbert space
    and let $(M,\cdot,e)$ be a topological monoid.
    A(n $\topSOT$\=/continuous) semigroup
    of bounded operators/contractions/isometries/unitaries
    over $M$ on $\HilbertRaum$
    is an ($\topSOT$\=/continuous) operator-valued map
        ${T : M \to \BoundedOps{\HilbertRaum}}$
    satisfying
        $T(x)$ is a bounded operator/contraction/isometry/unitary
        for each $x \in M$,
        $T(e) = \onematrix$,
        and
        $T(xy)  = T(x)T(y)$
        for all $x,y \in M$.
    A(n $\topSOT$\=/continuous) isometric/unitary dilation of $T$
    is a tuple $(\HilbertRaum^{\prime},r,U)$
    where
        $\HilbertRaum^{\prime}$ is a Hilbert space,
        $r \in \BoundedOps{\HilbertRaum}{\HilbertRaum^{\prime}}$ an isometry, and
        $U$ is a(n $\topSOT$\=/continuous) semigroup of bounded isometries/unitaries
        over $M$ on $\HilbertRaum^{\prime}$
        satisfying
            $T(\cdot) = r^{\ast}\:U(\cdot)\:r$.%
        \footnote{%
            One can readily check that this definition
            coincides with the usual definition for
            $(M,\cdot,e) = (\realsNonNeg^{d},+,\zerovector)$,
            $d\in\naturals$.
        }

    We let
        $\DilatableMonoids$
    be the class of
    all topological monoids $M$
    such that
    all $\topSOT$\=/continuous semigroups over $M$
    admit a dilation to an $\topSOT$\=/continuous unitary semigroup over $M$.

\item
    Let $\HilbertRaum$ be a Hilbert space
    and consider a pair $(G,M)$
    where $(G,\cdot,e)$ is a topological group
    and $M \subseteq G$ a submonoid endowed with the relative topology.
    In this setting a(n $\topSOT$\=/continuous) unitary dilation of $T$
    is a tuple $(\HilbertRaum^{\prime},r,U)$
    where
        $\HilbertRaum^{\prime}$ is a Hilbert space,
        $r \in \BoundedOps{\HilbertRaum}{\HilbertRaum^{\prime}}$ an isometry,
        and $U \in \Repr{G}{\HilbertRaum^{\prime}}$
        a(n $\topSOT$\=/continuous) representation of $G$ on $\HilbertRaum^{\prime}$
        such that
        the ($\topSOT$\=/continuous) unitary semigroup $U(\cdot)\restr{M}$
        is a dilation of $(\HilbertRaum,T)$.

    For $(M,\cdot,e) = (\realsNonNeg^{d},+,\zerovector)$, $d\in\naturals$,
    ($\topSOT$\=/continuous) unitary dilations shall mostly refer to this kind
    with $G \coloneqq \reals^{d}$.
    However note that in this special case,
    the above previous and current definitions are equivalent,
    since any ($\topSOT$\=/continuous) unitary semigroup
    over $\realsNonNeg^{d}$
    can canonically be extended to a(n $\topSOT$\=/continuous) representation of $\reals^{d}$.

    We let
        $\DilatableGroupMonoids$
    be the class of
    all pairs $(G,M)$ of topological groups and submonoids
    such that
    all $\topSOT$\=/continuous semigroups over $M$
    admit a dilation to an $\topSOT$\=/continuous unitary representation of $G$.
    Clearly
        $
            \{M \mid \exists{G:~}(G,M) \in \DilatableGroupMonoids\}
            \subseteq
            \DilatableMonoids
        $.
\end{itemize}



\subsection[Structure of the paper]{Structure of the paper}
\label{sec:introduction:structure:sig:article-free-raj-dahya}

\firstparagraph
In \S{}\ref{sec:result-bsk:sig:article-free-raj-dahya} a simple proof of
the \First main result is presented,
restricted to families of semigroups which arise from interpolations.
Then in \S{}\ref{sec:result-concrete:sig:article-free-raj-dahya}
an \emph{explicit construction} of free unitary dilations in full generality
is provided, relying on Sz.-Nagy's discrete-time constructions
and the theory of co-generators.
We also derive the \Second free dilation theorem
via a crucial reformulation of the Trotter--Kato theorem.
In \S{}\ref{sec:result-abstract:sig:article-free-raj-dahya}
we present a second full but \emph{abstract proof}
of the \First main result,
which relies on structure theorems
and mends the na\"{i}ve approach
(\cf \Cref{rem:naive-approach:sig:article-free-raj-dahya}).
And in \S{}\ref{sec:algebra:sig:article-free-raj-dahya}
this result is reframed algebraically
in terms of semigroups on free topological products.

The final two sections consider applications:
In \S{}\ref{sec:residuality:sig:article-free-raj-dahya}
various residuality results are derived.
And in \S{}\ref{sec:applications:sig:article-free-raj-dahya}
we discuss at length various problems
relevant to evolution families,
and bring to bear our reformulation of the Trotter--Kato theorem.




\section[Free dilations of interpolations]{Free dilations of interpolations}
\label{sec:result-bsk:sig:article-free-raj-dahya}

\firstparagraph
In this section we motivate the feasibility of
the \First free dilation theorem
by restricting our attention to certain kinds of semigroups.
Doing so allows for a simpler approach.

In \cite{%
    BhatSkeide2015PureCounterexamples,%
    Dahya2024interpolation%
}
techniques were introduced to construct
interpolations of families $\{S_{i}\}_{i \in I}$
of (not necessarily commuting) contractions
on a Hilbert space $\HilbertRaum$
to families of $\Cnought$\=/semigroups,
with the caveat that the interpolations exist in general on larger spaces.
The construction is as follows
(see \cite[\S{}2.1--2]{Dahya2024interpolation}):
Let $\Torus$ be endowed with the standard probability structure
and let $\Torus^{I}$ denote the product probability space.
Strongly right-continuous $1$-periodic families
    $\{P_{i}(t)\}_{t \in \reals} \in \BoundedOps{L^{2}(\Torus^{I})}$
of projections
as well as strongly continuous representations
    $W_{i} \in \Repr{\reals}{L^{2}(\Torus^{I})}$,
    $i \in I$
are defined,
such that
    $P_{i}(0) = \onematrix$
for $i \in I$
and
    $\{P_{i}(s),P_{j}(t)\mid s,t\in\reals\}$,
    $\{W_{i}(s),W_{j}(t)\mid s,t\in\reals\}$,
    and
    $\{P_{i}(s),W_{j}(t)\mid s,t\in\reals\}$
are commuting families
for $i,j \in I$ with $i \neq j$.
The \highlightTerm{Bhat--Skeide} interpolation
    $\{T_{i}\}_{i \in I}$
    of $\{S_{i}\}_{i \in I}$
is a family of contractive $\Cnought$\=/semigroups
on $L^{2}(\Torus^{I})\otimes\HilbertRaum$
can then be defined by

    \begin{displaymath}
        T_{i}(t)
        = W_{i}(t)P_{i}(t) \otimes S_{i}^{\floor{t}}
            +
            W_{i}(t)(\onematrix - P_{i}(t)) \otimes S_{i}^{\floor{t+1}}
    \end{displaymath}

\continueparagraph
for each $i \in I$ and $t \in \realsNonNeg$.
Some basic properties of these interpolations are:

\begin{kompaktenum}{\bfseries 1.}
\item
    $T_{i}(n) = \onematrix \otimes S_{i}^{n}$ for $n\in\naturalsZero$ and $i\in I$;
\item
    $\{T_{i}\}_{i \in I}$ is a commuting family
    if and only if
    $\{S_{i}\}_{i \in I}$ is a commuting family;
\item
    $\{T_{i}\}_{i \in I}$ is a family of contractive/isometric/unitary semigroups
    if and only if
    $\{S_{i}\}_{i \in I}$ is a family of contractions/isometries/unitary operators.
\end{kompaktenum}

It is natural to consider semigroups unitarily equivalent
to time-scaled interpolations.
To this extent we may define the following
(\cf \cite[\S{}2.4.2]{Dahya2024interpolation}):

\begin{defn}
\makelabel{defn:scaled-bsk:sig:article-free-raj-dahya}
    Let $I$ be a non-empty index set
    and $\{T_{i}\}_{i \in I}$ a family of contractive $\Cnought$\=/semigroups
    on a Hilbert space $\HilbertRaum$.
    We say that $\{T_{i}\}_{i \in I}$
    is \highlightTerm{similar to a time-scaled Bhat--Skeide dilation}
    in case there exists
    a family $\{S_{i}\}_{i \in I}$ of contractions on a Hilbert space $H$,
    a unitary operator
        $w \in \BoundedOps{L^{2}(\Torus^{I}) \otimes H}{\HilbertRaum}$,
    and
    scales $(\eps_{i})_{i \in I} \subseteq (0,\:\infty)$,
    such that

        \begin{restoremargins}
        \begin{equation}
        \label{eq:scaled-bsk:sig:article-free-raj-dahya}
            T_{i}(t)
            = w\:\Big(
                W_{i}(\tfrac{t}{\eps_{i}})P_{i}(\tfrac{t}{\eps_{i}})
                \otimes
                S_{i}^{\floor{\tfrac{t}{\eps_{i}}}}
                +
                W_{i}(\tfrac{t}{\eps_{i}})(\onematrix - P_{i}(\tfrac{t}{\eps_{i}}))
                \otimes
                S_{i}^{\floor{\tfrac{t}{\eps_{i}}+1}}
            \Big)\:w^{\ast}
        \end{equation}
        \end{restoremargins}

    \continueparagraph
    for $i \in I$, $t\in\realsNonNeg$,
    where $W_{i},P_{i}$ are as above.
\end{defn}

We shall also make use of the following notation:
For any projection $p$ on a Hilbert space,
$p^{\perp^{0}} \coloneqq p$
and
$p^{\perp^{1}} \coloneqq p^{\perp} \coloneqq \onematrix - p$.
This allows us to rewrite \eqcref{eq:scaled-bsk:sig:article-free-raj-dahya} as

    \begin{restoremargins}
    \begin{equation}
    \label{eq:scaled-bsk:products:sig:article-free-raj-dahya}
    \everymath={\displaystyle}
    \begin{array}[m]{rcl}
        \prod_{k=1}^{N}
            T_{i_{k}}(t_{k})
        &= &w\:\Big(
                \prod_{k=1}^{N}
                \sum_{e\in\{0,1\}}
                    W_{i_{k}}(\tfrac{t}{\eps_{i_{k}}})
                    P^{\perp^{e}}_{i_{k}}(\tfrac{t}{\eps_{i_{k}}})
                    \otimes
                    S_{i_{k}}^{\floor{\tfrac{t}{\eps_{i_{k}}} + e}}
            \Big)\:w^{\ast}\\
        &= &w\:
            \Big(
            \sum_{e\in\{0,1\}^{N}}
                \prod_{k=1}^{N}
                    W_{i_{k}}(\tfrac{t}{\eps_{i_{k}}})
                    P^{\perp^{e_{k}}}_{i_{k}}(\eps_{i_{k}}^{-1}t)
                    \otimes
                    S_{i_{k}}^{\floor{\tfrac{t}{\eps_{i_{k}}} + e_{k}}}
            \Big)\:w^{\ast}
    \end{array}
    \end{equation}
    \end{restoremargins}

\continueparagraph
for $N\in\naturals$,
$(i_{k})_{k=1}^{N} \subseteq I$,
and $(t_{k})_{k=1}^{N} \subseteq \realsNonNeg$
(\cf \cite[\S{}2.4.2]{Dahya2024interpolation}).


\subsection[Restricted proof of the Ist main result]{Restricted proof of the \First main result}
\label{sec:result-bsk:proof-1st:sig:article-free-raj-dahya}

\firstparagraph
By confining our attention to families of semigroups that arise via interpolations,
one readily obtains free dilations without any high-powered techniques.

\def\beweislabel{thm:free-dilation:1st:sig:article-free-raj-dahya}
\begin{proof}[of \Cref{thm:free-dilation:1st:sig:article-free-raj-dahya}, restricted context]
    Suppose that
        $\{T_{i}\}_{i \in I}$
    is similar to a time-scaled Bhat-Skeide dilation.
    That is, there exists a family
        $\{S_{i}\}_{i \in I}$
    of contractions on a Hilbert space $H$,
    a unitary operator
        $w \in \BoundedOps{L^{2}(\Torus^{I}) \otimes H}{\HilbertRaum}$,
    and $(\eps_{i})_{i \in I} \subseteq (0,\:\infty)$,
    such that \eqcref{eq:scaled-bsk:sig:article-free-raj-dahya} holds.

    By the free dilation theorem of Sz.-Nagy (\Cref{thm:free-dilation-discrete-time:sig:article-free-raj-dahya})
    there exists
        a family $\{V_{i}\}_{i \in I}$
        of unitaries on a Hilbert space $H^{\prime}$
        and
        an isometry $r \in \BoundedOps{H}{H^{\prime}}$,
    such that \eqcref{eq:free-dilation:discrete-time:sig:article-free-raj-dahya} holds.

    Using the Bhat--Skeide interpolation
    as well as the geometric transformations
    we obtain a family
        $\{U_{i}\}_{i \in I}$
    of $\topSOT$\=/continuous unitary representations
    of $\reals$
    on $\HilbertRaum^{\prime} \coloneqq L^{2}(\Torus^{I}) \otimes H^{\prime}$
    defined by

    \begin{displaymath}
        U_{i}(t)
            = \sum_{e \in \{0,1\}}
                W_{i}(\tfrac{t}{\eps_{i}})
                P_{i}(\tfrac{t}{\eps_{i}})^{\perp^{e}}
                \otimes
                V_{i}^{\floor{\tfrac{t}{\eps_{i}} + e}}\\
    \end{displaymath}

    \continueparagraph
    for $t \in \realsNonNeg$.
    \Withoutlog we may extend each $U_{i}$ to a strongly continuous unitary representation
    of $\reals$ on $\HilbertRaum^{\prime}$.
    Finally, letting
        $\tilde{r} \coloneqq (\onematrix \otimes r)\:w^{\ast} \in \BoundedOps{\HilbertRaum}{\HilbertRaum^{\prime}}$,
    which is an isometry,
    one computes for
        $N\in\naturals$,
        $(i_{k})_{k=1}^{N} \subseteq I$,
        and
        $(t_{k})_{k=1}^{N} \subseteq \realsNonNeg$

    \begin{longeqnarray}
        \tilde{r}^{\ast}
        \:\Big(
            \prod_{k=1}^{N}
                U_{i_{k}}(t_{k})
        \Big)
        \:\tilde{r}
            &\eqcrefoverset{eq:scaled-bsk:products:sig:article-free-raj-dahya}{=}
                &w
                \:(\onematrix \otimes r^{\ast})
                \:\Big(
                    \sum_{e \in \{0,1\}^{N}}
                    \prod_{k=1}^{N}
                        W_{i_{k}}(\tfrac{t_{k}}{\eps_{i_{k}}})
                        P^{\perp^{e_{k}}}_{i_{k}}(\tfrac{t_{k}}{\eps_{i_{k}}})
                        \otimes
                        V_{i_{k}}^{\floor{\tfrac{t_{k}}{\eps_{i_{k}}} + e_{k}}}
                \Big)
                \:(\onematrix \otimes r)
                \:w^{\ast}\\
            &=
                &w
                \:\Big(
                    \sum_{e \in \{0,1\}^{N}}
                        \Big(
                            \prod_{k=1}^{N}
                                W_{i_{k}}(\tfrac{t_{k}}{\eps_{i_{k}}})
                                P^{\perp^{e_{k}}}_{i_{k}}(\tfrac{t_{k}}{\eps_{i_{k}}})
                        \Big)
                        \otimes
                        r^{\ast}
                        \Big(
                        \prod_{k=1}^{N}
                            V_{i_{k}}^{\floor{\tfrac{t_{k}}{\eps_{i_{k}}} + e_{k}}}
                        \Big)
                        r
                \Big)
                \:w^{\ast}\\
            &\eqcrefoverset{eq:free-dilation:discrete-time:sig:article-free-raj-dahya}{=}
                &w
                \:\Big(
                    \sum_{e \in \{0,1\}^{N}}
                        \Big(
                        \prod_{k=1}^{N}
                            W_{i_{k}}(\tfrac{t_{k}}{\eps_{i_{k}}})
                            P^{\perp^{e_{k}}}_{i_{k}}(\tfrac{t_{k}}{\eps_{i_{k}}})
                        \Big)
                        \otimes
                        \Big(
                        \prod_{k=1}^{N}
                            S_{i_{k}}^{\floor{\tfrac{t_{k}}{\eps_{i_{k}}} + e_{k}}}
                        \Big)
                \Big)
                \:w^{\ast}\\
            &=
                &w
                \:\Big(
                    \sum_{e \in \{0,1\}^{N}}
                        \prod_{k=1}^{N}
                            W_{i_{k}}(\tfrac{t_{k}}{\eps_{i_{k}}})
                            P^{\perp^{e_{k}}}_{i_{k}}(\tfrac{t_{k}}{\eps_{i_{k}}})
                        \otimes
                            S_{i_{k}}^{\floor{\tfrac{t_{k}}{\eps_{i_{k}}} + e_{k}}}
                \Big)
                \:w^{\ast}\\
            &\eqcrefoverset{eq:scaled-bsk:products:sig:article-free-raj-dahya}{=}
                &\prod_{k=1}^{N}
                    T_{i_{k}}(t_{k}).
    \end{longeqnarray}

    \continueparagraph
    Thus $(\HilbertRaum^{\prime},\tilde{r},\{U_{i}\}_{i \in I})$
    constitutes a free unitary dilation of $(\HilbertRaum,\{T_{i}\}_{i \in I})$.
\end{proof}

Time-scaled Bhat--Skeide interpolations
are not entirely pathological constructions.
Indeed, at least in the commutative setting,
by \cite[Proposition~2.16]{Dahya2024interpolation}
such families can arbitrarily approximate
any family of $\Cnought$\=/semigroups
in a certain weak sense.
Whilst the above restricted result
cannot be used to obtain the general result,
it at least provides us an indication
that families of contractive $\Cnought$\=/semigroups
may in general admit free unitary dilations,
which shall indeed be proved in the subsequent sections.




\section[Free dilations explicitly constructed]{Free dilations explicitly constructed}
\label{sec:result-concrete:sig:article-free-raj-dahya}

\firstparagraph
Our first \emph{full proof} of
the \First free dilation theorem
is an explicit construction which rests on two techniques.
The steps in Sz.-Nagy's classical dilation results
\cite{%
    Nagy1960Article,%
    Nagy1970%
}
allow us to build operators with appropriate properties in discrete-time.
The theory of co-generators then allows us to switch between
the continuous- and discrete-time settings \akin S\l{}oci\'{n}ski
\cite{%
    Slocinski1974,%
    Slocinski1982%
}.


\subsection[Discrete Dilation]{Discrete Dilation}
\label{sec:result-concrete:discrete-dilation:sig:article-free-raj-dahya}

\firstparagraph
We begin by sketching a proof of
\Cref{thm:free-dilation-discrete-time:sig:article-free-raj-dahya}.
The constructions underlying this result
are due to Sch\"{a}ffer in the case of dilations of a single contraction,
but were later modified by Sz.-Nagy
(see
    \cite[\S{}I.5]{Nagy1970}
    as well as
    \cite{%
        Schaffer1955Article,%
        Nagy1960Article%
    }%
).
The following presentation modifies Sz.-Nagy's construction slightly,
but captures the essence of his result.

\begin{proof}[of \Cref{thm:free-dilation-discrete-time:sig:article-free-raj-dahya}, sketch]
    Let $\HilbertRaum$ be a Hilbert space,
    $I$ be a non-empty index set and $\{S_{i}\}_{i \in I} \subseteq \BoundedOps{\HilbertRaum}$
    a (not necessarily commuting) family of contractions.

    Set
    $
        H_{0}
        \coloneqq
            \HilbertRaum \otimes \ell^{2}(\naturalsZero)
    $
    and
    $
        H_{1}
        \coloneqq
            H_{0} \otimes \complex^{2}
        = \HilbertRaum \otimes \ell^{2}(\naturalsZero) \otimes \complex^{2}
    $.
    We further define the isometries
    $r_{0} \in \BoundedOps{\HilbertRaum}{H_{0}}$
    and $r_{1} \in \BoundedOps{\HilbertRaum}{H_{1}}$
    via
        $r_{0} \coloneqq \onematrix \otimes \BaseVector{0}$
    and
        $r_{1} \coloneqq \onematrix \otimes \BaseVector{0} \otimes \BaseVector{0}$,
    \idest
        $r_{0}\xi = \xi \otimes \BaseVector{0}$
    and
        $r_{1}\xi = \xi \otimes \BaseVector{0} \otimes \BaseVector{0}$
    for $\xi \in \HilbertRaum$.

    For any contraction $S \in \BoundedOps{\HilbertRaum}$
    define

    \begin{restoremargins}
    \begin{equation}
    \label{eq:schaeffer-nagy-dilation:isometry:sig:article-free-raj-dahya}
    \everymath={\displaystyle}
    \begin{array}[m]{rcl}
        \NagyIso_{S}
            &\coloneqq
                &S \otimes \ElementaryMatrix{0}{0}
                + \HalmosOp_{S} \otimes \ElementaryMatrix{1}{0}
                + I \otimes \sum_{n=1}^{\infty} \ElementaryMatrix{n+1}{n}\\
            &= &\begin{matrix}{ccccc}
                S &\zeromatrix &\zeromatrix &\zeromatrix &\cdots\\
                \HalmosOp_{S} &\zeromatrix &\zeromatrix &\zeromatrix &\cdots\\
                \zeromatrix &\onematrix &\zeromatrix &\zeromatrix &\cdots\\
                \zeromatrix &\zeromatrix &\onematrix &\zeromatrix &\cdots\\
                \vdots &\vdots &\vdots &\vdots &\ddots\\
            \end{matrix}
            \in \BoundedOps{H_{0}}
    \end{array}
    \end{equation}
    \end{restoremargins}

    \continueparagraph
    and

    \begin{restoremargins}
    \begin{equation}
    \label{eq:schaeffer-nagy-dilation:unitary:sig:article-free-raj-dahya}
    \everymath={\displaystyle}
    \begin{array}[m]{rcl}
        \NagyUtr_{S}
            &\coloneqq
                &\NagyIso_{S} \otimes \ElementaryMatrix{0}{0}
                + (\onematrix - \NagyIso_{S}\NagyIso_{S}^{\ast}) \otimes \ElementaryMatrix{0}{1}
                + \NagyIso_{S}^{\ast} \otimes \ElementaryMatrix{1}{1}\\
            &= &\begin{matrix}{cc}
                \NagyIso_{S} &\onematrix - \NagyIso_{S}\NagyIso_{S}^{\ast}\\
                \zeromatrix &\NagyIso_{S}^{\ast}\\
            \end{matrix}
            \in \BoundedOps{H_{1}},
    \end{array}
    \end{equation}
    \end{restoremargins}

    \continueparagraph
    where $\HalmosOp_{S} \coloneqq \sqrt{\onematrix - S^{\ast}S} \in \BoundedOps{\HilbertRaum}$.
    It is a straightforward exercise to verify that
        $\NagyIso_{S}$ is an isometry
        and
        $\NagyUtr_{S}$ a unitary operator.

    Thus $\{\NagyUtr_{S_{i}}\}_{i \in I}$ is a family of unitaries.
    Let $N \in \naturals$,
        $(i_{k})_{k=1}^{N} \subseteq I$,
        and
        $(n_{k})_{k=1}^{N} \subseteq \naturalsZero$.
    A simple induction argument reveals that
        $
            r_{1}^{\ast}
            \:\Big(
                \prod_{k=1}^{N}
                \NagyUtr_{S_{i_{k}}}^{n_{k}}
            \Big)
            \:r_{1}
            = r_{0}^{\ast}
            \:\Big(
                \prod_{k=1}^{N}
                \NagyIso_{S_{i_{k}}}^{n_{k}}
            \Big)
            \:r_{0}
        $
    and that there exists
        $\{R_{k}\}_{k=1}^{N} \subseteq \BoundedOps{\HilbertRaum}$
    such that
        $\prod_{k=1}^{N}\NagyIso_{S_{i_{k}}}^{n_{k}}
        = \prod_{k=1}^{N}
            S_{i_{k}}^{n_{k}}
            \otimes
            \ElementaryMatrix{0}{0}
            +
            \sum_{k=1}^{\infty}
                R_{k}
                \otimes
                \ElementaryMatrix{k+1}{k}
        $.
    Putting this together yields

    \begin{displaymath}
        r_{1}^{\ast}
        \:\Big(
            \prod_{k=1}^{N}
            \NagyUtr_{S_{i_{k}}}^{n_{k}}
        \Big)
        \:r_{1}
            =
                r_{0}^{\ast}
                \:\Big(
                    \prod_{k=1}^{N}
                    S_{i_{k}}^{n_{k}}
                    \otimes
                    \ElementaryMatrix{0}{0}
                    +
                    \sum_{k=1}^{\infty}
                        R_{k}
                        \otimes
                        \ElementaryMatrix{k+1}{k}
                \Big)
                \:r_{0}
            =
                \prod_{k=1}^{N}
                    S_{i_{k}}^{n_{k}},
    \end{displaymath}

    \continueparagraph
    which demonstrates that
    $(H_{1},r_{1},\{\NagyUtr_{S_{i}}\}_{i \in I})$
    is a free unitary dilation of
    $(\HilbertRaum,\{S_{i}\}_{i \in I})$.
\end{proof}

We now make note of a crucial spectral property of the dilation
in \eqcref{eq:schaeffer-nagy-dilation:unitary:sig:article-free-raj-dahya}.
First recall the following well-known result
(see \exempli \cite[Proposition~I.3.1]{Nagy1970}).

\begin{prop}
\makelabel{prop:eigenvalues-on-border:sig:article-free-raj-dahya}
    Let $S \in \BoundedOps{\HilbertRaum}$
    and $\lambda \in \complex$.
    If $\abs{\lambda} \geq \norm{S}$
    then
        $
            \ker(\lambda\cdot\onematrix - S)
            = \ker(\lambda^{\ast}\cdot\onematrix - S^{\ast})
        $.
\end{prop}

    \begin{proof}
        It clearly suffices to prove the $\subseteq$-inclusion.
        Let $\xi \in \ker(\lambda\cdot\onematrix - S) \without \{\zerovector\}$.
        Let $\alpha\in\complex$ and $\eta\in\{\xi\}^{\perp}$
        be such that
            $S^{\ast}\xi = \alpha\xi + \eta$.
        Then
            $
                \alpha^{\ast}\norm{\xi}^{2}
                = \brkt{\xi}{S^{\ast}\xi}
                = \brkt{S\xi}{\xi}
                = \lambda\norm{\xi}^{2}
            $.
        So $\alpha = \lambda^{\ast}$.
        Since $
            \norm{\eta}^{2} + \abs{\alpha}^{2}\norm{\xi}^{2}
            = \norm{S^{\ast}\xi}^{2}
            \leq \norm{S}^{2}\norm{\xi}^{2}
            \leq \abs{\lambda}^{2}\norm{\xi}^{2}
        $,
        it follows that $\eta=\zerovector$.
        Thus $S^{\ast}\xi = \lambda^{\ast}\xi$,
        whence $\xi \in \ker(\lambda^{\ast}\cdot\onematrix - S^{\ast}) \without \{\zerovector\}$.
    \end{proof}

\begin{prop}[Point spectrum of Sz.-Nagy dilations]
\makelabel{prop:eigenvalues-discrete-dilation:sig:article-free-raj-dahya}
    Let $S \in \BoundedOps{\HilbertRaum}$ be a contraction.
    Then

    \begin{displaymath}
        \sigma_{p}(S) \cap \Torus
        \subseteq
            \sigma_{p}(\NagyUtr_{S})
        \subseteq
            (\sigma_{p}(S) \cup \sigma_{p}(S)^{\ast}) \cap \Torus
    \end{displaymath}

    \continueparagraph
    where $\sigma_{p}$ denotes the point spectrum.
    In particular
    $1 \notin \sigma_{p}(\NagyUtr_{S})$
    if and only if
    $1 \notin \sigma_{p}(S)$.
\end{prop}

    \begin{proof}
        Towards first inclusion,
        let $\lambda \in \sigma_{p}(S) \cap \Torus$.
        Since $\abs{\lambda} = 1 \geq \norm{S}$,
        by \Cref{prop:eigenvalues-on-border:sig:article-free-raj-dahya}
        there exists
            $\xi \in \HilbertRaum \without \{\zerovector\}$,
        such that
            $S \xi = \lambda \xi$
            and
            $S^{\ast} \xi = \lambda^{\ast} \xi$.
        So
            $
                (\onematrix - S^{\ast}S)\xi
                = \xi - \lambda^{\ast}\lambda\xi
                = \zerovector
            $,
        which implies
            $\HalmosOp_{S}\xi = \zerovector$,
        since
            $
                \norm{\HalmosOp_{S}\xi}^{2}
                = \brkt{\HalmosOp_{S}^{2}\xi}{\xi}
                = \brkt{(\onematrix - S^{\ast}S)\xi}{\xi}
                = 0
            $.
        We thus have

            \begin{shorteqnarray}
                \NagyUtr_{S}\:r_{1}\xi
                &= &\NagyIso_{S}\:r_{0}\xi \otimes \BaseVector{0}\\
                &= &(S\xi \otimes \BaseVector{0} + \cancel{\HalmosOp_{S}\xi} \otimes \BaseVector{1}) \otimes \BaseVector{0}\\
                &= &S\xi \otimes \BaseVector{0} \otimes \BaseVector{0}\\
                &= &r_{1}\:(S\xi)\\
                &= &\lambda\:r_{1}\xi,
            \end{shorteqnarray}

        \continueparagraph
        whence $r_{1}\xi \in \ker(\lambda\cdot\onematrix - \NagyUtr_{S}) \without \{\zerovector\}$.
        So $\lambda \in \sigma_{p}(\NagyUtr_{S})$.

        Towards the second inclusion,
        let $\lambda \in \sigma_{p}(\NagyUtr_{S})$.
        Then $\lambda \in \Torus$, since $\NagyUtr_{S}$ is unitary.
        Let $x,y\in H_{0}$
        be such that $x \otimes \BaseVector{0} + y \otimes \BaseVector{1} \neq \zerovector$,
        \idest either $x \neq \zerovector$ or $y \neq \zerovector$,
        and such that

            \begin{displaymath}
                \begin{smatrix}
                    \lambda x\\
                    \lambda y\\
                \end{smatrix}
                = \NagyUtr_{S}
                    \begin{smatrix}
                        x\\
                        y\\
                    \end{smatrix}
                = \begin{smatrix}
                        \NagyIso_{S}x + (\onematrix - \NagyIso_{S}^{\ast}\NagyIso_{S})y\\
                        \NagyIso_{S}^{\ast}y\\
                    \end{smatrix}.
            \end{displaymath}

        \continueparagraph
        There are two cases to consider:

        \begin{enumerate}{\itshape {Case} 1.}
        \item
            If $y = \zerovector$, then $x \neq \zerovector$
            and
            $
                \lambda x
                = \NagyIso_{S}x + (\onematrix - \NagyIso_{S}^{\ast}\NagyIso_{S})\zerovector
                = \NagyIso_{S}x
            $.
            Writing
                $
                    x
                    = \sum_{n \in \naturalsZero}
                        x_{n} \otimes \BaseVector{n}
                $,
            one has

                \begin{displaymath}
                    \sum_{n=0}^{\infty}
                        \lambda x_{n} \otimes \BaseVector{n}
                    = \lambda x
                    = \NagyIso_{S}x
                    = Sx_{0} \otimes \BaseVector{0}
                    + \HalmosOp_{S}x_{0} \otimes \BaseVector{1}
                    + \sum_{n=1}^{\infty} x_{n} \otimes \BaseVector{n+1},
                \end{displaymath}

            \continueparagraph
            which implies
                $\lambda x_{0} = Sx_{0}$
            and
                $\lambda x_{n} = x_{n-1}$
            for all $n \geq 2$.
            A simple induction argument yields that
                $
                    x = x_{0} \otimes \BaseVector{0}
                    + \sum_{n=1}^{\infty}
                        \lambda^{-(n-1)}
                            x_{1} \otimes \BaseVector{n}
                $.
            Since $\abs{\lambda} = 1$,
            this is only possibly if $x_{1} = \zerovector$.
            Thus $x = x_{0} \otimes \BaseVector{0}$
            and since $x \neq \zerovector$,
            this implies $x_{0} \neq \zerovector$.
            Since $Sx_{0} = \lambda x_{0}$,
            this proves that $\lambda \in \sigma_{p}(S)$.

        \item
            Otherwise $y \neq \zerovector$.
            Since
                $\NagyIso_{S}^{\ast}y = \lambda y$
            and
                $\abs{\lambda} = 1 = \norm{\NagyIso_{S}}$,
            applying \Cref{prop:eigenvalues-on-border:sig:article-free-raj-dahya}
            yields that
                $\NagyIso_{S}y = \lambda^{\ast} y$.
            Arguing as in Case 1 yields that
                $\lambda^{\ast} \in \sigma_{p}(S)$.
        \end{enumerate}

        \continueparagraph
        Hence $\sigma_{p}(\NagyUtr_{S}) \subseteq (\sigma_{p}(S) \cup \sigma_{p}(S)^{\ast}) \cap \Torus$.
    \end{proof}



\subsection[Co-generators]{Co-generators}
\label{sec:result-concrete:cogenerators:sig:article-free-raj-dahya}

\firstparagraph
We first recall some basic facts which can be found in
    \cite[\S{}III.8, Theorem~III.8.1, and Proposition~III.8.2]{Nagy1970},
    see also
    \cite[\S{}I.3.1 and Theorem~I.3.4]{Eisner2010buchStableOpAndSemigroups}.
Given a contractive $\Cnought$\=/semigroup $T$ on a Hilbert space $\HilbertRaum$
with generator $A$, the \highlightTerm{co-generator} of $T$
is defined by
    $
        V = (A + \onematrix)(A - \onematrix)^{-1}
        = \onematrix - 2\opResolvent{A}{1}
    $
which is a contraction on $\HilbertRaum$
for which $1$ is not an eigenvalue.%
\footnote{%
    Recall that
    $
        \opResolvent{A}{\lambda}
        = (\lambda\cdot\onematrix - A)^{-1}
    $
    denotes the resolvent operator
    for $\lambda \in \complex$
    outside the spectrum of $A$.
}
Conversely, given a contraction $V \in \BoundedOps{\HilbertRaum}$
with $1 \notin \sigma_{p}(V)$,
there exists a unique
contractive $\Cnought$\=/semigroup $T$ on $\HilbertRaum$
whose co-generator is $V$.
Furthermore $T$ is a unitary (\resp isometric) semigroup
if and only if its co-generator $V$ is a unitary (\resp isometric) operator.

Our goal here is to provide uniform means
of recovering a $\Cnought$\=/semigroup
from its co-generator using only algebraic constructions.
To this end we develop the following
technical results.

\begin{prop}[Yosida-Approximants in terms of co-generators]
\makelabel{prop:hille-yosida-cogenerator:sig:article-free-raj-dahya}
    Let $T$ be a contractive $\Cnought$\=/semigroup
    on a Hilbert space $\HilbertRaum$
    with co-generator $V$.
    For each $\lambda \in (1,\:\infty)$
    set
    $\gamma_{\lambda} \coloneqq \frac{\lambda + 1}{\lambda - 1}$,
    $\alpha_{\lambda} \coloneqq \frac{\lambda}{\lambda - 1}$,
    and
    $\beta_{\lambda} \coloneqq 2\alpha_{\lambda}^{2}$.
    Then

    \begin{displaymath}
        T^{(\lambda)}(t)
        \coloneqq
        e^{
            t(
                \alpha_{\lambda}\cdot\onematrix
                -
                \beta_{\lambda}(\gamma_{\lambda}\cdot\onematrix - V)^{-1}
            )
        }
        \overset{\tinytopSOT}{\longrightarrow}
        T(t)
    \end{displaymath}

    \continueparagraph
    uniformly for $t$ in compact subsets of $\realsNonNeg$
    for ${(1,\:\infty)\ni\lambda \longrightarrow \infty}$.
\end{prop}

    \begin{proof}
        We first recall some basic facts about approximants
        (\cf
            \cite[Theorem~II.3.5]{EngelNagel2000semigroupTextBook},
            \cite[(12.3.4), p.~361]{Hillephillips1957faAndSg}%
        ).
        Let $A$ be the generator of $T$.
        The \highlightTerm{$\lambda$\textsuperscript{th} Yosida approximant} of $T$,
        is given by
        $
            \{ e^{tA^{(\lambda)}} \}_{t\in\realsNonNeg}
        $
        for each $\lambda \in (0,\:\infty)$,
        where

            \begin{restoremargins}
            \begin{equation}
            \label{eq:yosida-approx-generator:sig:article-free-raj-dahya}
                A^{(\lambda)}
                \colonequals
                    \lambda A \opResolvent{A}{\lambda}
                = \lambda^{2} \opResolvent{A}{\lambda}
                    - \lambda\cdot\onematrix
            \end{equation}
            \end{restoremargins}

        \continueparagraph
        is a bounded operator.
        The Yosida approximants satisfy
            $
                \sup_{t \in L}\norm{
                    (e^{tA^{(\lambda)}} - T(t))
                    \xi
                } \longrightarrow 0
            $
        for ${(0,\:\infty) \ni \lambda \longrightarrow \infty}$
        for each
            $\xi \in \HilbertRaum$
            and
            compact $L \subseteq \realsNonNeg$.

        So to prove the claim, it suffices to express
            $A^{(\lambda)}$ in terms of $V$
        for each $\lambda \in (1,\:\infty)$.
        Simple algebraic manipulation
        (see also \cite[Lemma~2.1~2)]{EisnerZwart2008cogenerator})
        yields
            $
                \opResolvent{A}{\lambda}
                = \tfrac{1}{\lambda - 1}(\onematrix - V)\opResolvent{V}{\gamma_{\lambda}}
            $,
        from which one obtains

            \begin{restoremargins}
            \begin{equation}
            \label{eq:yosida-in-terms-of-cogen:sig:article-free-raj-dahya}
            \everymath={\displaystyle}
            \begin{array}[m]{rcl}
                A^{(\lambda)}
                &\eqcrefoverset{eq:yosida-approx-generator:sig:article-free-raj-dahya}{=}
                    &\tfrac{\lambda^{2}}{\lambda-1}
                    (\onematrix - V)
                    \opResolvent{V}{\gamma_{\lambda}}
                    - \lambda\cdot\onematrix\\
                &= &\tfrac{\lambda^{2}}{\lambda-1}
                    \Big(
                        (\gamma_{\lambda}\cdot\onematrix - V)
                        +
                        (1 - \gamma_{\lambda})\onematrix
                    \Big)
                    \opResolvent{V}{\gamma_{\lambda}}
                    - \lambda\cdot\onematrix\\
                &= &\tfrac{\lambda^{2}}{\lambda-1}\cdot\onematrix
                    - \lambda\cdot\onematrix
                    +\tfrac{\lambda^{2}}{\lambda-1}
                    (1 - \gamma_{\lambda})
                    \opResolvent{V}{\gamma_{\lambda}}\\
                &= &\tfrac{\lambda}{\lambda-1}\cdot\onematrix
                    -\tfrac{2\lambda^{2}}{(\lambda - 1)^{2}}
                    \opResolvent{V}{\gamma_{\lambda}}\\
                &= &\alpha_{\lambda}\cdot\onematrix
                    - \beta_{\lambda}(\gamma_{\lambda}\cdot\onematrix - V)^{-1}
            \end{array}
            \end{equation}
            \end{restoremargins}

        \continueparagraph
        whence the claimed convergence holds.
    \end{proof}

\begin{lemm}
\makelabel{lemm:poly-approx-cogenerator:sig:article-free-raj-dahya}
    Let $t \in \realsNonNeg$.
    There exists a net $(p^{(\alpha)}_{t})_{\alpha \in \Lambda}$
    of real-valued polynomials
    with the following property:
    For any contractive $\Cnought$\=/semigroup $T$
    on a Hilbert space $\HilbertRaum$
    with co-generator $V$,
    it holds that
    ${p^{(\alpha)}_{t}(V) \underset{\alpha}{\longrightarrow} T(t)}$
    strongly.
\end{lemm}

    \begin{proof}
        Let $\lambda \in (1,\:\infty)$ be arbitrary.
        Let $\alpha_{\lambda},\beta_{\lambda},\gamma_{\lambda} \in \reals$
        be defined as in \Cref{prop:hille-yosida-cogenerator:sig:article-free-raj-dahya}.
        Since
            ${
                z
                \mapsto
                t(
                    \alpha_{\lambda}
                    -
                    \beta_{\lambda}
                    (\gamma_{\lambda} - z)^{-1}
                )
            }$
        is holomorphic on $\complex \without \{\gamma_{\lambda}\}$
        and its power series centred on $0$
        has convergence radius $\gamma_{\lambda}$,
        it follows that composition with the entire function ${z \mapsto \exp(z)}$
        has the same properties.
        That is,
            ${
                f^{(\lambda)}_{t} : \complex \without \{\gamma_{\lambda}\} \to \complex
            }$
        defined by
            $
                f^{(\lambda)}_{t}(z)
                = e^{
                    t(
                        \alpha_{\lambda}
                        -
                        \beta_{\lambda}
                        (\gamma_{\lambda} - z)^{-1}
                    )
                }
            $
        is holomorphic and its power series centred on $0$
        has convergence radius $\gamma_{\lambda}$.
        Thus there exists
            $
                \{c^{(\lambda, n)}_{t}\}_{n \in \naturalsZero}
                \subseteq
                \complex
            $
        such that
            $
                \sum_{n=0}^{\infty} c^{(\lambda,n)}_{t}z^{n}
                = f^{(\lambda)}_{t}(z)
            $,
        whereby the convergence of the series on the left
        is absolute for each $z \in \complex$ with $\abs{z} < \gamma_{\lambda}$.
        Since $\gamma_{\lambda} = \tfrac{\lambda + 1}{\lambda - 1} >  1$,
        it follows that

        \begin{restoremargins}
        \begin{equation}
        \label{eq:0:\beweislabel}
            C_{t}^{(\lambda,n)}
            \coloneqq
            \sum_{k=n+1}^{\infty}\abs{c^{(\lambda,k)}_{t}}
            \underset{n}{\longrightarrow}
            0
        \end{equation}
        \end{restoremargins}

        \continueparagraph
        for ${\naturals \ni n \longrightarrow \infty}$.
        Note that since
            $f^{(\lambda)}_{t}(\reals) \subseteq \reals$
        (since $\alpha_{\lambda},\beta_{\lambda},\gamma_{\lambda} \in \reals$)
        one has
        $
            c^{(\lambda,n)}_{t}
            =
                \tfrac{1}{n!}
                (
                    (\tfrac{\dee}{\dee x})^{n}
                    f^{(\lambda)}_{t}
                )\restr{x=0}
            \in \reals
        $
        for each $n\in\naturalsZero$.

        Define the index set
        $\Lambda \coloneqq \{
            (\lambda,n,\eps) \in (1,\:\infty) \times \naturalsZero \times (0,\:\infty)
            \mid
            C_{t}^{(\lambda,n)} < \eps
        \}$
        with ordering
        $(\lambda',n',\eps') \succeq (\lambda,n,\eps)$
        if and only if
        $\lambda' \geq \lambda$ and $\eps' \leq \eps$.
        Due to \eqcref{eq:0:\beweislabel}
        it is straightforward to see that $\Lambda$ is directed.
        Finally, we construct

        \begin{displaymath}
            p^{(\lambda,n,\eps)}_{t}
            \coloneqq
            \sum_{k=0}^{n}c^{(\lambda,k)}_{t}z^{k},
        \end{displaymath}

        \continueparagraph
        which are polynomials with real-valued coefficients.

        We now apply these constructions to concrete operators.
        Let $\HilbertRaum$, $T$, $V$ be as in the claim.
        Let further $\xi\in\HilbertRaum$ and $\delta > 0$ be arbitrary.
        By the convergence in \Cref{prop:hille-yosida-cogenerator:sig:article-free-raj-dahya}
        there exists $\lambda_{0} \in (1,\:\infty)$
        with
            $
                \norm{
                    (
                        T^{(\lambda)}(t)
                        -
                        T(t)
                    )
                    \xi
                }
                < \tfrac{\delta}{2}
            $
        for all $\lambda \in (1,\:\infty)$
        with $\lambda \geq \lambda_{0}$.
        Moreover, by the convergence in \eqcref{eq:0:\beweislabel}
        there exists $n_{0} \in \naturalsZero$
        such that
            $
                C_{t}^{(\lambda_{0},n_{0})}
                < \tfrac{\delta}{2\norm{\xi} + 1}
                \eqqcolon \eps_{0}
            $.
        By construction $(\lambda_{0},n_{0},\eps_{0}) \in \Lambda$.

        Consider now $(\lambda,n,\eps) \in \Lambda$
        with $(\lambda,n,\eps) \succeq (\lambda_{0},n_{0},\eps_{0})$.
        So
            $\lambda \geq \lambda_{0}$
            and $\eps \leq \eps_{0}$.
        Since $f^{(\lambda)}_{t}$ is holomorphic on an open neighbourhood of
            $\{z \in \complex \mid \abs{z} \leq 1\}$
        which contains $\sigma(V)$,
        one has

            \begin{shorteqnarray}
                \norm{
                    p^{(\lambda,n,\eps)}_{t}(V)
                    -
                    T^{(\lambda)}(t)
                }
                &= &\normLarge{
                        p^{(\lambda,n,\eps)}_{t}(V)
                        -
                        e^{t(\alpha_{\lambda}\cdot\onematrix - \beta_{\lambda}(\gamma_{\lambda}\cdot\onematrix - V)^{-1})}
                    }\\
                &= &\norm{
                        p^{(\lambda,n,\eps)}_{t}(V)
                        -
                        f^{(\lambda)}_{t}(V)
                    }\\
                &= &\normLarge{
                        \sum_{k=n+1}^{\infty}
                            c^{(\lambda,k)}_{t}
                            V^{k}
                    }\\
                &\leq &\sum_{k=n+1}^{\infty}
                        \abs{c^{(\lambda,k)}_{t}}
                        \underbrace{
                            \norm{V}^{k}
                        }_{\leq 1}
                \leq C_{t}^{(\lambda,n)}
                < \eps,
            \end{shorteqnarray}

        \continueparagraph
        whereby the last inequality holds by virtue
        of $(\lambda,n,\eps) \in \Lambda$.
        It follows that

            \begin{displaymath}
                \normLarge{
                    \Big(
                        p^{(\lambda,n,\eps)}_{t}(V)
                        -
                        T(t)
                    \Big)
                    \xi
                }
                \leq \underbrace{
                        \norm{
                            p^{(\lambda,n,\eps)}_{t}(V)
                            -
                            T^{(\lambda)}(t)
                        }
                    }_{
                        <\eps
                        \leq\eps_{0}
                        =\tfrac{\delta}{2\norm{\xi} + 1}
                    }
                    \norm{\xi}
                    +
                    \underbrace{
                        \norm{
                            (
                                T^{(\lambda)}(t)
                                -
                                T(t)
                            )
                            \xi
                        }
                    }_{< \tfrac{\delta}{2},~\text{since $\lambda \geq \lambda_{0}$}}
                < \delta.
            \end{displaymath}

        Thus for all $\xi\in\HilbertRaum$ and $\delta > 0$,
        there exists $\alpha_{0} \in \Lambda$
        such that $
            \norm{
                (
                    p^{(\alpha)}_{t}(V)
                    -
                    T(t)
                )
                \xi
            }
            < \delta
        $
        for all $\alpha \in \Lambda$
        with $\alpha \succeq \alpha_{0}$.
        The claimed convergence thus holds.
    \end{proof}



\subsection[Proof of the Ist main result]{Proof of the \First main result}
\label{sec:result-concrete:proof-1st:sig:article-free-raj-dahya}

\firstparagraph
We are now equipped to provide the first full proof of
the \First free dilation theorem.

\def\beweislabel{thm:free-dilation:1st:sig:article-free-raj-dahya}
\begin{proof}[of \Cref{thm:free-dilation:1st:sig:article-free-raj-dahya}]
    Let $I$ be a non-empty index set
    and $\{T_{i}\}_{i \in I}$ be a family of contractive $\Cnought$\=/semigroups on a Hilbert space $\HilbertRaum$.
    Let $V_{i}$ be the co-generator of $T_{i}$ for each $i \in I$.
    These are contractions with $1 \notin \sigma_{p}(V_{i})$ for each $i \in I$.
    Let $(H_{1},r_{1},\{\NagyUtr_{V_{i}}\}_{i \in I})$
    be the discrete-time free unitary dilation
    of $(\HilbertRaum,\{V_{i}\}_{i \in I})$
    as per the proof of \Cref{thm:free-dilation-discrete-time:sig:article-free-raj-dahya}
    in \S{}\ref{sec:result-concrete:discrete-dilation:sig:article-free-raj-dahya}.
    Let $i \in I$.
    By \Cref{prop:eigenvalues-discrete-dilation:sig:article-free-raj-dahya}
    we have $1 \notin \sigma_{p}(\NagyUtr_{V_{i}})$
    whence by the theory of co-generators,
        $\NagyUtr_{V_{i}}$
    is the co-generator
    of a (unique) unitary $\Cnought$\=/semigroup $U_{i}$,
    which \withoutlog we may extend to a
    strongly continuous unitary representation
    of $\reals$ on $H_{1}$.

    Let $N \in \naturals$,
        $(i_{k})_{k=1}^{N} \subseteq I$
    and $(t_{k})_{k=1}^{N} \subseteq \realsNonNeg$
    be arbitrary.
    Consider the nets of polynomials
    constructed in \Cref{lemm:poly-approx-cogenerator:sig:article-free-raj-dahya}.
    Since $(H_{1},r_{1},\{\NagyUtr_{V_{i}}\}_{i \in I})$
    is a free dilation of $(\HilbertRaum,\{V_{i}\}_{i \in I})$,
    we derive from \Cref{thm:free-dilation-discrete-time:sig:article-free-raj-dahya}
    that

        \begin{restoremargins}
        \begin{equation}
        \label{eq:poly:\beweislabel}
            \prod_{k=1}^{N}
                p_{t_{k}}^{(\alpha_{k})}(V_{i_{k}})
            = r_{1}^{\ast}
                \:\Big(
                \prod_{k=1}^{N}
                    p_{t_{k}}^{(\alpha_{k})}(\NagyUtr_{V_{i_{k}}})
                \Big)
                \:r_{1}
        \end{equation}
        \end{restoremargins}

    \continueparagraph
    for all $(\alpha_{k})_{k=1}^{N} \subseteq \Lambda$.
    Since by \Cref{lemm:poly-approx-cogenerator:sig:article-free-raj-dahya}

        \begin{shorteqnarray}
            p_{t_{k}}^{(\alpha_{k})}(V_{i_{k}})
                &\underset{\alpha_{k}}{\overset{\tinytopSOT}{\longrightarrow}}
                    &T_{i_{k}}(t_{k}),
            ~\text{and}\\
            p_{t_{k}}^{(\alpha_{k})}(\NagyUtr_{V_{i_{k}}})
                &\underset{\alpha_{k}}{\overset{\tinytopSOT}{\longrightarrow}}
                    &U_{i_{k}}(t_{k})
        \end{shorteqnarray}

    \continueparagraph
    for each $k \in \{1,2,\ldots,N\}$,
    taking limits of each factor in the products in \eqcref{eq:poly:\beweislabel}
    one-by-one yields
        \eqcref{eq:free-dilation:cts-time:sig:article-free-raj-dahya}.
    Hence $(H_{1},r_{1},\{U_{i}\}_{i \in I})$
    is a continuous-time
    free unitary dilation
    of $(\HilbertRaum,\{T_{i}\}_{i \in I})$.
\end{proof}

\begin{rem}[Computability of constructions]
\makelabel{rem:computability-of-dilation:sig:article-free-raj-dahya}
    Using the above setup,
        let $A_{i}$ and $V_{i}$
        denote again the generator \resp co-generator
        of $T_{i}$
    for each $i \in I$.
    Suppose that each $A_{i}$ is bounded
    (\exempli if $\dim(\HilbertRaum) < \aleph_{0}$).
    Then
        $
            W_{i}
            \coloneqq (\onematrix - V_{i})^{-1}
            = \frac{1}{2}(\onematrix - A_{i})
        $
    is bounded and can be directly computed.
    Considering the construction in
    \eqcref{eq:schaeffer-nagy-dilation:isometry:sig:article-free-raj-dahya},
    set

        \begin{shorteqnarray}
            D_{i}
                &\coloneqq
                    &(\onematrix - V_{i}) \otimes \ElementaryMatrix{0}{0}
                    + \sum_{n=1}^{\infty}\ElementaryMatrix{n}{n},\\
            \tilde{D}_{i}
                &\coloneqq
                    &D_{i}^{-1}
                = W_{i} \otimes \ElementaryMatrix{0}{0}
                    + \sum_{n=1}^{\infty}\ElementaryMatrix{n}{n},\\
            L_{i}
                &\coloneqq
                    &\HalmosOp_{V_{i}} \otimes \ElementaryMatrix{1}{0}
                    + \onematrix \otimes \sum_{n=1}^{\infty}\ElementaryMatrix{n+1}{n},
            ~\text{and}\\
            \tilde{L}_{i}
                &\coloneqq
                    &\tilde{D}_{i}L_{i}\tilde{D}_{i}
                = \HalmosOp_{V_{i}}\:W_{i} \otimes \ElementaryMatrix{1}{0}
                    + \onematrix \otimes \sum_{n=1}^{\infty}\ElementaryMatrix{n+1}{n}
        \end{shorteqnarray}

    \continueparagraph
    which are bounded operators on $H_{0} = \HilbertRaum \otimes \ell^{2}(\naturalsZero)$,
    and observe that
        $
            \onematrix - \NagyIso_{V_{i}}
            = D_{i} - L_{i}
        $
    and
        $
            (\onematrix - \NagyIso_{V_{i}})^{-1}
            = \tilde{D}_{i} + \tilde{L}_{i}
        $,
    \idest $\onematrix - \NagyIso_{V_{i}}$ has bounded inverse.
    Setting $P_{i} \coloneqq \NagyIso_{V_{i}}\NagyIso_{V_{i}}^{\ast}$,
    one obtains

        \begin{shorteqnarray}
            \onematrix - \NagyUtr_{V_{i}}
                &\eqcrefoverset{eq:schaeffer-nagy-dilation:unitary:sig:article-free-raj-dahya}{=}
                    &\begin{matrix}{cc}
                        \onematrix - \NagyIso_{V_{i}} &P_{i}\\
                        \zeromatrix &(\onematrix - \NagyIso_{V_{i}})^{\ast}\\
                    \end{matrix}
            ~\text{and}\\
            (\onematrix - \NagyUtr_{V_{i}})^{-1}
                &= &\begin{matrix}{cc}
                        (\onematrix - \NagyIso_{V_{i}})^{-1}
                            &-(\onematrix - \NagyIso_{V_{i}})^{-1}
                            P_{i}
                            ((\onematrix - \NagyIso_{V_{i}})^{-1})^{\ast}\\
                        \zeromatrix
                            &((\onematrix - \NagyIso_{V_{i}})^{-1})^{\ast}\\
                    \end{matrix}\\
                &= &\begin{matrix}{cc}
                        \tilde{D}_{i} + \tilde{L}_{i}
                            &-(\tilde{D}_{i} + \tilde{L}_{i})
                            P_{i}
                            (\tilde{D}_{i} + \tilde{L}_{i})^{\ast}\\
                        \zeromatrix
                            &(\tilde{D}_{i} + \tilde{L}_{i})^{\ast}\\
                    \end{matrix},
        \end{shorteqnarray}

    \continueparagraph
    \idest $\onematrix - \NagyUtr_{V_{i}}$ has bounded inverse on
       $
            H_{1}
            = H_{0} \otimes \complex^{2}
            = \HilbertRaum \otimes \ell^{2}(\naturalsZero) \otimes \complex^{2}
        $.
    It follows that the free unitary dilation
        $(H_{1},r_{1},\{U_{i}\}_{i \in I})$
    of $(\HilbertRaum,\{T_{i}\}_{i \in I})$
    determined in the above proof
    consists of strongly continuous unitary representations
    whose generators, $\{B_{i}\}_{i \in I}$,
    are bounded operators
    which can be explicitly computed via

        \begin{displaymath}
            B_{i}
                = \onematrix - 2\opResolvent{\NagyUtr_{V_{i}}}{1}
                = \begin{matrix}{cc}
                    \onematrix - 2(\tilde{D}_{i} + \tilde{L}_{i})
                        &2(\tilde{D}_{i} + \tilde{L}_{i})
                        P_{i}
                        (\tilde{D}_{i} + \tilde{L}_{i})^{\ast}\\
                    \zeromatrix
                        &\onematrix - 2(\tilde{D}_{i} + \tilde{L}_{i})^{\ast}\\
                \end{matrix}
        \end{displaymath}

    \continueparagraph
    for each $i \in I$.
\end{rem}

\begin{rem}
    Switching between the continuous- and discrete-time
    setting via the means of co-generators
    is an approach due to S\l{}oci\'{n}ski.
    In his proof of the dilation
    of commuting $\Cnought$\=/semigroups,
    use is made of parameterised analytic functions
        ${e_{t} : \complex \setminus \{1\} \to \complex}$
    defined by $e_{t}(z) = e^{t\tfrac{z+1}{z-1}}$.
    Semigroups can be recovered from their co-generators
    using the $e_{t}$ functions,
    however this involves high-powered machinery
    and functional calculi,
    which we have avoided here entirely.
    We refer the reader to \cite{%
        Slocinski1974,%
        Slocinski1982%
    }
    as well as \cite[III.8.1--2]{Nagy1970}
    for the details.
\end{rem}

\begin{rem}
    In discrete-time we considered primarily
    the dilation results of Sz.-Nagy and Bo\.{z}ejko
    and studied the possibility of continuous-time counterparts.
    In addition to these two,
    there are other unitary dilation results
    for special non-commutating families,
    \exempli
    in \cite{AtkinsonRamsey2017Article}
    dilations of freely independent contractions
    to freely independent unitaries
    are considered
    within the context of free probability theory.
    It would be of interest to know if such notions
    can be lifted to the continuous-time context.
\end{rem}



\subsection[Trotter--Kato theorem for co-generators]{Trotter--Kato theorem for co-generators}
\label{sec:result-concrete:trotter-kato:sig:article-free-raj-dahya}

\firstparagraph
The above proof of free dilations
suggests that the \Second main result can be achieved,
provided there is a continuous dependency
of semigroups on their co-generators.%
\footnote{%
    We refer the reader to the discussion at the start of
        \S{}\ref{sec:result-concrete:cogenerators:sig:article-free-raj-dahya}
    for basic facts about co-generators.
}
To this end the Trotter--Kato theorem comes to our aid.
This fundamental result in semigroup theory
consists of two parts:
a list of equivalent properties
and strictly sufficient conditions guaranteeing one
(and thus all) of the equivalent properties.
For our purposes, we shall only need the latter.
Note however, that the presentation
(see \exempli
    \cite[Theorem~I.7.3]{Goldstein1985semigroups},
    \cite[Theorem~III.4.8]{EngelNagel2000semigroupTextBook}%
)
of the Trotter--Kato theorem
is typically formulated in terms of sequences.
As such and for the sake of completeness,
we present and prove the equivalences here,
reformulated for our purposes.

\begin{thm}[cf. Trotter--Kato, 1958/9]
\makelabel{thm:trotter-kato-cogen:sig:article-free-raj-dahya}
    Let $\Omega$ be a compact topological space.
    Let $\{T_{\omega}\}_{\omega \in \Omega}$
    be a family of contractive $\Cnought$\=/semigroups
    on a Banach space $\BanachRaum$
    and let $A_{\omega}$, $V_{\omega}$
    denote the generator \resp co-generator
    of $T_{\omega}$ for each $\omega \in \Omega$.
    Then the following are equivalent:

    \begin{kompaktenum}{\bfseries (a)}[\rtab]
    \item\punktlabel{T}
        $\{T_{\omega}\}_{\omega \in \Omega}$
            is $\toplocSOT$\=/continuous in the index set,%
        \footnoteref{ft:continuity-defn:sec:result-concrete:trotter-kato:sig:article-free-raj-dahya}

    \item\punktlabel{R}
        ${\Omega \ni \omega \mapsto \opResolvent{A_{\omega}}{\lambda} \in \BoundedOps{\BanachRaum}}$
        is uniformly $\topSOT$\=/continuous
        for $\lambda$ on compact subsets
        of the open right half-plane $\{z \in \complex \mid \Re(z) > 0\}$;

    \item\punktlabel{R1}
        ${\Omega \ni \omega \mapsto \opResolvent{A_{\omega}}{1} \in \BoundedOps{\BanachRaum}}$
        is $\topSOT$\=/continuous;

    \item\punktlabel{V}
        ${\Omega \ni \omega \mapsto V_{\omega} \in \OpSpaceC{\BanachRaum}}$
        is $\topSOT$\=/continuous;
    \end{kompaktenum}

    \nvraum{1.5}

\end{thm}

\footnotetext[ft:continuity-defn:sec:result-concrete:trotter-kato:sig:article-free-raj-dahya]{%
    \idest
    ${\Omega \ni \omega \mapsto T_{\omega}}$
    is $\toplocSOT$\=/continuous.
}

In order to prove this, we need the following technical means,
which are not required in the standard presentation of the result
in which the index set is discrete.

\begin{prop}
\makelabel{prop:trotter-kato:selection:lsc:sig:article-free-raj-dahya}
    Let $\Omega$ be topological space,
        $\BanachRaum$ a Banach space,
    and
        $\{W_{\omega}\}_{\omega \in \Omega} \subseteq \BoundedOps{\BanachRaum}$
        a strongly continuous family of operators each with dense image.
    Then for all
        $g \in \Cts{\Omega}{\BanachRaum}$
        and
        $r > 0$,

        \begin{shorteqnarray}
            P_{r}
            \coloneqq
            \{
                (\omega,\xi)
                \in \Omega \times \xi
                \mid
                \norm{W_{\omega}\xi - g(\omega)}
                \leq r
            \}
        \end{shorteqnarray}

    \continueparagraph
    is lower semi-continuous,%
    \footnoteref{ft:1:\beweislabel}
    and $\Proj_{\Omega}P_{r} = \Omega$.
\end{prop}

    \footnotetext[ft:1:\beweislabel]{%
        For topological spaces, $X, Y$,
        a subset $S \subseteq X \times Y$,
        or equivalent a multi-valued
        map ${S : X \to \Pot(Y)}$,
        is called \highlightTerm{lower semi-continuous}
        if $\{x \in X \mid S(x) \cap V \neq \emptyset\}$
        is open for all open $V \subseteq Y$.
    }

    \begin{proof}
        That $\Proj_{\Omega}P_{r} = \Omega$ follows from the assumption
        that the $W_{\omega}$ have dense images.
        Let $V \subseteq \BanachRaum$ be an arbitrary non-empty open set
        and fix
            $(\omega_{0}, \xi_{0}) \in P_{r} \cap \Omega \times V$.
        We show that there is an open neighbourhood $U \subseteq \Omega$
        of $\omega_{0}$,
        such that for each $\omega \in U$
        a vector $\xi \in V$ exists with $(\omega,\xi) \in P_{r}$.
        Since $W_{\omega_{0}}$ has dense image,
        one can find $\xi_{1} \in \BanachRaum$
        such that
            $\norm{W_{\omega_{0}}\xi_{1} - g(\omega_{0})} < r$
        strictly.
        Now since $V$ is locally convex,
        one can find $t \in (0,\:1)$,
        such that $\xi_{t} \coloneqq (1-t)\xi_{0} + t\xi_{1} \in V$.
        Moreover,
        since $(\omega_{0},\xi_{0}) \in P_{r}$
        one has
            $
                \norm{
                    W_{\omega_{0}}\xi_{t} - g(\omega_{0})
                }
                \leq (1-t)\norm{W_{\omega_{0}}\xi_{0} - g(\omega_{0})}
                    + t\norm{W_{\omega_{0}}\xi_{1} - g(\omega_{0})}
                < r
            $
        strictly.
        Since
            ${f:\Omega \ni \omega \mapsto W_{\omega}\xi_{t} - g(\omega) \in \BanachRaum}$
        is norm-continuous,
        it follows that
            $U \coloneqq f^{-1}(\{\eta \in \BanachRaum \mid \norm{\eta} < r\})$
        is an open subset containing $\omega_{0}$.
        For each $\omega \in U$ one has
            $
                \norm{
                    W_{\omega}\xi_{t} - g(\omega)
                }
                = \norm{f(\omega)}
                < r
            $
        and thus
            $(\omega,\xi_{t}) \in P_{r} \cap \Omega \times V$.
    \end{proof}

\begin{prop}
\makelabel{prop:trotter-kato:selection:selection:sig:article-free-raj-dahya}
    Let $\Omega$ be a compact topological space,
        $\BanachRaum$ a Banach space,
    and
        $\{W_{\omega}\}_{\omega \in \Omega} \subseteq \BoundedOps{\BanachRaum}$
        a strongly continuous family of operators each with dense image.
    Then for all
        $g \in \Cts{\Omega}{\BanachRaum}$
        and
        $\eps > 0$,
    there exists $f \in \Cts{\Omega}{\BanachRaum}$
    such that

        \begin{restoremargins}
        \begin{equation}
        \label{eq:cogenerator-approx-michael-selection:sig:article-free-raj-dahya}
            \sup_{\omega \in \Omega}
                \norm{W_{\omega}f(\omega) - g(\omega)}
            < \eps
        \end{equation}
        \end{restoremargins}

    \continueparagraph
    holds.
\end{prop}

    \begin{proof}
        Consider the lower semi-continuous set
            $P_{\eps/2}$
        in \Cref{prop:trotter-kato:selection:lsc:sig:article-free-raj-dahya}.
        Since in addition
            $\{\xi \in \BanachRaum \mid (\omega,\xi) \in P_{\eps/2}\}$
        is a non-empty closed convex subset of $\BanachRaum$
        for each $\omega \in \Omega$,
        Michael's selection theorem
        (%
            see
            \cite[Theorem~3.2'']{Michael1956ArticleI},
            \cite[Theorem~1]{Michael1956ArticleII}%
        )
        can be applied
        to obtain a continuous function
        $f \in \Cts{\Omega}{\BanachRaum}$
        such that $\Gph{f} \subseteq P_{\eps/2}$.
        By construction of $P_{\eps/2}$, it follows
        that $
            \sup_{\omega \in \Omega}
                \norm{W_{\omega}f(\omega) - g(\omega)}
            \leq \eps/2 < \eps
        $.
    \end{proof}

We can now proceed to prove \Cref{thm:trotter-kato-cogen:sig:article-free-raj-dahya}.
The following proof is inspired by the approach presented
in \cite[Theorem~III.4.8]{EngelNagel2000semigroupTextBook}.

    \def\beweislabel{thm:trotter-kato-cogen:sig:article-free-raj-dahya}
    \begin{proof}[of \Cref{\beweislabel}]
        The implication
            \punktcref{T}{}\ensuremath{\implies}{}\punktcref{R}
        is a straightforward consequence of the
        integral representation of the resolvent

            \begin{displaymath}
                \opResolvent{A_{\omega}}{\lambda}
                    = \int_{t=0}^{\infty}
                            e^{- \lambda t}
                            T_{\omega}(t)
                        \:\dee t
            \end{displaymath}

        \continueparagraph
        for $\lambda \in \complex$ with $\Re \lambda > 0$,
        where the integral is computed strongly via Bochner integrals
        (see \exempli \cite[Theorem~II.1.10]{EngelNagel2000semigroupTextBook}).
        The implication
            \punktcref{R}{}\ensuremath{\implies}{}\punktcref{R1}
        is obvious.
        And recalling that
            $
                V_{\omega} = \onematrix - 2\opResolvent{A_{\omega}}{1}
            $
            for $\omega \in \Omega$,
        the equivalence
        \punktcref{R1}{}\ensuremath{\iff}{}\punktcref{V}
        is clear.
        In the remainder of the proof
        we establish the final implication
        \punktcref{V}{}\ensuremath{\implies}{}\punktcref{T}.

        \paragraph{Diagonal construction:}
        Consider the Banach space
            $\BanachRaum_{\Omega} \coloneqq \Cts{\Omega}{\BanachRaum}$
        of bounded continuous functions on $\Omega$ with values in $\BanachRaum$,
        endowed with the (complete) norm

            \begin{displaymath}
                \norm{f}_{\Omega}
                \coloneqq
                \sup_{\omega \in \Omega}\norm{f(\omega)}
            \end{displaymath}

        \continueparagraph
        for $f \in \Cts{\Omega}{\BanachRaum}$.%
        \footnote{%
            Note that no conditions need be imposed
            upon $\Omega$ or the Banach space $\BanachRaum$
            in oder to obtain the Cauchy-completeness of
            the normed vector space
                $(\Cts{\Omega}{\BanachRaum},\norm{\cdot}_{\Omega})$
            (see \exempli
                \cite[Theorem~7.9~a)]{Kelley1955Book}%
            ).
        }

        Under assumption \punktcref{V},
        \idest that
            ${\Omega \ni \omega \mapsto V_{\omega} \in \OpSpaceC{\BanachRaum}}$
        is $\topSOT$\=/continuous,
        it is routine to verify that
            ${
                V_{\Omega} : \Cts{\Omega}{\BanachRaum} \to \FnRm{\Omega}{\BanachRaum}
            }$
        defined by

            \begin{displaymath}
                (V_{\Omega}f)(\omega)
                    \coloneqq
                        V_{\omega}\:f(\omega)
            \end{displaymath}

        \continueparagraph
        for $f \in \Cts{\Omega}{\BanachRaum}$,
        $\omega\in\Omega$
        is linear,
        maps $\Cts{\Omega}{\BanachRaum}$ to $\Cts{\Omega}{\BanachRaum}$,
        and satisfies
            $
                \norm{V_{\Omega}}
                = \sup_{\omega \in \Omega}\norm{V_{\omega}}
                \leq 1
            $,
        \idest $V_{\Omega} \in \OpSpaceC{\BanachRaum_{\Omega}}$.
        We refer to this as a \emph{diagonal construction}.
        We now claim that $V_{\Omega}$
        is the co-generator of a (necessarily unique)
        contractive $\Cnought$\=/semigroup.
        By \cite[Theorem~2.2]{EisnerZwart2008cogenerator},
        this holds if and only if

        \begin{enumerate}[label={\upshape\bfseries {\roman*})}]
            \item\label{cogen:1:\beweislabel}
                $\onematrix - V_{\Omega}$ is injective,
            \item\label{cogen:2:\beweislabel}
                $
                    \norm{
                        (\onematrix - V_{\Omega})
                        \opResolvent{V_{\Omega}}{\mu}
                    }
                    \leq \frac{2}{\mu + 1}
                $
                for all $\mu > 1$, and
            \item\label{cogen:3:\beweislabel}
                $\onematrix - V_{\Omega}$ has dense range
        \end{enumerate}

        \continueparagraph
        all hold.
        Property \ref{cogen:1:\beweislabel} for $V_{\Omega}$
        easily follows from \ref{cogen:1:\beweislabel} holding for all $V_{\omega}$.
        Towards \ref{cogen:2:\beweislabel},
        let $\mu > 1$ and $f \in \BanachRaum_{\Omega}$ be arbitrary.
        Since $V_{\Omega}$ and each $V_{\omega}$ are contractions,
        one has
            $\norm{\mu^{-1}V_{\Omega}} < 1$
            and
            $\norm{\mu^{-1}V_{\omega}} < 1$
        for each $\omega \in \Omega$.
        One thus computes

            \begin{shorteqnarray}
                \norm{
                    (\onematrix - V_{\Omega})
                    \opResolvent{V_{\Omega}}{\mu}
                    f
                }_{\Omega}
                &= &\norm{
                    \sum_{k=0}^{\infty}
                        \mu^{-(k+1)}
                        (\onematrix - V_{\Omega})
                        V_{\Omega}^{k}
                        f
                    }_{\Omega}
                    \\
                &= &\sup_{\omega \in \Omega}
                    \norm{
                        \sum_{k=0}^{\infty}
                        \Big(
                            \mu^{-(k+1)}
                            (\onematrix - V_{\Omega})
                            V_{\Omega}^{k}
                            f
                        \Big)(\omega)
                    }
                    \\
                &= &\sup_{\omega \in \Omega}
                    \norm{
                        \sum_{k=0}^{\infty}
                            \mu^{-(k+1)}
                            (\onematrix - V_{\omega})
                            V_{\omega}^{k}
                            f(\omega)
                        }
                    \\
                &= &\sup_{\omega \in \Omega}
                    \norm{
                        (\onematrix - V_{\omega})
                        \opResolvent{V_{\omega}}{\mu}
                        f(\omega)
                    }
                    \\
                &\leq &\sup_{\omega \in \Omega}
                    \norm{
                        (\onematrix - V_{\omega})
                        \opResolvent{V_{\omega}}{\mu}
                    }
                    \sup_{\omega \in \Omega}
                        \norm{f(\omega)}
                    \\
                &\leq &\frac{2}{\mu + 1}
                    \norm{f}_{\Omega},
                \\
            \end{shorteqnarray}

        \continueparagraph
        whereby the final inequality follows
        from \ref{cogen:2:\beweislabel}
        for each $V_{\omega}$.
        So $
            \norm{
                (\onematrix - V_{\Omega})
                \opResolvent{V_{\Omega}}{\mu}
            }_{\Omega}
            \leq \frac{2}{\mu + 1}
        $.
        Towards \ref{cogen:3:\beweislabel},
        let $g \in \Cts{\Omega}{\BanachRaum}$ and $\eps > 0$ be arbitrary.
        Since \ref{cogen:3:\beweislabel} holds for each $V_{\omega}$,
        one has that
            $W_{\omega} \coloneqq \onematrix - V_{\omega}$
        has dense range for each $\omega \in \Omega$.
        Since ${\Omega \ni \omega \mapsto V_{\omega} \in \BoundedOps{\BanachRaum}}$
        is strongly continuous,
        we may apply \Cref{prop:trotter-kato:selection:selection:sig:article-free-raj-dahya},
        and obtain a continuous function
            $f \in \Cts{\Omega}{\BanachRaum}$,
        such that
            $
                \norm{(\onematrix - V_{\Omega})f - g}_{\Omega}
                = \sup_{\omega \in \Omega}
                    \norm{(\onematrix - V_{\omega})f(\omega) - g(\omega)}
                = \sup_{\omega \in \Omega}
                    \norm{W_{\omega}f(\omega) - g(\omega)}
                \eqcrefoverset{eq:cogenerator-approx-michael-selection:sig:article-free-raj-dahya}{<} \eps
            $.

        So $V_{\Omega}$ satisfies \ref{cogen:1:\beweislabel}, \ref{cogen:2:\beweislabel}, and \ref{cogen:3:\beweislabel},
        and is thus the co-generator of a contractive $\Cnought$\=/semigroup
            $T_{\Omega}$
        on $\BanachRaum_{\Omega}$.
        We now establish that $T_{\Omega}$
        is a diagonal construction.
        Fix an arbitrary
            $\omega \in \Omega$
        and let
            ${
                \pi_{\omega}
                : \BanachRaum_{\Omega} \to \BanachRaum
            }$
        denote the surjective contraction
        defined by
            $\pi_{\omega}f = f(\omega)$
        for $f\in\BanachRaum_{\Omega}$.
        By construction we have
            $
                \pi_{\omega}\:V_{\Omega}f
                =(V_{\Omega}f)(\omega)
                = V_{\omega}f(\omega)
                = V_{\omega}\:\pi_{\omega}f
            $
        for all $f \in \BanachRaum_{\Omega}$,
        \idest
            $
                \pi_{\omega}\:V_{\Omega}
                = V_{\omega}\:\pi_{\omega}
            $.
        By induction it follows that
            $
                \pi_{\omega}\:V_{\Omega}^{n}
                = V_{\omega}^{n}\:\pi_{\omega}
            $
        for $n \in \naturalsZero$
        and thus
            $
                \pi_{\omega}\:p(V_{\Omega})
                = p(V_{\omega})\:\pi_{\omega}
            $
        for all polynomials $p \in \complex[X]$.
        Let $t \in \realsNonNeg$ be arbitrary.
        By \Cref{lemm:poly-approx-cogenerator:sig:article-free-raj-dahya}
        there exists a net
            $(p^{(\alpha)}_{t})_{\alpha \in \Lambda}$
        of real-valued polynomials
        such that
            ${
                p^{(\alpha)}_{t}(V_{\omega})
                \underset{\alpha}{\overset{\tinytopSOT}{\longrightarrow}}
                T_{\omega}(t)
            }$
        and
            ${
                p^{(\alpha)}_{t}(V_{\Omega})
                \underset{\alpha}{\overset{\tinytopSOT}{\longrightarrow}}
                T_{\Omega}(t)
            }$.
        Thus

            \begin{restoremargins}
            \begin{equation}
            \label{eq:commutation-sg-proj:\beweislabel}
            \everymath={\displaystyle}
            \begin{array}[m]{rcl}
                \pi_{\omega} \: T_{\Omega}(t)
                    &= &\pi_{\omega} \: \lim_{\alpha} p^{(\alpha)}_{t}(V_{\Omega})\\
                    &= &\lim_{\alpha} \pi_{\omega} \: p^{(\alpha)}_{t}(V_{\Omega})\\
                    &= &\lim_{\alpha} p^{(\alpha)}_{t}(V_{\omega}) \: \pi_{\omega}\\
                    &= &\Big(\lim_{\alpha} p^{(\alpha)}_{t}(V_{\omega})\Big) \: \pi_{\omega}
                    = T_{\omega}(t) \: \pi_{\omega}\\
            \end{array}
            \end{equation}
            \end{restoremargins}

        \continueparagraph
        for all $t\in\realsNonNeg$.

        \paragraph{Proof of \punktcref{V}{}\ensuremath{\implies}{}\punktcref{T}:}
        Relying on the diagonal construction
        we now demonstrate that \punktcref{T} holds.
        Fix arbitrary
            $\eps > 0$,
            $\hat{\omega} \in \Omega$,
            $\hat{\xi} \in \BanachRaum$,
            and
            compact $K \subseteq \realsNonNeg$.
        Let ${\emb : \BanachRaum \to \BanachRaum_{\Omega}}$
        be the isometric embedding
        defined by
            $(\emb\:\xi)(\omega) = \xi$
        for $\xi\in\BanachRaum$, $\omega\in\Omega$.
        In particular
            $\pi_{\omega}\:\emb = \id_{\BanachRaum}$
        for $\omega \in \Omega$.

        Observe firstly that for each $t \in \realsNonNeg$,
        since
            $
                f_{t}
                \coloneqq T_{\Omega}(t)\:\emb\:\hat{\xi}
                \in \BanachRaum_{\Omega}
                = \Cts{\Omega}{\BanachRaum}
            $,
        there exists a neighbourhood
            $W_{t} \subseteq \Omega$
        of $\hat{\omega}$
        such that

            \begin{restoremargins}
            \begin{equation}
            \label{eq:1:\beweislabel}
            \everymath={\displaystyle}
            \begin{array}[m]{rcl}
                \norm{
                    (
                        T_{\omega}(t)
                        -
                        T_{\hat{\omega}}(t)
                    )
                    \:\hat{\xi}
                }
                    &= &\norm{
                            T_{\omega}(t)
                            \:\pi_{\omega}
                            \:\emb\:\hat{\xi}
                            -
                            T_{\hat{\omega}}(t)
                            \:\pi_{\hat{\omega}}
                            \:\emb\:\hat{\xi}
                        }\\
                    &\eqcrefoverset{eq:commutation-sg-proj:\beweislabel}{=}
                        &\norm{
                            \pi_{\omega}
                            \:T_{\Omega}(t)
                            \:\emb\:\hat{\xi}
                            -
                            \pi_{\hat{\omega}}
                            \:T_{\Omega}(t)
                            \:\emb\:\hat{\xi}
                        }\\
                    &=
                        &\norm{
                            \pi_{\omega}f_{t}
                            -
                            \pi_{\hat{\omega}}f_{t}
                        }\\
                    &= &\norm{
                            f_{t}(\omega)
                            -
                            f_{t}(\hat{\omega})
                        }
                    < \eps/4
            \end{array}
            \end{equation}
            \end{restoremargins}

        \continueparagraph
        for $\omega \in W_{t}$.
        Observe secondly that since $T_{\Omega}$ is $\topSOT$\=/continuous,
            $
                \{
                    T_{\Omega}(t)\emb\:\hat{\xi}
                    \mid
                    t \in K
                \}
            $
        is a norm-compact subset of $\BanachRaum_{\Omega}$.
        As such there exists a finite subset
            $F \subseteq K$
        such that

            \begin{restoremargins}
            \begin{equation}
            \label{eq:2:\beweislabel}
                \bigcup_{\tau \in F}
                    \oBall{\eps/4}{T_{\Omega}(\tau)\emb\:\hat{\xi}}
                \supseteq
                \{
                    T_{\Omega}(t)\emb\:\hat{\xi}
                    \mid
                    t \in K
                \}
            \end{equation}
            \end{restoremargins}

        \continueparagraph
        holds.
        Consider now arbitrary
            $t \in K$.
        By \eqcref{eq:2:\beweislabel},
        there exists $\tau_{t} \in F$
        such that
            $
                \norm{
                    T_{\Omega}(t)\emb\:\hat{\xi}
                    -
                    T_{\Omega}(\tau_{t})\emb\:\hat{\xi}
                }
                < \eps/4
            $.
        So

            \begin{restoremargins}
            \begin{equation}
            \label{eq:3:\beweislabel}
            \everymath={\displaystyle}
            \begin{array}[m]{rcl}
                \norm{
                    (
                        T_{\omega}(t)
                        -
                        T_{\omega}(\tau_{t})
                    )
                    \:\hat{\xi}
                }
                    &= &\norm{
                            T_{\omega}(t)
                            \:\pi_{\omega}
                            \:\emb\:\hat{\xi}
                            -
                            T_{\omega}(\tau_{t})
                            \:\pi_{\omega}
                            \:\emb\:\hat{\xi}
                        }\\
                    &\eqcrefoverset{eq:commutation-sg-proj:\beweislabel}{=}
                        &\norm{
                            \pi_{\omega}
                            \:
                            T_{\Omega}(t)
                            \:\emb\:\hat{\xi}
                            -
                            \pi_{\omega}
                            T_{\Omega}(\tau_{t})
                            \:\emb\:\hat{\xi}
                        }\\
                    &\leq &\norm{\pi_{\omega}}
                        \norm{
                            T_{\Omega}(t)\emb\:\hat{\xi}
                            -
                            T_{\Omega}(\tau_{t})\emb\:\hat{\xi}
                        }\\
                    &\eqcrefoverset{eq:2:\beweislabel}{<}
                        &1 \cdot \tfrac{\eps}{4}
            \end{array}
            \end{equation}
            \end{restoremargins}

        \continueparagraph
        for all $\omega \in \Omega$.

        Set
            $
                \hat{W}
                \coloneqq
                \bigcap_{\tau \in F}W_{\tau}
            $,
        which is a neighbourhood of $\hat{\omega}$.
        The above two observations yield
            $
                \norm{
                    (
                        T_{\omega}(t)
                        -
                        T_{\hat{\omega}}(t)
                    )
                    \:\hat{\xi}
                }
                    \leq \norm{
                                (
                                    T_{\omega}(t)
                                    -
                                    T_{\omega}(\tau_{t})
                                )
                                \:\hat{\xi}
                            }
                            + \norm{
                                (
                                    T_{\omega}(\tau_{t})
                                    -
                                    T_{\hat{\omega}}(\tau_{t})
                                )
                                \:\hat{\xi}
                            }
                            + \norm{
                                (
                                    T_{\hat{\omega}}(\tau_{t})
                                    -
                                    T_{\hat{\omega}}(t)
                                )
                                \:\hat{\xi}
                            }
                    \textoverset{
                        \eqcref{eq:1:\beweislabel}
                        +
                        \eqcref{eq:3:\beweislabel}
                    }{<}
                        \tfrac{\eps}{4}
                        + \tfrac{\eps}{4}
                        + \tfrac{\eps}{4}
            $
        for all $\omega$ in the neighbourhood $\hat{W}$ of $\hat{\omega}$.
        Hence
            $
                \sup_{t \in K}
                    \norm{
                        (
                            T_{\omega}(t)
                            -
                            T_{\hat{\omega}}(t)
                        )
                        \:\hat{\xi}
                    }
                \leq 3\eps/4
                < \eps
            $
        for all $\omega \in \hat{W}$.
        This establishes the $\toplocSOT$\=/continuity
        of the map
        ${
            \Omega
            \ni
            \omega
            \mapsto T_{\omega}
        }$.
    \end{proof}

\begin{rem}
    The classical Trotter--Kato theorem can be recovered
    by considering the one point compactification
    $\Omega = \naturals \cup \{\infty\}$.
    This requires us to assert the strong convergence
    ${V_{n} \underset{n}{\longrightarrow} V_{\infty}}$
    for ${n \longrightarrow \infty}$,
    or equivalent statements about the generators.
\end{rem}

The above formulation of the Trotter--Kato theorem immediately implies several results,
some of which are interesting in and of themselves.

\begin{prop}
\makelabel{prop:trotter-kato-cogen:sot-star:sig:article-free-raj-dahya}
    Let $\Omega$ be a compact topological space.
    Let $\{T_{\omega}\}_{\omega \in \Omega}$
    be a family of contractive $\Cnought$\=/semigroups
    on a Hilbert space $\HilbertRaum$.%
    \footnote{%
        Note that
            $\{T_{\omega}(\cdot)^{\ast}\}_{\omega\in\Omega}$
        is also a family of contractive semigroups on $\HilbertRaum$,
        which are weakly and thus strongly continuous
        (see \exempli
            \cite[Theorem~I.5.8]{EngelNagel2000semigroupTextBook},
            \cite[Theorem~9.3.1 and Theorem~10.2.1--3]{Hillephillips1957faAndSg}%
        ).
    }
    Let further $V_{\omega}$
    denote the co-generator
    of $T_{\omega}$ for each $\omega \in \Omega$.
    Then
        $\{T_{\omega}\}_{\omega \in \Omega}$
        is $\toplocSOTstar$\=/continuous in $\Omega$
    if and only if
        ${\Omega \ni \omega \mapsto V_{\omega} \in \OpSpaceC{\HilbertRaum}}$
    is $\topSOTstar$\=/continuous.%
    \footnote{%
        \idest
        both
        ${\Omega \ni \omega \mapsto V_{\omega} \in \OpSpaceC{\HilbertRaum}}$
        and
        ${\Omega \ni \omega \mapsto V_{\omega}^{\ast} \in \OpSpaceC{\HilbertRaum}}$
        are $\topSOT$\=/continuous.
    }
\end{prop}

Let $\Omega$ be compact
and $\BanachRaum$ a Banach space.
Consider the Banach space
    $\BanachRaum_{\Omega} \coloneqq \Cts{\Omega}{\BanachRaum}$,
the isometric embedding
    ${\emb : \BanachRaum \to \Cts{\Omega}{\BanachRaum}}$,
and the surjective contractions
    ${\pi_{\omega} : \Cts{\Omega}{\BanachRaum} \to \BanachRaum}$
for each $\omega\in\Omega$
defined as above.

\begin{prop}[First diagonalisation]
\makelabel{prop:first-diagonalisation:sig:article-free-raj-dahya}
    A family
        $\{T_{\omega}\}_{\omega \in \Omega}$
    of contractive $\Cnought$\=/semigroups
    on $\BanachRaum$
    is $\toplocSOT$\=/continuous in the index set
    if and only if
    there exists a
    (necessarily unique)
    contractive $\Cnought$\=/semigroup
        $T_{\Omega}$
    on $\BanachRaum_{\Omega}$
    satisfying
        $
            \pi_{\omega}\:T_{\Omega}(t)
            = T_{\omega}(t)\:\pi_{\omega}
        $
    and thus

        \begin{restoremargins}
        \begin{equation}
        \label{eq:first-diagonalisation:dilation:sig:article-free-raj-dahya}
            \pi_{\omega}\:T_{\Omega}(t)\:\emb = T_{\omega}(t)
        \end{equation}
        \end{restoremargins}

    \continueparagraph
    for all $\omega \in \Omega$ and $t \in \realsNonNeg$.
\end{prop}

\Cref{prop:first-diagonalisation:sig:article-free-raj-dahya}
can be directly extracted from the proof of
\Cref{thm:trotter-kato-cogen:sig:article-free-raj-dahya},
noting that $\pi_{\omega}\:\emb = \onematrix$
for each $\omega \in \Omega$.
This result, in particular \eqcref{eq:first-diagonalisation:dilation:sig:article-free-raj-dahya},
states that families of contractive $\Cnought$\=/semigroups
are $\toplocSOT$\=/continuous in their index set
exactly in case they arise as certain
Banach space dilations
of a single contractive $\Cnought$\=/semigroup.
In fact by using Stroescu's dilation theorem,
this can be strengthened to a Banach space representation:

\begin{prop}[Second diagonalisation]
\makelabel{prop:second-diagonalisation:sig:article-free-raj-dahya}
    A family
        $\{T_{\omega}\}_{\omega \in \Omega}$
    of contractive $\Cnought$\=/semigroups
    on $\BanachRaum$
    is $\toplocSOT$\=/continuous in the index set
    if and only if
    there exists
    a Banach space $\tilde{\BanachRaum}$,
    an isometric embedding
        $r \in \BoundedOps{\BanachRaum}{\tilde{\BanachRaum}}$,
    strongly continuous family
        $\{j_{\omega}\}_{\omega \in \Omega} \in \BoundedOps{\tilde{\BanachRaum}}{\BanachRaum}$
    of surjective isometries,
    and an
        $\topSOT$\=/continuous representation
        $U_{\Omega}\in\Repr{\reals}{\tilde{\BanachRaum}}$
    consisting of surjective isometries on $\tilde{\BanachRaum}$,%
    \footnoteref{ft:1:\beweislabel}
    such that

        \begin{restoremargins}
        \begin{equation}
        \label{eq:second-diagonalisation:dilation:sig:article-free-raj-dahya}
            j_{\omega}\:U_{\Omega}(t)\:r = T_{\omega}(t)
        \end{equation}
        \end{restoremargins}

    \continueparagraph
    for all $\omega \in \Omega$ and $t \in \realsNonNeg$.
\end{prop}

    \footnotetext[ft:1:\beweislabel]{%
        Note that surjective isometries are
        the Banach space counterpart to
        unitary operators on Hilbert spaces.
    }

    \begin{proof}
        Towards the \usesinglequotes{only if}-direction,
        first apply \Cref{prop:first-diagonalisation:sig:article-free-raj-dahya}
        which yields
        a contractive $\Cnought$\=/semigroup
            $T_{\Omega}$
        on the Banach space $\BanachRaum_{\Omega} = \Cts{\Omega}{\BanachRaum}$,
        an embedding $\emb \in \BoundedOps{\BanachRaum}{\BanachRaum_{\Omega}}$,
        and a family
            $\{\pi_{\omega}\}_{\omega \in \Omega} \subseteq \BoundedOps{\BanachRaum_{\Omega}}{\BanachRaum}$
        of surjective contractions,
        such that \eqcref{eq:first-diagonalisation:dilation:sig:article-free-raj-dahya}
        holds.
        By the Stroescu dilation theorem \cite[Corollary~1, p.~259]{Stroescu1973ArticleBanachDilations},
        there exists a Banach space $\tilde{\BanachRaum}$,
        a strongly continuous representation
            $U_{\Omega}\in\Repr{\reals}{\tilde{\BanachRaum}}$
            consisting of surjective isometries,
        an isometric embedding
            $r_{0} \in \BoundedOps{\BanachRaum_{\Omega}}{\tilde{\BanachRaum}}$
        and a surjective isometry
            $j_{0} \in \BoundedOps{\tilde{\BanachRaum}}{\BanachRaum_{\Omega}}$,
        such that
            $j_{0}\:U_{\Omega}(t)\:r_{0} = T_{\Omega}(t)$
        for all $t \in \realsNonNeg$.
        Finally we set
            $r \coloneqq r_{0}\:\emb \in \BoundedOps{\BanachRaum}{\tilde{\BanachRaum}}$,
            which is an isometric embedding,
            and
            $j_{\omega} \coloneqq \pi_{\omega}\:j_{0} \in \BoundedOps{\tilde{\BanachRaum}}{\BanachRaum}$,
            which are surjective contractions
            for each $\omega\in\Omega$.
        For $\xi \in \tilde{\BanachRaum}$
        and all $\omega \in \Omega$
        one has
            $j_{\omega}\xi = \pi_{\omega}\:(j_{0}\xi) = f(\omega)$,
        where $f = j_{0}\xi \in \BanachRaum_{\Omega} = \Cts{\Omega}{\BanachRaum}$.
        So
            ${\Omega \ni \omega \mapsto j_{\omega}\xi \in \BanachRaum}$
        is norm-continuous for all $\xi \in \tilde{\BanachRaum}$.
        Thus $\{j_{\omega}\}_{\omega\in\Omega}$ is a strongly continuous family.
        Finally, by the Stroescu-dilation,
        one has
            $
                j_{\omega}\:U_{\Omega}(t)\:r
                = \pi_{\omega}\:j_{0}\:U_{\Omega}(t)\:r_{0}\:\emb
                = \pi_{\omega}\:T_{\Omega}(t)\:\emb
                = T_{\omega}(t)
            $
        for all $\omega \in \Omega$ and $t \in \realsNonNeg$.

        Towards the \usesinglequotes{if}-direction,
        suppose that $\{T_{\omega}\}_{\omega\in\Omega}$
        is given by such a diagonalisation.
        Let
            $K \subseteq \realsNonNeg$ be compact,
            $\omega\in\Omega$,
            $\xi \in \BanachRaum$,
            and
            $\eps > 0$.
        By the strong continuity of $U_{\Omega}$,
        there exists a finite subset $F \subseteq K$
        such that
            $
                \{U_{\Omega}(t)\:r\xi \mid t \in K\}
                \subseteq
                \bigcup_{s \in F}
                    \oBall{\eps/4}{U_{\Omega}(s)\:r\xi}
            $.
        And by the strong continuity of
            $\{j_{\omega}\}_{\omega\in\Omega}$,
        we can find an open neighbourhood
            $W \subseteq \Omega$ of $\omega$,
        such that
            $
                \sup_{\omega' \in W}
                    \norm{
                        (j_{\omega'} - j_{\omega})
                        \:U_{\Omega}(s)\:r\xi
                    }
                < \eps/4
            $
        for all $s \in F$.
        Applying these inequalities
        to the dilation \eqcref{eq:second-diagonalisation:dilation:sig:article-free-raj-dahya},
        one obtains

            \begin{shorteqnarray}
                \sup_{t \in K}
                \norm{
                    (
                        T_{\omega'}(t)
                        -
                        T_{\omega'}(t)
                    )
                    \xi
                }
                &=
                    &\sup_{t \in K}
                        \norm{
                            (
                                j_{\omega'}
                                -
                                j_{\omega}
                            )
                            U_{\Omega}(t)\:r\xi
                        }\\
                &\leq
                    &\sup_{t \in K}
                    \min_{s \in F}
                    \Big(
                        \norm{
                            (
                                j_{\omega'}
                                -
                                j_{\omega}
                            )
                            U_{\Omega}(s)\:r\xi
                        }
                        +
                        \norm{j_{\omega'} - j_{\omega}}
                        \norm{
                            U_{\Omega}(t)\:r\xi
                            -
                            U_{\Omega}(s)\:r\xi
                        }
                    \Big)
                    \\
                &\leq
                    &\max_{s \in F}
                        \norm{
                            (
                                j_{\omega'}
                                -
                                j_{\omega}
                            )
                            U_{\Omega}(s)\:r\xi
                        }
                        +
                    2
                    \sup_{t \in K}
                    \min_{s \in F}
                        \norm{
                            U_{\Omega}(t)\:r\xi
                            -
                            U_{\Omega}(s)\:r\xi
                        }
                    \\
                &\leq
                    &\frac{\eps}{4}
                    +
                    2\cdot\frac{\eps}{4}
                < \eps
            \end{shorteqnarray}

        \continueparagraph
        for $\omega' \in \Omega$.
        Thus $\{T_{\omega}\}_{\omega \in \Omega}$
        is $\toplocSOT$\=/continuous in the index set.
    \end{proof}

\begin{rem}[Properties of the embeddings]
\makelabel{rem:second-diagonalisation:properties:sig:article-free-raj-dahya}
    In light of the Banach space dilation
    in \eqcref{eq:second-diagonalisation:dilation:sig:article-free-raj-dahya}
    one necessarily has that
        $j_{\omega}\:r = \onematrix$
    for all $\omega \in \Omega$.
    It follows that
        $
            \{P_{\omega} \coloneqq r\:j_{\omega}\}_{\omega \in \Omega}
            \subseteq \BoundedOps{\tilde{\BanachRaum}}
        $
    constitutes a strongly continuous family of
    idempotent Banach space contractions.%
    \footnote{%
        Note that unlike in the Hilbert space setting,
        Banach space idempotents are not necessarily contractions.
    }
    Moreover following the construction
    in the proof of \Cref{prop:second-diagonalisation:sig:article-free-raj-dahya},
    one has
    $
        P_{\omega'}P_{\omega}
        = (r_{0}\:\emb\:\pi_{\omega'}\:j_{0})
            (r_{0}\:\emb\:\pi_{\omega}\:j_{0})
        = r_{0}\:\emb\:\pi_{\omega'}
            \:\cancel{j_{0}\:r_{0}}
            \:\emb\:\pi_{\omega}\:j_{0}
        = r_{0}\:\emb
            \:\cancel{\pi_{\omega'}\:\emb}
            \:\pi_{\omega}\:j_{0}
        = r_{0}\:\emb\:\pi_{\omega}\:j_{0}
        = r\:j_{\omega}
        = P_{\omega}
    $
    for all $\omega,\omega'\in\Omega$.
\end{rem}



\subsection[Proof of the IInd main result]{Proof of the \Second main result}
\label{sec:result-concrete:proof-2nd:sig:article-free-raj-dahya}

\firstparagraph
Using the Trotter--Kato theorem
and our explicit proof of the \First main result,
we are now able to prove the
\Second free dilation theorem:

\def\beweislabel{thm:free-dilation:2nd:sig:article-free-raj-dahya}
\begin{proof}[of \Cref{thm:free-dilation:2nd:sig:article-free-raj-dahya}]
    Let $V_{\omega}$ denote the co-generator of $T_{\omega}$
    for each $\omega \in \Omega$.
    We recall the constructions
    in \S{}\ref{sec:result-concrete:discrete-dilation:sig:article-free-raj-dahya}
    of the discrete-time free dilation:
    Let
        $H_{0} \coloneqq \HilbertRaum \otimes \ell^{2}(\naturalsZero)$
        and
        $H_{1} \coloneqq H_{0} \otimes \complex^{2}$,
    and let
        $r_{1} \in \BoundedOps{\HilbertRaum}{H_{1}}$
        be the isometry defined via
        $r_{1} \coloneqq \onematrix \otimes \BaseVector{0} \otimes \BaseVector{0}$.
    We recall the isometries \resp unitaries
    constructed in
    \eqcref{eq:schaeffer-nagy-dilation:isometry:sig:article-free-raj-dahya}
    \resp
    \eqcref{eq:schaeffer-nagy-dilation:unitary:sig:article-free-raj-dahya}:

        \begin{restoremargins}
        \begin{equation}
        \label{eq:1:\beweislabel}
        \everymath={\displaystyle}
        \begin{array}[m]{rcl}
            \NagyIso_{V_{\omega}}
                &=
                    &V_{\omega} \otimes \ElementaryMatrix{0}{0}
                    + \HalmosOp_{V_{\omega}} \otimes \ElementaryMatrix{1}{0}
                    + I \otimes \sum_{n=1}^{\infty} \ElementaryMatrix{n+1}{n}
                \in \BoundedOps{H_{0}},~\text{and}\\
            \NagyUtr_{V_{\omega}}
                &\coloneqq
                    &\NagyIso_{V_{\omega}} \otimes \ElementaryMatrix{0}{0}
                    + (\onematrix - \NagyIso_{V_{\omega}}\NagyIso_{V_{\omega}}^{\ast}) \otimes \ElementaryMatrix{0}{1}
                    + \NagyIso_{V_{\omega}}^{\ast} \otimes \ElementaryMatrix{1}{1}
                \in \BoundedOps{H_{1}},
        \end{array}
        \end{equation}
        \end{restoremargins}

    \continueparagraph
    for each $\omega \in \Omega$.

    As per the proof
    in \S{}\ref{sec:result-concrete:discrete-dilation:sig:article-free-raj-dahya}
    of \Cref{thm:free-dilation-discrete-time:sig:article-free-raj-dahya},
        $(H_{1},r_{1},\{\NagyUtr_{V_{\omega}}\}_{\omega\in\Omega})$
        is a \uline{discrete-time free unitary dilation} of
        $(\HilbertRaum,\{V_{\omega}\}_{\omega \in \Omega})$.
    And by the proof
    in \S{}\ref{sec:result-concrete:proof-1st:sig:article-free-raj-dahya}
    of \Cref{thm:free-dilation:1st:sig:article-free-raj-dahya},
        $(H_{1},r_{1},\{U_{\omega}\}_{\omega \in \Omega})$
        is a \uline{continuous-time free unitary dilation} of
        $(\HilbertRaum,\{T_{\omega}\}_{\omega \in \Omega})$,
    where
        $U_{\omega}$
    is the $\topSOT$\=/continuous unitary representation
    of $\reals$ on $H_{1}$
    with co-generator $\NagyUtr_{V_{\omega}}$.
    So to prove the claim of the theorem
    by the Trotter--Kato theorem
    (see
        \Cref{thm:trotter-kato-cogen:sig:article-free-raj-dahya}
        and
        \Cref{prop:trotter-kato-cogen:sot-star:sig:article-free-raj-dahya}
        as well as
        \Cref{rem:equivalence-of-topologies-in-unitary-case:sig:article-free-raj-dahya}%
    ),
    it suffices to verify the equivalence of the following statements:

    \begin{kompaktenum}{\bfseries (a)}[\rtab]
    \item\punktlabel{1}
        ${\Omega \ni \omega \mapsto V_{\omega} \in \OpSpaceC{\HilbertRaum}}$
        is $\topSOTstar$\=/continuous;
    \item\punktlabel{2}
        ${\Omega \ni \omega \mapsto \NagyUtr_{V_{\omega}} \in \OpSpaceU{H_{1}}}$
        is $\topSOT$\=/continuous.
    \end{kompaktenum}

    \paragraph{\punktcref{1}{}\ensuremath{\implies}{}\punktcref{2}:}
    Since
        ${\omega \mapsto V_{\omega}}$
    and
        ${\omega \mapsto V_{\omega}^{\ast}}$
    are uniformly bounded and $\topSOT$\=/continuous,
    one has that
        ${\omega \mapsto V_{\omega}V_{\omega}^{\ast}}$
    is $\topSOT$\=/continuous.
    Applying the spectral theory of self-adjoint operators,
    since
        ${f : \reals \ni t \mapsto \sqrt{1 - \max\{1,\abs{t}\}}}$
    is a bounded continuous function,
    it follows that
        $f$ is $\topSOT$\=/continuous
    on the space of self-adjoint operators
    (see \exempli \cite[Theorem~4.3.2]{Murphy1990}).
    Hence
        ${
            \omega
            \mapsto
            f(V_{\omega}V_{\omega}^{\ast})
            = \sqrt{\onematrix - V_{\omega}V_{\omega}^{\ast}}
            = \HalmosOp_{V_{\omega}}
        }$
    is $\topSOT$\=/continuous.
    By the constructions in \eqcref{eq:1:\beweislabel}
    it follows that
        ${\omega \mapsto \NagyIso_{V_{\omega}}}$
        and
        ${\omega \mapsto \NagyIso_{V_{\omega}}^{\ast}}$
    are $\topSOT$\=/continuous.
    Since these maps are uniformly bounded,
        ${\omega \mapsto \NagyIso_{V_{\omega}}\NagyIso_{V_{\omega}}^{\ast}}$
    is also $\topSOT$\=/continuous.
    Appealing again to the constructions in \eqcref{eq:1:\beweislabel},
    one readily sees that
        ${\omega \mapsto \NagyUtr_{V_{\omega}}}$
    is $\topSOT$\=/continuous.

    \paragraph{\punktcref{2}{}\ensuremath{\implies}{}\punktcref{1}:}
    Since the $\topSOT$ and $\topSOTstar$
    topologies coincide on $\OpSpaceU{H_{1}}$,
    ${\omega \mapsto \NagyUtr_{V_{\omega}}}$
    is $\topSOTstar$\=/continuous.
    Applying the discrete-time free dilation mentioned in the second paragraph,
    one has
        $V_{\omega} = r_{1}^{\ast}\:\NagyUtr_{V_{\omega}}\:r_{1}$
    and thus also
        $V_{\omega}^{\ast} = r_{1}^{\ast}\:\NagyUtr_{V_{\omega}}^{\ast}\:r_{1}$
    for $\omega\in\Omega$.
    Hence
        ${\omega \mapsto V_{\omega}}$
    and
        ${\omega \mapsto V_{\omega}^{\ast}}$
    are uniformly bounded and $\topSOT$\=/continuous.
\end{proof}




\section[Free dilations implicitly constructed]{Free dilations implicitly constructed}
\label{sec:result-abstract:sig:article-free-raj-dahya}

\firstparagraph
The na\"{i}ve approach mentioned in \Cref{rem:naive-approach:sig:article-free-raj-dahya}
can be made to work by studying the structure of dilations.
This paves the way for an abstract proof of the \First
free dilation theorem.


\subsection[Structure theorems]{Structure theorems}
\label{sec:result-abstract:theorems:sig:article-free-raj-dahya}

\firstparagraph
In this section we work with semigroups defined over topological monoids
and dilations \emph{either} to the semigroups over the same monoid
\emph{or} to an extension of the topological monoid to a topological group
(recall the definitions in \S{}\ref{sec:introduction:notation:sig:article-free-raj-dahya}).
Recall also the natural
\onetoone\=/correspondence
between ($\topSOT$\=/continuous) unitary semigroups defined over $M$
and ($\topSOT$\=/continuous) unitary representations of $G$
in the special case of
    $(G,M) = (\reals^{d},\realsNonNeg^{d})$,
    $d\in\naturals$.

\begin{lemm}[Sarason, 1965]
\makelabel{lemm:sarason:sig:article-free-raj-dahya}
    Let $M$ be a (topological) monoid,
    $\HilbertRaum,\HilbertRaum^{\prime}$ be Hilbert spaces
    and
    $r \in \BoundedOps{\HilbertRaum}{\HilbertRaum^{\prime}}$
    an isometry.
    Let ${U : M \to \BoundedOps{\HilbertRaum^{\prime}}}$
    be a(n $\topSOT$\=/continuous) semigroup over $M$ on $\HilbertRaum^{\prime}$.
    Consider the (continuous) operator-valued function
        ${
            T \coloneqq r^{\ast}\:U(\cdot)\:r
            : M \to \BoundedOps{\HilbertRaum}
        }$.
    Then $T$ satisfies the semigroup law,
    \idest constitutes a(n $\topSOT$\=/continuous) semigroup over $M$ on $\HilbertRaum$,
    if and only if there exists a decomposition

    \begin{displaymath}
        \HilbertRaum^{\prime}
        = r\HilbertRaum \oplus H_{0} \oplus H_{1}
    \end{displaymath}

    \continueparagraph
    such that the subspaces
        $r\HilbertRaum \oplus H_{0}$
    and
        $H_{0}$
    are $U$-invariant.
\end{lemm}

See \cite[Lemma~0]{Sarason1965Article},
\cite[Theorem~3.10]{Shalit2021DilationBook}
for a proof
(noting that we have reformulated things for our context).%
\footnote{%
    \viz the decomposition is usually (equivalently)
    stated as
        $\HilbertRaum^{\prime} = H_{0} \oplus H_{1}$,
    where
        $H_{0} \supseteq r\HilbertRaum$,
    and such that
        $H_{0}$ and $H_{0} \ominus r\HilbertRaum$
    are $U$-invariant subspaces.
}



\begin{lemm}[Cooper, 1947]
\makelabel{lemm:coopers-thm:sig:article-free-raj-dahya}
    Let $V$ be a $\Cnought$\=/semigroup of isometries on a Hilbert space $\HilbertRaum$.
    Then $V$ admits a unitary dilation of the form
        $(\HilbertRaum\oplus\HilbertRaum,\iota_{1},U)$
    for some strongly continuous representation
        $U \in \Repr{\reals}{\HilbertRaum\oplus\HilbertRaum}$.%
    \footnoteref{ft:ith-inclusion:\beweislabel}
\end{lemm}

\footnotetext[ft:ith-inclusion:\beweislabel]{%
    Recall that
        ${\iota_{i} : \HilbertRaum \to \HilbertRaum\oplus\HilbertRaum}$
    denotes the canonical inclusion into the $i$\textsuperscript{th} component
    for $i \in \{1,2\}$.
}

    \begin{proof}[Sketch]
        By \cite{Cooper1947Article},
        there exist
            $V$-invariant subspaces $H_{u},H_{s} \subseteq \HilbertRaum$,
            a Hilbert space $H$,
            and
            a unitary operator
                $\theta \in \BoundedOps{L^{2}(\realsNonNeg) \otimes H}{H_{s}}$,
        such that
            $\HilbertRaum = H_{u} \oplus H_{s}$
        and such that
            $V(\cdot)\restr{H_{u}}$
        is a unitary semigroup and
            $
                V(\cdot)\restr{H_{s}}
                = \theta\:(\ShiftRight(\cdot) \otimes \id_{H})\:\theta^{\ast}
            $,
        where $\ShiftRight$
        here denotes the continuous-time forwards-shift.%
        \footnote{%
            \idest
            $\ShiftRight(t)f = \einser_{[t, \infty)}(\cdot)\:f(\cdot - t)$
            for $t \in \realsNonNeg$, $f \in L^{2}(\realsNonNeg)$
        }

        Let $W \in \Repr{\reals}{H_{u}}$
        be the strongly continuous unitary representation
        corresponding to $V(\cdot)\restr{H_{u}}$.
        For $f \in L^{2}(\realsNonNeg)$
        let $\tilde{f} \in L^{2}(\reals)$
        denote the unique function extending $f$
        and equal to $0$ on $(-\infty,\:0)$.
        And for each $t \in \reals$
        let
            $f_{t,+},f_{t,-}\in L^{2}(\realsNonNeg)$
        denote the unique elements satisfying
            $
                \tilde{f}(\cdot - t)
                = f_{t,+}(\cdot)
                + f_{t,-}(-\cdot)
            $.
        Finally, construct $U \in \Repr{\reals}{\HilbertRaum\oplus\HilbertRaum}$
        via the following conditions

            \begin{shorteqnarray}
                U(t)\:\iota_{1}\xi &= &\iota_{1}\:W(t)\xi,\\
                U(t)\:\iota_{2}\xi &= &\iota_{2}\:\xi,\\
                U(t)\:\iota_{1} \theta (f \otimes \eta)
                    &= &\iota_{1}\:\theta\:(f_{t,+} \otimes \eta)
                        + \iota_{2}\:\theta\:(f_{t,-} \otimes \eta),
                ~\text{and}\\
                U(t)\:\iota_{2} \theta (f \otimes \eta)
                    &= &\iota_{1}\:\theta\:(f_{-t,-} \otimes \eta)
                        + \iota_{2}\:\theta\:(f_{-t,+} \otimes \eta)
            \end{shorteqnarray}

        \continueparagraph
        for
            $t \in \reals$,
            $\xi \in H_{u}$,
            $\eta \in H$,
            and
            $f \in L^{2}(\realsNonNeg)$.
        It is now routine to see that
        $U$ is indeed a strongly continuous representation
        and that $(\HilbertRaum\oplus\HilbertRaum,\iota_{1},U)$
        dilates $(\HilbertRaum,V)$.
    \end{proof}

Cooper's result provides us uniform means to dilate
semigroups of isometries to semigroups of unitaries,
whereby the construction exhibits the intertwining property,
\viz
    $U(\cdot) \circ \iota_{1} = \iota_{1} \circ V(\cdot)$.
This however holds in general in the abstract setting.

\begin{prop}[Intertwining property]
\makelabel{prop:invariance-of-rH-under-utr-dil-of-isometry:sig:article-free-raj-dahya}
    Let $V$ be a(n $\topSOT$\=/continuous) semigroup
    of isometries over a (topological) monoid $M$
    on a Hilbert space $\HilbertRaum$.
    Let $(\HilbertRaum^{\prime},r,U)$
    be any dilation of $(\HilbertRaum,V)$
    to a(n $\topSOT$\=/continuous) semigroup of
    unitaries over $M$ on $\HilbertRaum^{\prime}$.
    Then

    \begin{displaymath}
        U(x)\:r = r\:V(x)
    \end{displaymath}

    \continueparagraph
    for all $x \in M$.
    In particular, $r\HilbertRaum$ is $U$-invariant.
\end{prop}

    \begin{proof}
        Let $x \in M$ and $\xi \in \HilbertRaum$ be arbitrary.
        Since $\HilbertRaum^{\prime} = \ran(r) \oplus \ker(r^{\ast})$,
        there exist $\xi^{\prime} \in \HilbertRaum$
        and $\eta \in \ker(r^{\ast})$,
        such that $U(x)\:r\xi = r\xi^{\prime} + \eta$.
        The dilation yields
        $
            V(x)\xi
            = r^{\ast}\:U(x)\:r\xi
            = \xi^{\prime} + \zerovector
        $.
        Observing that
        $
            \norm{\eta}^{2}
            = \norm{r\xi^{\prime} + \eta}^{2}
            - \norm{r\xi^{\prime}}^{2}
            = \norm{\xi}^{2}
            - \norm{\xi^{\prime}}^{2}
        $
        and
        $
            \norm{\xi^{\prime}}
            = \norm{V(x)\xi}
            = \norm{\xi}
        $,
        it follows that $\eta = \zerovector$.
        So $U(x)\:r\xi = \xi^{\prime} + \zerovector = r\:V(x)\xi$.
    \end{proof}



\subsection[Proof of the Ist main result]{Proof of the \First main result}
\label{sec:result-abstract:proof-1st:sig:article-free-raj-dahya}

\firstparagraph
In this section we prove an abstract result which establishes
a second full proof of \Cref{thm:free-dilation:1st:sig:article-free-raj-dahya}
as an immediate consequence.
We first provide a definition and a simple technical manipulation.
Let $I$ be a non-empty index set,
$M_{i}$ topological monoids for each $i \in I$,
and $\HilbertRaum$ a Hilbert space.
Let $\{T_{i}\}_{i \in I}$ be a family of operator-valued functions
with $T_{i}$ being an $\topSOT$\=/continuous contractive semigroup over $M_{i}$ on $\HilbertRaum$
for each $i \in I$.
The results here will work with both variants of dilation
defined in \S{}\ref{sec:introduction:notation:sig:article-free-raj-dahya}.

\begin{defn}
\makelabel{defn:free-dilation-as-families:sig:article-free-raj-dahya}
    Let $\HilbertRaum^{\prime}$ be a Hilbert space,
    $r \in \BoundedOps{\HilbertRaum}{\HilbertRaum^{\prime}}$ an isometry,
    and $\{U_{i}\}_{i \in I}$
    a family of operator-valued functions with
        $U_{i}$ being an $\topSOT$\=/continuous unitary/isometric semigroup
        over $M_{i}$ on $\HilbertRaum^{\prime}$
    for each $i \in I$.
    We say that $(\HilbertRaum^{\prime}, r, \{U_{i}\}_{i \in I})$
    is an $\topSOT$\=/continuous \highlightTerm{free unitary/isometric dilation}
    of $(\HilbertRaum, \{T_{i}\}_{i \in I})$ if

    \begin{restoremargins}
    \begin{equation}
    \label{eq:free-dilation:abstract:sig:article-free-raj-dahya}
        \prod_{k=1}^{N}
            T_{i_{k}}(x_{k})
        = r^{\ast}
            \:\Big(
                \prod_{k=1}^{N}
                    U_{i_{k}}(x_{k})
            \Big)
            \:r
    \end{equation}
    \end{restoremargins}

    \continueparagraph
    for all $N \in \naturals$,
    all sequences $(i_{k})_{k=1}^{N} \subseteq I$,
    and all
        $(x_{k})_{k=1}^{N} \in \prod_{k=1}^{N}M_{i_{k}}$.

    In the \usesinglequotes{unitary} case,
    if we instead consider pairs $(G_{i},M_{i})$
    of topological groups and submonoids for each $i \in I$,
    then the above concept is analogously defined
    with instead each $U_{i}$ being a (continuos) unitary representation
    of $G_{i}$ on $\HilbertRaum^{\prime}$.
\end{defn}

\begin{prop}
\makelabel{prop:replace-hilbert-and-embedding-by-same-objects:sig:article-free-raj-dahya}
    If each $T_{i}$ admits a dilation
    to an $\topSOT$\=/continuous unitary semigroup over $M_{i}$,
    then there exists
        a common Hilbert space $\HilbertRaum^{\prime}$
        and
        a common isometric embedding $r\in\BoundedOps{\HilbertRaum}{\HilbertRaum^{\prime}}$,
        as well as
        a family $\{U_{i}\}_{i \in I}$
        with $U_{i}$ being an $\topSOT$\=/continuous unitary semigroup over $M_{i}$ on $\HilbertRaum^{\prime}$
        for each $i \in I$,
    such that
        $(\HilbertRaum^{\prime},r,U_{i})$
        is a dilation of $(\HilbertRaum,T_{i})$
    for each $i \in I$.
\end{prop}

    \begin{proof}
        By assumption there exists an $\topSOT$\=/continuous unitary dilation
            $(H_{i}, r_{i}, U_{i})$
            of
            $(\HilbertRaum, T_{i})$
        for each $i \in I$.
        Set $\HilbertRaum^{\prime} \coloneqq \bigoplus_{i \in I}H_{i}$.
        Let $i \in I$.
        Setting
            $
                U'_{i}(\cdot)
                \coloneqq
                U_{i}(\cdot) \oplus \onematrix_{\HilbertRaum^{\prime} \ominus H_{i}}
            $,
        it is easy to verify that
            $(\HilbertRaum^{\prime}, \iota_{i} \circ r_{i},U'_{i})$
        is a unitary dilation of $(\HilbertRaum,T_{i})$
        to an $\topSOT$\=/continuous unitary semigroup over $M_{i}$ on $\HilbertRaum^{\prime}$.%
        \footnote{%
            Recall that
                ${\iota_{i} : H_{i} \to \HilbertRaum^{\prime}}$
            denotes the canonical inclusion into the $i$\textsuperscript{th} component
            for $i \in I$.
        }
        Now, since the subspaces
            $(\iota_{i} \circ r_{i})\HilbertRaum \subseteq \HilbertRaum^{\prime}$
        are all isomorphic,
        we can find unitary operators $w_{i} \in \BoundedOps{\HilbertRaum^{\prime}}$
        such that $\iota_{i} \circ r_{i} = w_{i} \circ r$
        for each $i \in I$
        and some isometry $r \in \BoundedOps{\HilbertRaum}{\HilbertRaum^{\prime}}$.
        Setting
            $U''_{i}(\cdot) \coloneqq w_{i}^{\ast}\:U'_{i}(\cdot)\:w_{i}$
        one can easily verify that $(\HilbertRaum^{\prime},r,U''_{i})$
        is a dilation of $(\HilbertRaum,T_{i})$
        to an $\topSOT$\=/continuous semigroup of unitaries over $M_{i}$ on $\HilbertRaum^{\prime}$
        for each $i \in I$.
    \end{proof}

As per the discussion in \Cref{rem:naive-approach:sig:article-free-raj-dahya}
this result does not immediately give us a free dilation.
Using structure theorems however allows us to attain this goal.

\begin{highlightboxWithBreaks}
\begin{thm}[Free dilations, abstract formulation]
\makelabel{thm:free-dilation:abstract:sig:article-free-raj-dahya}
    Let $I$ be a non-empty index set
    and let $M_{i} \in \DilatableMonoids$
    be topological monoids for each $i \in I$.%
    \footnoteref{ft:all-semigroups-have-a-dilation:single:sec:result-abstract:proof-1st:sig:article-free-raj-dahya}
    Let $\{T_{i}\}_{i \in I}$
    be a family of operator-valued functions
    with $T_{i}$ being an $\topSOT$\=/continuous contractive semigroup over $M_{i}$ on $\HilbertRaum$
    for each $i \in I$.
    Then
        $(\HilbertRaum, \{T_{i}\}_{i \in I})$
    admits a free dilation to a family
    of $\topSOT$\=/continuous unitary semigroups.
\end{thm}
\end{highlightboxWithBreaks}

\footnotetext[ft:all-semigroups-have-a-dilation:single:sec:result-abstract:proof-1st:sig:article-free-raj-dahya]{%
    \idest any $\topSOT$\=/continuous contractive semigroup over $M_{i}$
    admits a dilation to an $\topSOT$\=/continuous unitary semigroup over $M_{i}$
    for each $i \in I$
    (see \S{}\ref{sec:introduction:notation:sig:article-free-raj-dahya}).
}

Note that by \cite[Theorem~I.8.1]{Nagy1970},
\Cref{thm:free-dilation:abstract:sig:article-free-raj-dahya}
directly implies the concrete version of
the \First free dilation theorem
(\Cref{thm:free-dilation:1st:sig:article-free-raj-dahya}).
The proof of this abstract result consists of two parts:
1) Construction of a family of semigroups of isometries,
    instrumentalising Sarason's lemma;
2) Dilation of these to semigroups of unitaries,
    relying on the intertwining property.

\def\beweislabel{thm:free-dilation:abstract:sig:article-free-raj-dahya}
\begin{proof}[of \Cref{\beweislabel}]
    \paragraph{Construction of a free isometric dilation:}
    By assumption there is a(n $\topSOT$\=/continuous) unitary dilation
    $(H_{i},r_{i},V_{i})$
    of
    $(\HilbertRaum,T_{i})$
    for each $i \in I$.
    Moreover, by \Cref{prop:replace-hilbert-and-embedding-by-same-objects:sig:article-free-raj-dahya}
    we may assume that
    each $H_{i} = H$
    and $r_{i} = r$
    for some Hilbert space $H$
    and some isometry $r \in \BoundedOps{\HilbertRaum}{H}$.
    By Sarason's result (\Cref{lemm:sarason:sig:article-free-raj-dahya}),
    there exist decompositions

    \begin{shorteqnarray}
        H
        = r\HilbertRaum \oplus H^{(i)}_{0} \oplus H^{(i)}_{1}
    \end{shorteqnarray}

    \continueparagraph
    where $r\HilbertRaum \oplus H^{(i)}_{0}$
    and $H^{(i)}_{0}$
    are $V_{i}$-invariant for each $i \in I$.
    By expanding the underlying spaces of the dilations,
    we may assume that
        $
            \dim(H \ominus r\HilbertRaum)
            \geq
            \dim(\bigoplus_{i \in I}H^{(i)}_{0})
        $.%
    \footnote{%
        Simply replace $H$ by $H \oplus H'$
        for a large enough Hilbert space $H'$
        and set
        $
            \tilde{V}_{i}(\cdot)
            \coloneqq
            V_{i}(\cdot) \oplus \onematrix_{H'}
        $
        for each $i \in I$.
    }
    Since the auxiliary spaces
        $H^{(i)}_{0} \oplus H^{(i)}_{1}$
    are large enough,
    we may adjust the dilations
    via unitaries
    and further ensure that
        $\{H^{(i)}_{0}\}_{i \in I}$
    is a family of orthogonal subspaces of $H$.
    Set

        \begin{shorteqnarray}
            H_{0} &\coloneqq &\bigoplus_{i \in I}H^{(i)}_{0},\\
            \tilde{H}
                &\coloneqq &r\HilbertRaum \oplus H_{0},
            ~\text{and}\\
            \tilde{V}_{i}(\cdot)
                &\coloneqq
                    &V_{i}(\cdot)\restr{r\HilbertRaum \oplus H^{(0)}_{i}}
                    \oplus
                    \onematrix_{H_{0} \ominus H^{(0)}_{i}}
        \end{shorteqnarray}

    \continueparagraph
    for each $i \in I$
    and let $\tilde{r}\in\BoundedOps{\HilbertRaum}{\tilde{H}}$
    be the isometry defined by
        $\tilde{r}\xi = r\xi$
    for $\xi \in \HilbertRaum$.
    The invariance of the subspaces guarantees that
    each $\tilde{V}_{i}$ is a well-defined
    ($\topSOT$\=/continuous) semigroup of isometries over $M$ on $\tilde{H}$
    and each
        $(\tilde{H},r,\tilde{V}_{i})$
    remains a dilation
    of $(\HilbertRaum,T_{i})$.%
    \footnote{%
        By restricting to the constructed invariant subspaces,
        $\tilde{V}_{i}(x)$ may cease to be surjective
        for $x \in M_{i}$, $i \in I$.
    }

    Set $p \coloneqq \tilde{r}\tilde{r}^{\ast}$,
    which is the projection in $\tilde{H}$
    onto $\ran(\tilde{r}) = \ran(r) = r\HilbertRaum$.
    Since for each $i \in I$,
    both $H_{i}$
    and $H_{0} \ominus H_{i} = \bigoplus_{j \in I \setminus \{i\}}H_{j}$
    are $\tilde{V}_{i}$-invariant subspaces,
    one has
        $
            p\:\tilde{V}_{i}(\cdot)\:(\onematrix - p)
            = \zeromatrix
        $.
    Thus

        \begin{restoremargins}
        \begin{equation}
        \label{eq:0:\beweislabel}
            p\:\tilde{V}_{i}(x)\:p\:\tilde{V}_{j}(y)
            = p\:\tilde{V}_{i}(x)\:(p + (\onematrix - p))\:\tilde{V}_{j}(y)
            = p\:\tilde{V}_{i}(x)\tilde{V}_{j}(y)
        \end{equation}
        \end{restoremargins}

    \continueparagraph
    for $x,y \in M$, $i,j \in I$.
    For,
        for $N\in\naturals$,
        $(i_{k})_{k=1}^{N} \subseteq I$,
        and
        $(x_{k})_{k=1}^{N} \subseteq M$,
    we obtain

        \begin{displaymath}
            \prod_{k=1}^{N}
                T_{i_{k}}(x_{k})
            = \prod_{k=1}^{N}
                    \tilde{r}^{\ast}\:\tilde{V}_{i_{k}}(x_{k})\:\tilde{r}
            = \tilde{r}^{\ast}
                \:\Big(
                    \prod_{k=1}^{N}
                        p
                        \:\tilde{V}_{i_{k}}(x_{k})
                \Big)
                \:\tilde{r},
        \end{displaymath}

    \continueparagraph
    and applying \eqcref{eq:0:\beweislabel},
    the product in parentheses
    can be reduced successively to
        $p\:\tilde{V}_{i_{1}}(x_{1})\tilde{V}_{i_{2}}(x_{2})\ldots\tilde{V}_{i_{N}}(x_{N})$.
    So

        \begin{displaymath}
            \prod_{k=1}^{N}
                T_{i_{k}}(x_{k})
            = \tilde{r}^{\ast}
                p
                \:\Big(
                    \prod_{k=1}^{N}
                    \tilde{V}_{i_{k}}(x_{k})
                \Big)
                \:\tilde{r},
        \end{displaymath}

    \continueparagraph
    whence
    $(\tilde{H},\tilde{r},\{\tilde{V}_{i}\}_{i \in I})$
    is a (continuous) free isometric dilation
    of $(\HilbertRaum, \{T_{i}\}_{i \in I})$,
    since $\tilde{r}^{\ast}p = \tilde{r}^{\ast}$.

    \paragraph{Transition to a free unitary dilation:}
    Applying the main assumption again,
    there is a(n $\topSOT$\=/continuous) unitary dilation
        $(\breve{H}_{i},s_{i},U_{i})$
        of
        $(\tilde{H},\tilde{V}_{i})$
    for each $i \in I$.
    And by \Cref{prop:replace-hilbert-and-embedding-by-same-objects:sig:article-free-raj-dahya}
    we may assume that
    each $\breve{H}_{i} = \breve{H}$ and $s_{i} = s$
    for some Hilbert space $\breve{H}$
    and some isometry $s \in \BoundedOps{\HilbertRaum}{\breve{H}}$.
    By the intertwining property of unitary dilations of semigroups of isometries
    (see \Cref{prop:invariance-of-rH-under-utr-dil-of-isometry:sig:article-free-raj-dahya}),
    we have
        $U_{i}(x)\:s = s\:V_{i}(t)$
    for $x \in M_{i}$, $i \in I$.
    Repeated applications of this intertwining yields

        \begin{displaymath}
            (s \circ \tilde{r})^{\ast}
            \:\Big(
                \prod_{k=1}^{N}
                    U_{i_{k}}(x_{k})
            \Big)
            \:(s \circ \tilde{r})
            = \tilde{r}^{\ast}
                \:s^{\ast}
                \:\Big(
                    \prod_{k=1}^{N}
                        U_{i_{k}}(x_{k})
                \Big)
                \:s
                \:\tilde{r}
            = \tilde{r}^{\ast}
                \:\cancel{s^{\ast}s}
                \:\Big(
                    \prod_{k=1}^{N}
                        \tilde{V}_{i_{k}}(x_{k})
                \Big)
                \:\tilde{r}
        \end{displaymath}

    \continueparagraph
    for $N\in\naturals$,
        $(i_{k})_{k=1}^{N} \subseteq I$,
        and $(x_{k})_{k=1}^{N} \in \prod_{k=1}^{N}M_{i_{k}}$.
    As $\{\tilde{V}_{i}\}_{i \in I}$
    is a free dilation of $\{T_{i}\}_{i \in I}$,
    it follows that
        $(\breve{H}, s \circ \tilde{r}, \{U_{i}\}_{i \in I})$
    is a(n $\topSOT$\=/continuous) free unitary dilation of
    $(\HilbertRaum,\{T_{i}\}_{i \in I})$.
\end{proof}

\begin{rem}
    It is a straightforward exercise to see that
    the first claims in
    \Cref{prop:replace-hilbert-and-embedding-by-same-objects:sig:article-free-raj-dahya}
    and
    \Cref{thm:free-dilation:abstract:sig:article-free-raj-dahya},
    continue to hold if
    \usesinglequotes{unitary}
    is replaced by \usesinglequotes{isometric}
    and $\DilatableMonoids$
    is replaced by the class of monoids $M$
    for which all
        $\topSOT$\=/continuous contractive semigroups over $M$
    admit a $\topSOT$\=/continuous isometric dilation over $M$.
\end{rem}

\begin{rem}
\makelabel{rem:free-dilation:abstract:repr-version:sig:article-free-raj-dahya}
    In \Cref{prop:replace-hilbert-and-embedding-by-same-objects:sig:article-free-raj-dahya}
    and \Cref{thm:free-dilation:abstract:sig:article-free-raj-dahya}
    we may instead consider pairs $(G_{i},M_{i})$
    of topological groups and submonoids for each $i \in I$.
    In \Cref{prop:replace-hilbert-and-embedding-by-same-objects:sig:article-free-raj-dahya},
    it is straightforward to see that if each $T_{i}$ admits a dilation
    to an $\topSOT$\=/continuous unitary representation of $G_{i}$,
    then the result continues to hold with
    $U_{i}$ being instead a continuos unitary representation of $G_{i}$.
    Relying on this version of the \namecref{prop:replace-hilbert-and-embedding-by-same-objects:sig:article-free-raj-dahya},
    if each $(G_{i},M_{i}) \in \DilatableGroupMonoids$
    in \Cref{thm:free-dilation:abstract:sig:article-free-raj-dahya},%
    \footnote{%
        \idest any $\topSOT$\=/continuous contractive semigroup over $M_{i}$
        admits a dilation to an $\topSOT$\=/continuous unitary representation of $G_{i}$
        for each $i \in I$
        (see \S{}\ref{sec:introduction:notation:sig:article-free-raj-dahya}).
    }
    then arguing in essentially the same way as above,
    one obtains that
        $(\HilbertRaum, \{T_{i}\}_{i \in I})$
    admits a free dilation to a family
    of $\topSOT$\=/continuous unitary representations.
\end{rem}




\section[Semigroups over free topological products]{Semigroups over free topological products}
\label{sec:algebra:sig:article-free-raj-dahya}

\firstparagraph
Our aim is to now reformulate
the \First free dilation theorem
in terms of semigroups defined over a submonoid
of free products of topological groups.
This provides a natural algebraic setting
for non-commuting families of semigroups.
We first recall some algebraic notions.


\subsection[Free topological products]{Free topological products}
\label{sec:algebra:free-products:sig:article-free-raj-dahya}

\firstparagraph
Consider topological groups $G_{i}$ for $i \in I$
and some non-empty index set $I$.

\begin{defn}[Free topological product]
\makelabel{defn:free-prod-group:sig:article-free-raj-dahya}
    A topological group $G$
    together with a continuous homomorphism
    ${\iota_{i} : G_{i} \to G}$
    constitute a \highlightTerm{free topological product}
    of $(G_{i})_{i \in I}$
    if

    \begin{enumerate}[label={\upshape\texttt{FP}\textsubscript{\arabic*}}]
        \item\label{ax:freeproducts:gen:sig:article-free-raj-dahya}
            $G$ is algebraically generated by $\bigcup_{i \in I}\iota_{i}(G_{i})$;
            and

        \item\label{ax:freeproducts:ext:sig:article-free-raj-dahya}
            For each topological group $\Gamma$
            and continuous homomorphisms
            ${\theta_{i} : G_{i} \to \Gamma}$
            for $i \in I$,
            there exists a continuous homomorphism
            ${\theta : G \to \Gamma}$
            satisfying $\theta \circ \iota_{i} = \theta_{i}$
            for each $i \in I$.
    \end{enumerate}

    \nvraum{1.5}

\end{defn}

There are a few things to immediately observe:
In the presence of axiom \eqcref{ax:freeproducts:gen:sig:article-free-raj-dahya},
the induced homomorphism $\theta$ in \eqcref{ax:freeproducts:ext:sig:article-free-raj-dahya}
is necessarily unique.
And by \eqcref{ax:freeproducts:ext:sig:article-free-raj-dahya} each
    $\iota_{i}$
is necessarily a homeomorphic embedding.%
\footnote{%
    Let $\Gamma \coloneqq G_{i}$.
    Consider the continuous homomorphisms
        ${\theta_{i} \coloneqq \id_{\Gamma} : \Gamma \ni x \mapsto x \in G_{i}}$
    and
        ${\theta_{j} \coloneqq \ConstOne_{j} : \Gamma \ni x \mapsto e \in G_{j}}$
    for $j \in I \without \{i\}$,
    By \eqcref{ax:freeproducts:ext:sig:article-free-raj-dahya}
    there exists a (unique) continuous homomorphism
        ${\pi_{i} : G \to \Gamma = G_{i}}$
    satisfying
        $\pi_{i} \circ \iota_{i} = \id_{\Gamma}$
        and
        $\pi_{i} \circ \iota_{j} = \ConstOne_{j}$
    for $j \in I \without \{i\}$.
    Since
        $\pi_{i} \circ \iota_{i} = \id_{G_{i}}$
    and all three functions in this expressions are continuous,
    it follows that $\iota_{i}$ is injective
    and bi-continuous.
}

\begin{thm}[Graev, 1950]
\makelabel{thm:graev:free-prod-exists:sig:article-free-raj-dahya}
    The free topological product exists
    and is unique up to algebraic and topological isomorphism.
\end{thm}

See \cite{Graev1950ArticleFreeProducts},
    \cite{Hulanicki1967Article},
    \etcetera
for a proof.
We highlight some basic aspects of these free products:

\begin{kompaktenum}{\bfseries 1.}
\item
    The underlying group can be defined as the (unique up to isomorphism)
    free algebraic product, $G$, of the $G_{i}$.
    Formally this construction consists of
    injective homomorphisms
        ${\iota_{i} : G_{i} \to G}$
    for each $i \in I$,
    such that the images $\iota_{i}(G_{i})$
    are pairwise disjoint.
    Each element of $x \in G$
    has a unique presentation

    \begin{restoremargins}
    \begin{equation}
    \label{eq:free-prod-elements:sig:article-free-raj-dahya}
        x = \prod_{k=1}^{N} \iota_{i_{k}}(x_{k})
    \end{equation}
    \end{restoremargins}

    \continueparagraph
    for some $N \in \naturalsZero$,
    $(i_{k})_{k=1}^{N} \subseteq I$ bubble-swap free,
    and some
    $(x_{k})_{k=1}^{N} \in \prod_{k=1}^{N}(G_{k} \without \{e\})$.
    We refer to this as the \highlightTerm{minimal representation} of $x$.

\item
    Given two elements $x,y \neq e$ in the free product,
    say with minimal representations
        $x = \prod_{j=1}^{M}\iota_{i_{j}}(x_{j})$
        and
        $y = \prod_{j=1}^{N}\iota_{i'_{j}}(y_{j})$
    (see \eqcref{eq:free-prod-elements:sig:article-free-raj-dahya}),
    the product $xy$ is determined by checking
    which elements
    \usesinglequotes{at the end} of $x$
    cancel out with elements
    \usesinglequotes{at the start} of $y$.
    Supposing that $l \leq \max\{M,N\}$
    is maximal with
    $i_{M+1-k} = i'_{k}$
    for $k \in \{1,2,\ldots,l\}$
    and
    $y_{k} = x_{M+1-k}^{-1}$
    for $k \in \{1,2,\ldots,l-1\}$
    one has

        \begin{displaymath}
            xy = \begin{displaycases}
                \prod_{j=1}^{M-l}\iota_{i_{j}}(x_{j})
                \cdot
                \prod_{j=l+1}^{N}\iota_{i'_{j}}(y_{j})
                    &: &y_{l} = x_{l}^{-1}\\
                \prod_{j=1}^{M-l}\iota_{i_{j}}(x_{j})
                \cdot
                \iota_{i'_{l}}(x_{l}y_{l})
                \cdot
                \prod_{j=l+1}^{N}\iota_{i'_{j}}(y_{j})
                    &: &\text{otherwise}\\
            \end{displaycases}
        \end{displaymath}

    \continueparagraph
    as the minimal representation of $xy$.
    \Cref{fig:free-product:sig:article-free-raj-dahya}
    depicts some examples.
\end{kompaktenum}

\begin{conv}
    We denote the underlying set of the topological free product
    of $(G_{i})_{i \in I}$
    by $\FreeProduct{i \in I}{G_{i}}$.
    Given subsets $S_{i} \subseteq G_{i}$ for $i \in I$
    we let
        $
            \FreeProduct{i \in I}{S_{i}}
            \coloneqq
            \bigcup_{N\in\naturalsZero}
            \{
                \prod_{k=1}^{N}
                    \iota_{i_{k}}(x_{k})
                \mid
                (i_{k})_{k=1}^{N} \subseteq I,
                ~(x_{k})_{k=1}^{N} \in \prod_{k=1}^{N}S_{i_{k}}
            \}
        $.
    If each $G_{i} = G$ for some (topological) group $G$
    and each $S_{i} = S$ for some subset $S \subseteq G$,
    then we shorten the above to
        $\FreePower{I}{G}$
        \resp
        $\FreePower{I}{S}$,
    or
        $\FreePower{d}{G}$
        \resp
        $\FreePower{d}{S}$
    if $I = \{1,2,\ldots,d\}$ for some $d\in\naturals$.
\end{conv}


\begin{figure}[!ht]
\begin{mdframed}[%
    backgroundcolor=background_light_grey,%
    linewidth=0pt,%
]
    \centering
    \begin{subfigure}[t]{0.3\textwidth}
        \centering
        \begin{tikzpicture}[node distance=1cm,thick]
            \pgfmathsetmacro\hunit{1}
            \pgfmathsetmacro\vunit{1}
            \pgfmathsetmacro\hscale{1}
            \pgfmathsetmacro\vscale{1}

            \draw [draw=black, line width=0.5pt, fill=none]
                ({-1.8599999999999999 * \hscale * \hunit}, {-1.4400000000000002 * \vscale * \vunit})
                -- ({1.2 * \hscale * \hunit}, {-1.4400000000000002 * \vscale * \vunit})
                -- ({1.2 * \hscale * \hunit}, {1.44 * \vscale * \vunit})
                -- ({-1.8599999999999999 * \hscale * \hunit}, {1.44 * \vscale * \vunit})
                -- cycle;

            \draw [draw=none, fill=none]
                ({-1.8599999999999999 * \hscale * \hunit}, {1.44 * \vscale * \vunit})
                node[label={[below right, xshift=0pt, yshift=-2pt] \scriptsize $\reals^{2}$}]{};

            \node (O) at (0, 0) {$\bullet$};

            \draw [->, dashed, draw=black, line width=0.5pt, fill=none]
                ({-1.8599999999999999 * \hscale * \hunit}, {0 * \vscale * \vunit})
                -- ({1.2 * \hscale * \hunit}, {0 * \vscale * \vunit});

            \draw [->, dashed, draw=black, line width=0.5pt, fill=none]
                ({0 * \hscale * \hunit}, {-1.4400000000000002 * \vscale * \vunit})
                -- ({0 * \hscale * \hunit}, {1.44 * \vscale * \vunit});

            \draw [->, draw=black, line width=1pt, fill=none]
                ({0.0 * \hscale * \hunit}, {0.0 * \vscale * \vunit})
                -- ({-0.5 * \hscale * \hunit}, {0.0 * \vscale * \vunit});
            \draw [->, draw=black, line width=1pt, fill=none]
                ({-0.5 * \hscale * \hunit}, {0.0 * \vscale * \vunit})
                -- ({-0.5 * \hscale * \hunit}, {1.2 * \vscale * \vunit});
            \draw [->, draw=black, line width=1pt, fill=none]
                ({-0.5 * \hscale * \hunit}, {1.2 * \vscale * \vunit})
                -- ({1.0 * \hscale * \hunit}, {1.2 * \vscale * \vunit});
            \draw [->, draw=black, line width=1pt, fill=none]
                ({1.0 * \hscale * \hunit}, {1.2 * \vscale * \vunit})
                -- ({1.0 * \hscale * \hunit}, {0.7999999999999999 * \vscale * \vunit});
            \draw [->, draw=black, line width=1pt, fill=none]
                ({1.0 * \hscale * \hunit}, {0.7999999999999999 * \vscale * \vunit})
                -- ({0.7 * \hscale * \hunit}, {0.7999999999999999 * \vscale * \vunit});

        \end{tikzpicture}
        \caption{Element $x$}
        \label{fig:free-product:a:sig:article-free-raj-dahya}
    \end{subfigure}
    \begin{subfigure}[t]{0.3\textwidth}
        \centering
        \begin{tikzpicture}[node distance=1cm,thick]
            \pgfmathsetmacro\hunit{1}
            \pgfmathsetmacro\vunit{1}
            \pgfmathsetmacro\hscale{1}
            \pgfmathsetmacro\vscale{1}

            \draw [draw=black, line width=0.5pt, fill=none]
                ({-1.8599999999999999 * \hscale * \hunit}, {-1.4400000000000002 * \vscale * \vunit})
                -- ({1.2 * \hscale * \hunit}, {-1.4400000000000002 * \vscale * \vunit})
                -- ({1.2 * \hscale * \hunit}, {1.44 * \vscale * \vunit})
                -- ({-1.8599999999999999 * \hscale * \hunit}, {1.44 * \vscale * \vunit})
                -- cycle;

            \draw [draw=none, fill=none]
                ({-1.8599999999999999 * \hscale * \hunit}, {1.44 * \vscale * \vunit})
                node[label={[below right, xshift=0pt, yshift=-2pt] \scriptsize $\reals^{2}$}]{};

            \node (O) at (0, 0) {$\bullet$};

            \draw [->, dashed, draw=black, line width=0.5pt, fill=none]
                ({-1.8599999999999999 * \hscale * \hunit}, {0 * \vscale * \vunit})
                -- ({1.2 * \hscale * \hunit}, {0 * \vscale * \vunit});

            \draw [->, dashed, draw=black, line width=0.5pt, fill=none]
                ({0 * \hscale * \hunit}, {-1.4400000000000002 * \vscale * \vunit})
                -- ({0 * \hscale * \hunit}, {1.44 * \vscale * \vunit});

            \draw [->, draw=black, line width=1pt, fill=none]
                ({0.0 * \hscale * \hunit}, {0.0 * \vscale * \vunit})
                -- ({0.3 * \hscale * \hunit}, {0.0 * \vscale * \vunit});
            \draw [->, draw=black, line width=1pt, fill=none]
                ({0.3 * \hscale * \hunit}, {0.0 * \vscale * \vunit})
                -- ({0.3 * \hscale * \hunit}, {0.4 * \vscale * \vunit});
            \draw [->, draw=black, line width=1pt, fill=none]
                ({0.3 * \hscale * \hunit}, {0.4 * \vscale * \vunit})
                -- ({-0.3 * \hscale * \hunit}, {0.4 * \vscale * \vunit});
            \draw [->, draw=black, line width=1pt, fill=none]
                ({-0.3 * \hscale * \hunit}, {0.4 * \vscale * \vunit})
                -- ({-0.3 * \hscale * \hunit}, {-1.2000000000000002 * \vscale * \vunit});
            \draw [->, draw=black, line width=1pt, fill=none]
                ({-0.3 * \hscale * \hunit}, {-1.2000000000000002 * \vscale * \vunit})
                -- ({-1.55 * \hscale * \hunit}, {-1.2000000000000002 * \vscale * \vunit});

        \end{tikzpicture}
        \caption{Element $y$}
        \label{fig:free-product:b:sig:article-free-raj-dahya}
    \end{subfigure}
    \begin{subfigure}[t]{0.3\textwidth}
        \centering
        \begin{tikzpicture}[node distance=1cm,thick]
            \pgfmathsetmacro\hunit{1}
            \pgfmathsetmacro\vunit{1}
            \pgfmathsetmacro\hscale{1}
            \pgfmathsetmacro\vscale{1}

            \draw [draw=black, line width=0.5pt, fill=none]
                ({-1.8599999999999999 * \hscale * \hunit}, {-1.4400000000000002 * \vscale * \vunit})
                -- ({1.2 * \hscale * \hunit}, {-1.4400000000000002 * \vscale * \vunit})
                -- ({1.2 * \hscale * \hunit}, {1.44 * \vscale * \vunit})
                -- ({-1.8599999999999999 * \hscale * \hunit}, {1.44 * \vscale * \vunit})
                -- cycle;

            \draw [draw=none, fill=none]
                ({-1.8599999999999999 * \hscale * \hunit}, {1.44 * \vscale * \vunit})
                node[label={[below right, xshift=0pt, yshift=-2pt] \scriptsize $\reals^{2}$}]{};

            \node (O) at (0, 0) {$\bullet$};

            \draw [->, dashed, draw=black, line width=0.5pt, fill=none]
                ({-1.8599999999999999 * \hscale * \hunit}, {0 * \vscale * \vunit})
                -- ({1.2 * \hscale * \hunit}, {0 * \vscale * \vunit});

            \draw [->, dashed, draw=black, line width=0.5pt, fill=none]
                ({0 * \hscale * \hunit}, {-1.4400000000000002 * \vscale * \vunit})
                -- ({0 * \hscale * \hunit}, {1.44 * \vscale * \vunit});

            \draw [->, draw=black, line width=1pt, fill=none]
                ({0.0 * \hscale * \hunit}, {0.0 * \vscale * \vunit})
                -- ({-0.5 * \hscale * \hunit}, {0.0 * \vscale * \vunit});
            \draw [->, draw=black, line width=1pt, fill=none]
                ({-0.5 * \hscale * \hunit}, {0.0 * \vscale * \vunit})
                -- ({-0.5 * \hscale * \hunit}, {1.2 * \vscale * \vunit});
            \draw [->, draw=black, line width=1pt, fill=none]
                ({-0.5 * \hscale * \hunit}, {1.2 * \vscale * \vunit})
                -- ({0.4 * \hscale * \hunit}, {1.2 * \vscale * \vunit});
            \draw [->, draw=black, line width=1pt, fill=none]
                ({0.4 * \hscale * \hunit}, {1.2 * \vscale * \vunit})
                -- ({0.4 * \hscale * \hunit}, {-0.40000000000000013 * \vscale * \vunit});
            \draw [->, draw=black, line width=1pt, fill=none]
                ({0.4 * \hscale * \hunit}, {-0.40000000000000013 * \vscale * \vunit})
                -- ({-0.85 * \hscale * \hunit}, {-0.40000000000000013 * \vscale * \vunit});

        \end{tikzpicture}
        \caption{Product of elements $xy$}
        \label{fig:free-product:c:sig:article-free-raj-dahya}
    \end{subfigure}

    \null

    \begin{subfigure}[t]{0.3\textwidth}
        \centering
        \begin{tikzpicture}[node distance=1cm,thick]
            \pgfmathsetmacro\hunit{1}
            \pgfmathsetmacro\vunit{1}
            \pgfmathsetmacro\hscale{1}
            \pgfmathsetmacro\vscale{1}

            \draw [draw=black, line width=0.5pt, fill=none]
                ({-1 * \hscale * \hunit}, {-1 * \vscale * \vunit})
                -- ({2.16 * \hscale * \hunit}, {-1 * \vscale * \vunit})
                -- ({2.16 * \hscale * \hunit}, {1.92 * \vscale * \vunit})
                -- ({-1 * \hscale * \hunit}, {1.92 * \vscale * \vunit})
                -- cycle;

            \draw [draw=none, fill=none]
                ({-1 * \hscale * \hunit}, {1.92 * \vscale * \vunit})
                node[label={[below right, xshift=0pt, yshift=-2pt] \scriptsize $\reals^{2}$}]{};

            \node (O) at (0, 0) {$\bullet$};

            \draw [->, dashed, draw=black, line width=0.5pt, fill=none]
                ({-1 * \hscale * \hunit}, {0 * \vscale * \vunit})
                -- ({2.16 * \hscale * \hunit}, {0 * \vscale * \vunit});

            \draw [->, dashed, draw=black, line width=0.5pt, fill=none]
                ({0 * \hscale * \hunit}, {-1 * \vscale * \vunit})
                -- ({0 * \hscale * \hunit}, {1.92 * \vscale * \vunit});

            \draw [->, draw=black, line width=1pt, fill=none]
                ({0.0 * \hscale * \hunit}, {0.0 * \vscale * \vunit})
                -- ({0.5 * \hscale * \hunit}, {0.0 * \vscale * \vunit});
            \draw [->, draw=black, line width=1pt, fill=none]
                ({0.5 * \hscale * \hunit}, {0.0 * \vscale * \vunit})
                -- ({0.5 * \hscale * \hunit}, {1.0 * \vscale * \vunit});
            \draw [->, draw=black, line width=1pt, fill=none]
                ({0.5 * \hscale * \hunit}, {1.0 * \vscale * \vunit})
                -- ({0.8 * \hscale * \hunit}, {1.0 * \vscale * \vunit});
            \draw [->, draw=black, line width=1pt, fill=none]
                ({0.8 * \hscale * \hunit}, {1.0 * \vscale * \vunit})
                -- ({0.8 * \hscale * \hunit}, {1.3 * \vscale * \vunit});

        \end{tikzpicture}
        \caption{Element $x^{\prime} \in \FreePower{d}{\realsNonNeg}$}
        \label{fig:free-product:d:sig:article-free-raj-dahya}
    \end{subfigure}
    \begin{subfigure}[t]{0.3\textwidth}
        \centering
        \begin{tikzpicture}[node distance=1cm,thick]
            \pgfmathsetmacro\hunit{1}
            \pgfmathsetmacro\vunit{1}
            \pgfmathsetmacro\hscale{1}
            \pgfmathsetmacro\vscale{1}

            \draw [draw=black, line width=0.5pt, fill=none]
                ({-1 * \hscale * \hunit}, {-1 * \vscale * \vunit})
                -- ({2.16 * \hscale * \hunit}, {-1 * \vscale * \vunit})
                -- ({2.16 * \hscale * \hunit}, {1.92 * \vscale * \vunit})
                -- ({-1 * \hscale * \hunit}, {1.92 * \vscale * \vunit})
                -- cycle;

            \draw [draw=none, fill=none]
                ({-1 * \hscale * \hunit}, {1.92 * \vscale * \vunit})
                node[label={[below right, xshift=0pt, yshift=-2pt] \scriptsize $\reals^{2}$}]{};

            \node (O) at (0, 0) {$\bullet$};

            \draw [->, dashed, draw=black, line width=0.5pt, fill=none]
                ({-1 * \hscale * \hunit}, {0 * \vscale * \vunit})
                -- ({2.16 * \hscale * \hunit}, {0 * \vscale * \vunit});

            \draw [->, dashed, draw=black, line width=0.5pt, fill=none]
                ({0 * \hscale * \hunit}, {-1 * \vscale * \vunit})
                -- ({0 * \hscale * \hunit}, {1.92 * \vscale * \vunit});

            \draw [->, draw=black, line width=1pt, fill=none]
                ({0.0 * \hscale * \hunit}, {0.0 * \vscale * \vunit})
                -- ({0.0 * \hscale * \hunit}, {0.3 * \vscale * \vunit});
            \draw [->, draw=black, line width=1pt, fill=none]
                ({0.0 * \hscale * \hunit}, {0.3 * \vscale * \vunit})
                -- ({1.0 * \hscale * \hunit}, {0.3 * \vscale * \vunit});

        \end{tikzpicture}
        \caption{Element $y^{\prime} \in \FreePower{d}{\realsNonNeg}$}
        \label{fig:free-product:e:sig:article-free-raj-dahya}
    \end{subfigure}
    \begin{subfigure}[t]{0.3\textwidth}
        \centering
        \begin{tikzpicture}[node distance=1cm,thick]
            \pgfmathsetmacro\hunit{1}
            \pgfmathsetmacro\vunit{1}
            \pgfmathsetmacro\hscale{1}
            \pgfmathsetmacro\vscale{1}

            \draw [draw=black, line width=0.5pt, fill=none]
                ({-1 * \hscale * \hunit}, {-1 * \vscale * \vunit})
                -- ({2.16 * \hscale * \hunit}, {-1 * \vscale * \vunit})
                -- ({2.16 * \hscale * \hunit}, {1.92 * \vscale * \vunit})
                -- ({-1 * \hscale * \hunit}, {1.92 * \vscale * \vunit})
                -- cycle;

            \draw [draw=none, fill=none]
                ({-1 * \hscale * \hunit}, {1.92 * \vscale * \vunit})
                node[label={[below right, xshift=0pt, yshift=-2pt] \scriptsize $\reals^{2}$}]{};

            \node (O) at (0, 0) {$\bullet$};

            \draw [->, dashed, draw=black, line width=0.5pt, fill=none]
                ({-1 * \hscale * \hunit}, {0 * \vscale * \vunit})
                -- ({2.16 * \hscale * \hunit}, {0 * \vscale * \vunit});

            \draw [->, dashed, draw=black, line width=0.5pt, fill=none]
                ({0 * \hscale * \hunit}, {-1 * \vscale * \vunit})
                -- ({0 * \hscale * \hunit}, {1.92 * \vscale * \vunit});

            \draw [->, draw=black, line width=1pt, fill=none]
                ({0.0 * \hscale * \hunit}, {0.0 * \vscale * \vunit})
                -- ({0.5 * \hscale * \hunit}, {0.0 * \vscale * \vunit});
            \draw [->, draw=black, line width=1pt, fill=none]
                ({0.5 * \hscale * \hunit}, {0.0 * \vscale * \vunit})
                -- ({0.5 * \hscale * \hunit}, {1.0 * \vscale * \vunit});
            \draw [->, draw=black, line width=1pt, fill=none]
                ({0.5 * \hscale * \hunit}, {1.0 * \vscale * \vunit})
                -- ({0.8 * \hscale * \hunit}, {1.0 * \vscale * \vunit});
            \draw [->, draw=black, line width=1pt, fill=none]
                ({0.8 * \hscale * \hunit}, {1.0 * \vscale * \vunit})
                -- ({0.8 * \hscale * \hunit}, {1.6 * \vscale * \vunit});
            \draw [->, draw=black, line width=1pt, fill=none]
                ({0.8 * \hscale * \hunit}, {1.6 * \vscale * \vunit})
                -- ({1.8 * \hscale * \hunit}, {1.6 * \vscale * \vunit});

        \end{tikzpicture}
        \caption{Product of elements $x^{\prime}y^{\prime} \in \FreePower{d}{\realsNonNeg}$}
        \label{fig:free-product:f:sig:article-free-raj-dahya}
    \end{subfigure}
    \caption{%
        Examples of elements
        and operations of the free product
        $\reals^{\freepr d}$
        with $d=2$.\\
        The elements are visualised via continuously parameterised
        paths in $\reals^{d}$ starting in the origin.
    }
    \label{fig:free-product:sig:article-free-raj-dahya}
\end{mdframed}
\end{figure}
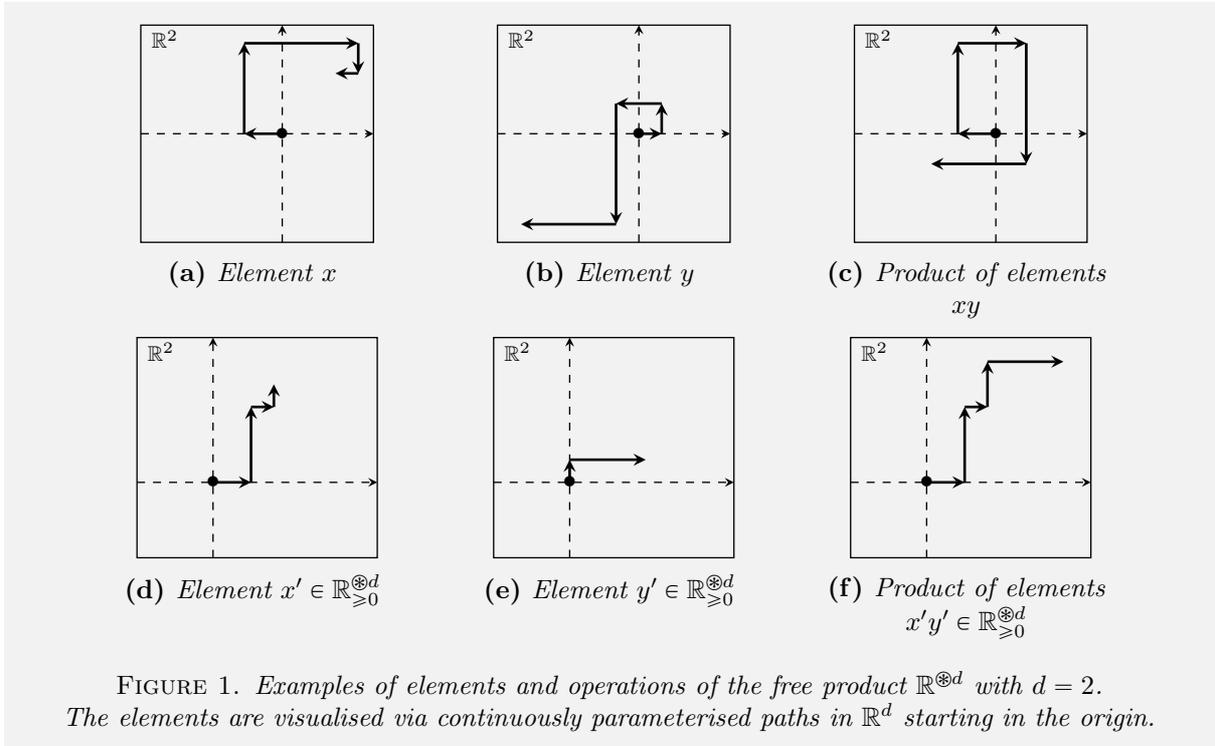


\begin{thm}[Graev, 1950]
\makelabel{thm:free-topological-product-Hausdorff:graev:sig:article-free-raj-dahya}
    If each $G_{i}$ is Hausdorff,
    then so too is the free product.
    In particular
    $\FreePower{I}{\reals}$
    is a Hausdorff topological group
    for each non-empty index set $I$.
\end{thm}

See \cite{Graev1950ArticleFreeProducts} for a proof.
For our purposes, we are particularly concerned with $\FreePower{I}{\reals}$.
We note in particular, that since $\reals$ is a non-trivial, locally compact, connected group,
by
    \cite[Theorem~(vi)]{Morris1975ArticleLocalCompactness},
    \cite{Ordman1974Article},
    \cite[Theorem~3]{Morris1976ArticleLocalInvariance},
the free product
$\FreePower{d}{\reals}$
is Hausdorff (but not locally compact!)
for $d\in\naturals$ with $d \geq 2$.

Before proceeding,
we fix some notation for representations
and show a basic result.
Let $I$ be a non-empty index set,
    $(G_{i},M_{i})$
be pairs of topological groups and submonoids

For any
    $N \in \naturalsZero$
    and
    $\mathbf{i} \coloneqq (i_{k})_{k=1}^{N} \in I^{N}$
denote
    $G_{\mathbf{i}} \coloneqq \prod_{i \in I}G_{i}$
and
    $M_{\mathbf{i}} \coloneqq \prod_{i \in I}M_{i}$.

\begin{conv}
    For $x \in \FreeProduct{i \in I}G_{i}$
    and
    $N\in\naturalsZero$,
    $\mathbf{i} \coloneqq (i_{k})_{k=1}^{N} \in I^{N}$,
    and
    $\mathbf{x} \coloneqq (x_{k})_{k=1}^{N} \in G_{\mathbf{i}}$
    with
    $
        \prod_{k=1}^{N}\iota_{i_{k}}(x_{k}) = x
    $,
    we shall say that
    $(N,\mathbf{i},\mathbf{x})$
    or simply
    $(\mathbf{i},\mathbf{x})$
    is a representation of $x$.
\end{conv}

Under certain conditions,
minimal representations of elements
within certain substructures of a freely presented group
can be obtained via reductions occurring entirely
with the substructure
(\cf \exempli
    \cite[\S{}1]{Chouraqui2009MonoidEmbedding}%
).
The following is a simple instance of this well-known result:

\begin{prop}
\makelabel{prop:free-product-submonoids:sig:article-free-raj-dahya}
    Given any representation
        $(N,\mathbf{i},\mathbf{x})$
    of some element $x \in \FreeProduct{i \in I}M_{i}$
    with $\mathbf{x} \in M_{\mathbf{i}}$,
    there exists
        $m \in \naturals$
    and a sequence
        $
            \{
                (
                    N^{(s)},
                    \mathbf{i}^{(s)},
                    \mathbf{x}^{(s)}
                )
            \}_{s=0}^{m}
        $
    such that

    \begin{kompaktenum}{\bfseries (a)}[\rtab]
        \item\punktlabel{N-seq}
            $(N^{(s)})_{s=0}^{m} \subseteq \naturalsZero$
            is strictly monotone descending;
        \item\punktlabel{ix-seq}
            each $\mathbf{i}^{(s)} \in I^{N^{(s)}}$
            and each $\mathbf{x}^{(s)} \in M_{\mathbf{i}^{(s)}}$;
        \item\punktlabel{repr}
            $(N^{(s)},\mathbf{i}^{(s)},\mathbf{x}^{(s)})$
            is a representation of $x$
            for $s \in \{0,1,\ldots,m\}$;
        \item\punktlabel{init}
            $
                (N^{(0)},\mathbf{i}^{(0)},\mathbf{x}^{(0)})
                = (N,\mathbf{i},\mathbf{x})
            $; and
        \item\punktlabel{min}
            $(N^{(m)},\mathbf{i}^{(m)},\mathbf{x}^{(m)})$
            is a minimal representation of $x$,
            \idest $(i^{(m)}_{k})_{k=1}^{N^{(m)}}$ is bubble-swap free
            and each $x^{(m)}_{k} \in M_{i^{(m)}_{k}} \setminus \{e\}$.
    \end{kompaktenum}

    \nvraum{1.5}

\end{prop}

    \begin{proof}
        We first define a sequence
        $
            \{
                (
                    N^{(s)},
                    \mathbf{i}^{(s)},
                    \mathbf{x}^{(s)}
                )
            \}_{s=0}^{\infty}
        $
        recursively as follows:
        Set
            $
                (N^{(0)},\mathbf{i}^{(0)},\mathbf{x}^{(0)})
                \coloneqq (N,\mathbf{i},\mathbf{x})
            $.
        Given $s \in \naturalsZero$
        and supposing that
            $
                \{
                    (
                        N^{(s')},
                        \mathbf{i}^{(s')},
                        \mathbf{x}^{(s')}
                    )
                \}_{s'=0}^{s}
            $
        have been defined such that
            \punktcref{ix-seq},
            \punktcref{repr},
            and
            \punktcref{init},
        hold for all $s' \in \{0,1,\ldots,s\}$,
        there are three cases:

        \begin{kompaktenum}{\itshape {Case} 1.}
        \item\label{case:1:\beweislabel}
            For some $k_{s} \in \{1,2,\ldots,N^{(s)}\}$
            (\withoutlog we may choose the smallest such $k_{s}$)
            it holds that
                $x^{(s)}_{k_{s}} = 0$.
            In this case
            set $N^{(s+1)} \coloneqq N^{(s)} - 1$,

                \begin{displaymath}
                \everymath={\displaystyle}
                \begin{array}[m]{rclccl}
                    \mathbf{i}^{(s+1)} &= &\Big(
                        \ldots,
                        &i^{(s)}_{k_{s}-1},
                        &i^{(s)}_{k_{s}+1},
                        &\ldots
                    \Big),
                    ~\text{and}\\
                    \mathbf{x}^{(s+1)} &= &\Big(
                        \ldots,
                        &x^{(s)}_{k_{s}-1},
                        &x^{(s)}_{k_{s}+1},
                        &\ldots
                    \Big).
                \end{array}
                \end{displaymath}

        \item\label{case:2:\beweislabel}
            For some $k_{s} \in \{1,2,\ldots,N^{(s)}-1\}$
            (\withoutlog we may choose the smallest such $k_{s}$)
            it holds that
                $i^{(s)}_{k_{s}} = i^{(s)}_{k_{s}+1} \eqqcolon \hat{i}_{s}$.
            In this case
            set $N^{(s+1)} \coloneqq N^{(s)} - 1$,

                \begin{displaymath}
                \everymath={\displaystyle}
                \begin{array}[m]{rclcccl}
                    \mathbf{i}^{(s+1)}
                    &= &\Big(
                            \ldots,
                            &i^{(s)}_{k_{s}-1},
                            &\hat{i}_{s},
                            &i^{(s)}_{k_{s}+2},
                            &\ldots
                        \Big),
                    ~\text{and}\\
                    \mathbf{x}^{(s+1)}
                    &= &\Big(
                            \ldots,
                            &x^{(s)}_{k_{s}-1},
                            &x^{(s)}_{k_{s}} \cdot x^{(s)}_{k_{s}+1},
                            &x^{(s)}_{k_{s}+2},
                            &\ldots
                        \Big).
                \end{array}
                \end{displaymath}

        \item\label{case:3:\beweislabel}
            Otherwise each $x^{(s)}_{k} \in M_{i^{(s)}_{k}} \setminus \{e\}$.
            and $(i^{(s)}_{k})_{k=1}^{N^{(s)}}$ must be bubble-swap free.
            By \punktcref{repr},
                $(N^{(s)},\mathbf{i}^{(s)},\mathbf{x}^{(s)})$
            is already a (the) minimal representation of $x$.
            In this case set
                $
                    (N^{(s+1)},\mathbf{i}^{(s+1)},\mathbf{x}^{(s+1)})
                    \coloneqq (N^{(s)},\mathbf{i}^{(s)},\mathbf{x}^{(s)})
                $.
        \end{kompaktenum}

        It is now a simple matter to verify that
        the sequence
        $
            \{
                (
                    N^{(s)},
                    \mathbf{i}^{(s)},
                    \mathbf{x}^{(s)}
                )
            \}_{s=0}^{\infty}
        $
        satisfies
            \punktcref{N-seq},
            \punktcref{ix-seq},
            \punktcref{repr},
            and
            \punktcref{init},
        Furthermore, the sequence must terminate,
        otherwise only Cases~\ref{case:1:\beweislabel} and \ref{case:2:\beweislabel}
        occur for each $s\in\naturalsZero$,
        in which case
        $N = N^{(0)} > N^{(1)} > \ldots$,
        which is a contradiction.
        Letting $m$ be the smallest $s\in\naturalsZero$
        for which Case~\ref{case:3:\beweislabel} holds,
        we thus obtain a finite sequence
        satisfying
            \punktcref{N-seq},
            \punktcref{ix-seq},
            \punktcref{repr},
            \punktcref{init},
            and
            \punktcref{min}.
    \end{proof}



\subsection[Algebraic formulation of free dilations]{Algebraic formulation of free dilations}
\label{sec:algebra:proof-1st:sig:article-free-raj-dahya}

\firstparagraph
Let $I$ be a non-empty index set and $\HilbertRaum$ a Hilbert space.
Further let
    $(G_{i},M_{i})$
be pairs of topological groups and submonoids
for each $i \in I$.
Consider a family
    $\{U_{i}\}_{i \in I}$
where $U_{i} \in \Repr{G_{i}}{\HilbertRaum}$
for each $i \in I$.
By definition of the free product,
in particular axiom \eqcref{ax:freeproducts:ext:sig:article-free-raj-dahya},
there exists a (necessarily unique)
continuous homomorphism
    ${\mathcal{U} : \FreeProduct{i \in I}{G_{i}} \to (\OpSpaceU{\HilbertRaum},\topSOT)}$,
    \idest
    ${\mathcal{U} \in \Repr{\FreeProduct{i \in I}{G_{i}}}{\HilbertRaum}}$,
such that
    $\mathcal{U} \circ \iota_{i} = U_{i}$
for each $i \in I$.
\emph{Conversely}, starting with
    ${\mathcal{U} \in \Repr{\FreeProduct{i \in I}{G_{i}}}{\HilbertRaum}}$,
it is easy to see that
    $\mathcal{U} \circ \iota_{i} \in \Repr{G_{i}}{\HilbertRaum}$
for each $i \in I$.
There is thus a natural \onetoone\=/correspondence
between families
    $
        \{U_{i}\}_{i \in I}
        \in\prod_{i \in I}\Repr{G_{i}}{\HilbertRaum}
    $
of strongly continuous unitary representations
and strongly continuous unitary representations
in
    $\Repr{\FreeProduct{i \in I}{G_{i}}}{\HilbertRaum}$.
We shall denote the representation corresponding
to $\{U_{i}\}_{i \in I}$
via $\freeprBig_{i \in I}U_{i}$
or if $I = \{1,2,\ldots,d\}$ for some $d\in\naturals$,
we write $\freeprBig_{i=1}^{d}U_{i}$
or ${U_{1} \freepr U_{2} \freepr \ldots \freepr U_{d}}$.

We can ask if something similar holds with
\usesinglequotes{$\topSOT$\=/continuous unitary representations}
replaced by \usesinglequotes{$\topSOT$\=/continuous contractive semigroups}.
The \textbf{algebraic part} of this is straightforward:
Consider a family
    $\{T_{i}\}_{i \in I}$
where each
    ${T_{i} : M_{i} \to \BoundedOps{\HilbertRaum}}$
is a (not necessarily continuous) contractive (\resp isometric \resp unitary) semigroup
over $M_{i}$ on $\HilbertRaum$.
First observe the following:

\begin{prop}
\makelabel{prop:free-product-well-defined:sig:article-free-raj-dahya}
    Let $x \in \FreeProduct{i \in I}{M_{i}}$.
    Given any representation
        $(N,\mathbf{i},\mathbf{x})$
    of $x$
    with $\mathbf{x} \in M_{\mathbf{i}}$,
    the operators

    \begin{restoremargins}
    \begin{equation}
    \label{eq:free-product-well-defined:sig:article-free-raj-dahya}
        \mathcal{T}_{\mathbf{i},\mathbf{x}}(x)
            \coloneqq
            \prod_{k=1}^{N}
                T_{i_{k}}(x_{k}).
    \end{equation}
    \end{restoremargins}

    \continueparagraph
    are independent of the choice of representation.
\end{prop}

    \begin{proof}
        By
            \Cref{prop:free-product-submonoids:sig:article-free-raj-dahya},
        a finite sequence
            $
                \{
                    (
                        N^{(s)},
                        \mathbf{i}^{(s)},
                        \mathbf{x}^{(s)}
                    )
                \}_{s=0}^{m}
            $
        exists with properties
            \eqcref{it:N-seq:prop:free-product-submonoids:sig:article-free-raj-dahya},
            \eqcref{it:ix-seq:prop:free-product-submonoids:sig:article-free-raj-dahya},
            \eqcref{it:repr:prop:free-product-submonoids:sig:article-free-raj-dahya},
            \eqcref{it:init:prop:free-product-submonoids:sig:article-free-raj-dahya},
            and
            \eqcref{it:min:prop:free-product-submonoids:sig:article-free-raj-dahya}.
        Note that by
            \eqcref{it:ix-seq:prop:free-product-submonoids:sig:article-free-raj-dahya}
            and
            \eqcref{it:repr:prop:free-product-submonoids:sig:article-free-raj-dahya},
            $(N^{(s)}, \mathbf{i}^{(s)}, \mathbf{x}^{(s)})$
            is a representation of $x$
            with $\mathbf{x}^{(s)} \in M_{\mathbf{i}^{(s)}}$
        and by
            \eqcref{it:init:prop:free-product-submonoids:sig:article-free-raj-dahya}
            and
            \eqcref{it:min:prop:free-product-submonoids:sig:article-free-raj-dahya}
        this sequence iteratively reduces the representation
            $
                (N,\mathbf{i},\mathbf{x})
                = (N^{(0)},\mathbf{i}^{(0)},\mathbf{x}^{(0)})
            $
        to the (unique!) minimal representation
            $
                (N',\mathbf{i}',\mathbf{x}')
                \coloneqq (N^{(0)},\mathbf{i}^{(0)},\mathbf{x}^{(0)})
            $
        of $x$.
        Now by that proof,
        for each $s \in \{0,1,\ldots,m-1\}$ either
        \emph{Case~\ref{case:1:prop:free-product-submonoids:sig:article-free-raj-dahya}}
        holds, in which case

            \begin{longeqnarray}
                \mathcal{T}_{\mathbf{i}^{(s)},\mathbf{x}^{(s)}}(x)
                    &\eqcrefoverset{eq:free-product-well-defined:sig:article-free-raj-dahya}{=}
                        &\prod_{k=1}^{k_{s}-1}
                            T_{i^{(s)}_{k}}(x^{(s)}_{k})
                        \:\cdot
                            T_{i^{(s)}_{k_{s}}}(0)
                        \:\cdot
                        \prod_{k_{s}+1}^{N^{(s)}}
                            T_{i^{(s)}_{k}}(x^{(s)}_{k})\\
                    &=
                        &\prod_{k=1}^{k_{s}-1}
                            T_{i^{(s+1)}_{k}}(x^{(s+1)}_{k})
                        \:\cdot
                            \onematrix
                        \:\cdot
                        \prod_{k_{s}+1}^{N^{(s)}}
                            T_{i^{(s+1)}_{k-1}}(x^{(s+1)}_{k-1})\\
                    &=
                        &\prod_{k=1}^{N^{(s+1)}}
                            T_{i^{(s+1)}_{k}}(x^{(s+1)}_{k}),
            \end{longeqnarray}

        \continueparagraph
        or else \emph{Case~\ref{case:2:prop:free-product-submonoids:sig:article-free-raj-dahya}}
        holds, in which case

            \begin{longeqnarray}
                \mathcal{T}_{\mathbf{i}^{(s)},\mathbf{x}^{(s)}}(x)
                    &\eqcrefoverset{eq:free-product-well-defined:sig:article-free-raj-dahya}{=}
                        &\prod_{k=1}^{k_{s}-1}
                            T_{i^{(s)}_{k}}(x^{(s)}_{k})
                        \:\cdot
                            T_{\hat{i}_{s}}(x^{(s)}_{k_{s}})
                            T_{\hat{i}_{s}}(x^{(s)}_{k_{s}+1})
                        \:\cdot
                        \prod_{k_{s}+2}^{N^{(s)}}
                            T_{i^{(s)}_{k}}(x^{(s)}_{k})\\
                    &=
                        &\prod_{k=1}^{k_{s}-1}
                            T_{i^{(s+1)}_{k}}(x^{(s+1)}_{k})
                        \:\cdot
                            T_{\hat{i}_{s}}(x^{(s)}_{k_{s}}x^{(s)}_{k_{s}+1})
                        \:\cdot
                        \prod_{k_{s}+2}^{N^{(s)}}
                            T_{i^{(s+1)}_{k-1}}(x^{(s+1)}_{k-1})\\
                    &=
                        &\prod_{k=1}^{N^{(s+1)}}
                            T_{i^{(s+1)}_{k}}(x^{(s+1)}_{k}),
            \end{longeqnarray}

        So
            $
                \mathcal{T}_{\mathbf{i}^{(s+1)},\mathbf{x}^{(s+1)}}(x)
                =\mathcal{T}_{\mathbf{i}^{(s)},\mathbf{x}^{(s)}}(x)
            $
        for all $s \in \{0,1,\ldots,m-1\}$.
        In particular
            $
                \mathcal{T}_{\mathbf{i},\mathbf{x}}(x)
                = \mathcal{T}_{\mathbf{i}^{(0)},\mathbf{x}^{(0)}}(x)
                = \mathcal{T}_{\mathbf{i}^{(m)},\mathbf{x}^{(m)}}(x)
                = \mathcal{T}_{\mathbf{i}',\mathbf{x}'}(x)
            $.
        Since the minimal representation
            $(N',\mathbf{i}',\mathbf{x}')$
        of $x$ is unique,
        one similarly obtains
            $
                \mathcal{T}_{\mathbf{i}'',\mathbf{x}''}(x)
                = \mathcal{T}_{\mathbf{i}',\mathbf{x}'}(x)
            $
        for any other representation
            $(N'',\mathbf{i}'',\mathbf{x}'')$
        of $x$.
        Hence the products in \eqcref{eq:free-product-well-defined:sig:article-free-raj-dahya}
        depend only on $x$ and not the representation.
    \end{proof}

Continuing the above discussion,
by \Cref{prop:free-product-well-defined:sig:article-free-raj-dahya}
we may thus define
    ${\mathcal{T} : \FreeProduct{i \in I}{M_{i}} \to \BoundedOps{\HilbertRaum}}$
via
    $\mathcal{T}(x) \coloneqq \mathcal{T}_{\mathbf{i},\mathbf{x}}(x)$
for $x \in \FreeProduct{i \in I}{M_{i}}$
and any representation
    $(N,\mathbf{i},\mathbf{x})$
    of $x$
with $\mathbf{x} \in M_{\mathbf{i}}$.
We shall denote this construction
via $\freeprBig_{i \in I}T_{i}$
or if $I = \{1,2,\ldots,d\}$ for some $d\in\naturals$,
we write $\freeprBig_{i=1}^{d}T_{i}$
or ${T_{1} \freepr T_{2} \freepr \ldots \freepr T_{d}}$.
Clearly, $\mathcal{T}$ is a well-defined
contraction/isometry/unitary-valued (but possibly not continuous), homomorphism,%
\footnote{%
    \idest
    $\mathcal{T}(e) = \onematrix$,
    where $e$ is the identity of $\FreeProduct{i \in I}{G_{i}}$
    (which also lies in $\FreeProduct{i \in I}{M_{i}}$),
    and
    $\mathcal{T}(xy) = \mathcal{T}(x)\mathcal{T}(y)$ for $x,y\in\FreeProduct{i \in I}{M_{i}}$.
}
\idest
a \highlightTerm{%
    contractive/isometric/unitary semigroup
    over $\FreeProduct{i \in I}{M_{i}}$ on $\HilbertRaum$%
}.
It clearly holds that $\mathcal{T} \circ \iota_{i}\restr{M_{i}} = T_{i}$
for each $i \in I$.
\emph{Conversely}, given a contractive/isometric/unitary semigroup
$\mathcal{T}$ over $\FreeProduct{i \in I}{M_{i}}$ on $\HilbertRaum$,
one has that
    $
        \{
            T_{i} \coloneqq \mathcal{T} \circ \iota_{i}\restr{M_{i}}
        \}_{i \in I}
    $
is a family of (not necessarily continuous)
contractive/isometric/unitary semigroups on $\HilbertRaum$
and since $\mathcal{T}$ is a homomorphism,
one has
    $
        \freeprBig_{i \in I}T_{i}
        = \freeprBig_{i \in I} (\mathcal{T} \circ \iota_{i}\restr{M_{i}})
        = \mathcal{T}
    $.
There is thus a natural \onetoone\=/correspondence
between families $\{T_{i}\}_{i \in I}$
of contractive/isometric/unitary semigroups
(where each $T_{i}$ is a contractive/isometric/unitary semigroup over $M_{i}$ on $\HilbertRaum$)
and contractive/isometric/unitary semigroups over $\FreeProduct{i \in I}{M_{i}}$ on $\HilbertRaum$.
To extend this to a \textbf{topological correspondence},
we make use of free unitary dilations.

\begin{lemm}
\makelabel{lemm:correspondence-free-product-continuity:sig:article-free-raj-dahya}
    Let $I$ be a non-empty index set and $\HilbertRaum$ a Hilbert space.
    Let further
        $(G_{i},M_{i}) \in \DilatableGroupMonoids$
    for each $i \in I$.%
    \footnoteref{ft:all-semigroups-have-a-dilation:pair:sec:algebra:proof-1st:sig:article-free-raj-dahya}
    Let $\{T_{i}\}_{i \in I}$ be a family of operator-valued functions
    with each $T_{i}$ being a contractive/isometric/unitary semigroup
    over $M_{i}$ on $\HilbertRaum$.
    And consider the corresponding
    contractive/isometric/unitary semigroup
        $\mathcal{T} \coloneqq \freeprBig_{i \in I}T_{i}$
    over $\FreeProduct{i \in I}{M_{i}}$ on $\HilbertRaum$.
    Then $\mathcal{T}$ is $\topSOT$\=/continuous if and only if each of the $T_{i}$ are.
\end{lemm}

\footnotetext[ft:all-semigroups-have-a-dilation:pair:sec:algebra:proof-1st:sig:article-free-raj-dahya]{%
    \idest any $\topSOT$\=/continuous contractive semigroup over $M_{i}$
    admits a dilation to an $\topSOT$\=/continuous unitary representation of $G_{i}$
    for each $i \in I$
    (see \S{}\ref{sec:introduction:notation:sig:article-free-raj-dahya}).
}

    \begin{proof}
        The \usesinglequotes{only if}-direction is clear,
        since $T_{i} = \mathcal{T} \circ \iota_{i}\restr{M_{i}}$
        and
            ${
                \iota_{i}\restr{M_{i}}
                : M_{i} \to \FreeProduct{j \in I}{M_{j}}
            }$
        is $\topSOT$\=/continuous for each $i \in I$.
        Towards the \usesinglequotes{if}-direction,
        suppose that each of the $T_{i}$
        are $\topSOT$\=/continuous.
        By the abstract version of the free dilation result
        (see \Cref{thm:free-dilation:abstract:sig:article-free-raj-dahya}
        and \Cref{rem:free-dilation:abstract:repr-version:sig:article-free-raj-dahya}),
        there exists a free unitary dilation
            $(\HilbertRaum^{\prime},r,\{U_{i}\}_{i \in I})$
        of $(\HilbertRaum,\{T_{i}\}_{i \in I})$,
        where
            $
                \{U_{i}\}_{i \in I}
                \in
                \prod_{i \in I}\Repr{G_{i}}{\HilbertRaum^{\prime}}
            $
        are $\topSOT$\=/continuous unitary representations.
        By the above discussion,
            $\{U_{i}\}_{i \in I}$
            corresponds to a unique $\topSOT$\=/continuous representation
            $
                \mathcal{U}
                \coloneqq \freeprBig_{i \in I}U_{i}
                \in \Repr{\FreeProduct{i \in I}{G_{i}}}{\HilbertRaum^{\prime}}
            $.
        Applying the correspondences to the dilation,
        it is now routine to see that
            \eqcref{eq:free-dilation:abstract:sig:article-free-raj-dahya}
        can be expressed as
            $\mathcal{T}(x) = r^{\ast}\:\mathcal{U}(x)\:r$
        for $x \in \FreeProduct{i \in I}{M_{i}}$.
        Since $\mathcal{U}$ is $\topSOT$\=/continuous,
        so too is $\mathcal{T}$.
    \end{proof}

\begin{rem}
    Considering \exempli the case $(G_{i},M_{i}) = (\reals,\realsNonNeg)$,
    it would be of interest to know if there is a dilation-free proof
    of the continuity of semigroups
        $\freeprBig_{i \in I}T_{i}$
    corresponding to families of contractive $\Cnought$\=/semigroups $\{T_{i}\}_{i \in I}$.
\end{rem}

The above proof immediately yields the following
algebraic reformulation of
\Cref{thm:free-dilation:abstract:sig:article-free-raj-dahya}
and thus generalisation of
\Cref{thm:free-dilation:1st:sig:article-free-raj-dahya}:

\begin{highlightboxWithBreaks}
\begin{thm}[Free dilations, free algebraic formulation]
\makelabel{thm:free-dilation:algebraic:sig:article-free-raj-dahya}
    Let $I$ be a non-empty index set and $\HilbertRaum$ a Hilbert space.
    Let further
        $(G_{i},M_{i}) \in \DilatableGroupMonoids$
    for each $i \in I$.%
    \footnoteref{ft:all-semigroups-have-a-dilation:pair:sec:algebra:proof-1st:sig:article-free-raj-dahya}
    For every $\topSOT$\=/continuous contractive semigroup
        $\mathcal{T}$ over $\FreeProduct{i \in I}{M_{i}}$ on $\HilbertRaum$
    (\idest an $\topSOT$\=/continuous contraction-valued homomorphism),
    there exists
        a Hilbert space $\HilbertRaum^{\prime}$,
        an isometry $r \in \BoundedOps{\HilbertRaum}{\HilbertRaum^{\prime}}$,
        and an $\topSOT$\=/continuous unitary representation
        $\mathcal{U} \in \Repr{\FreeProduct{i \in I}{G_{i}}}{\HilbertRaum^{\prime}}$,
        such that
            $\mathcal{T}(x) = r^{\ast}\:\mathcal{U}(x)\:r$
        for all $x \in \FreeProduct{i \in I}{M_{i}}$.
\end{thm}
\end{highlightboxWithBreaks}

As a direct consequence of
\Cref{thm:free-dilation:algebraic:sig:article-free-raj-dahya}
as well as classical dilation results
(see
    \cite[Theorem~I.4.2 and Theorem~I.8.1]{Nagy1970},
    \cite{Ando1963pairContractions}
    \cite{Slocinski1974},
    and
    \cite[Theorem~2]{Slocinski1982}%
),
we immediately obtain:

\begin{highlightboxWithBreaks}
\begin{cor}
\makelabel{cor:M2-closed-under-free-products:sig:article-free-raj-dahya}
    The class
        $\DilatableGroupMonoids$
    contains
        $
            \{
                (\integers,\naturalsZero),
                (\integers^{2},\naturalsZero^{2}),
                (\reals,\realsNonNeg),
                (\reals^{2},\realsNonNeg^{2})
            \}
        $
    and is closed under free topological products.%
    \footnoteref{ft:products-of-pairs:\beweislabel}
\end{cor}
\end{highlightboxWithBreaks}

\footnotetext[ft:products-of-pairs:\beweislabel]{%
    Here we specifically mean,
    that
        $
            (\FreeProduct{i \in I}{G_{i}},\FreeProduct{i \in I}{M_{i}})
            \in
            \DilatableGroupMonoids
        $
    for
        $\{(G_{i},M_{i})\}_{i \in I} \subseteq \DilatableGroupMonoids$
    and non-empty index sets $I$.
}




\section[Residuality results]{Residuality results}
\label{sec:residuality:sig:article-free-raj-dahya}

\firstparagraph
Free dilations entail various immediate residuality results,
when considering appropriately topologised spaces of semigroups.


\subsection[Weak unitary approximations]{Weak unitary approximations}
\label{sec:residuality:approx:sig:article-free-raj-dahya}

\firstparagraph
In \cite[Theorem~2.1 and Remark~2.3]{Krol2009},
Kr\'{o}l established that contractive $\Cnought$\=/semigroups
on infinite-dimensional Hilbert spaces
could be weakly approximated by unitary ones.
This was generalised to contractive semigroups over submonoids of locally compact groups
via the complete boundedness of certain maps on group \TextCStarAlgs
(see \cite[Theorem~1.18]{Dahya2024approximation}).
This machinery cannot be applied to
semigroups over the submonoid
    $\FreePower{I}{\realsNonNeg}$
of the non-locally compact group
    $\FreePower{I}{\reals}$.
Nonetheless, weak approximability may be obtained
as a corollary of the free dilation theorem.

\begin{highlightboxWithBreaks}
\begin{cor}[Weak unitary approximations]
\makelabel{cor:approximation-free-semigroup:sig:article-free-raj-dahya}
    Let $\HilbertRaum$ be an infinite dimensional Hilbert space
    and $I$ a non-empty index set with $\card{I} \leq \dim(\HilbertRaum)$.
    Let $\{T_{i}\}_{i \in I}$ be a family of contractive $\Cnought$\=/semigroups on $\HilbertRaum$.
    Then there exists a family
        $(\{U^{(\alpha)}_{i}\}_{i \in I})_{\alpha \in \Lambda}$
    of families of strongly continuous unitary representations of $\reals$ on $\HilbertRaum$,
    such that for all $\xi,\eta \in \HilbertRaum$

    \begin{restoremargins}
    \begin{equation}
    \label{eq:free-unitary-approx:sig:article-free-raj-dahya}
        \sup_{\substack{
            N \in \naturalsZero,\\
            \{i_{k}\}_{k=1}^{N} \subseteq  I,\\
            (t_{k})_{k=1}^{N} \subseteq \realsNonNeg
        }}
            \absLong{\brktLong{
                \Big(
                    \prod_{k=1}^{N}T_{i_{k}}(t_{k})
                    -
                    \prod_{k=1}^{N}U^{(\alpha)}_{i_{k}}(t_{k})
                \Big)
                \xi
            }{
                \eta
            }}
        \underset{\alpha}{\longrightarrow}
        0.
    \end{equation}
    \end{restoremargins}

    \nvraum{1.5}

\end{cor}
\end{highlightboxWithBreaks}

    \begin{proof}
        As per \Cref{lemm:correspondence-free-product-continuity:sig:article-free-raj-dahya}
        $\mathcal{T} \coloneqq \freeprBig_{i \in I}T_{i}$
        is an $\topSOT$\=/continuous contractive semigroup
        with $\mathcal{T} \circ \iota_{i} = T_{i}$
        for each $i \in I$.
        By \Cref{thm:free-dilation:algebraic:sig:article-free-raj-dahya},
        there exists
            a Hilbert space $\HilbertRaum^{\prime}$,
            an $\topSOT$\=/continuous unitary representation
            $\mathcal{U} \in \Repr{\FreePower{I}{\reals}}{\HilbertRaum^{\prime}}$,
            and
            an isometry $r \in \BoundedOps{\HilbertRaum}{\HilbertRaum^{\prime}}$,
        such that
            $\mathcal{T}(x) = r^{\ast}\:\mathcal{U}(x)\:r$
        for all $x \in \FreePower{I}{\realsNonNeg}$.
        Now it is a simple exercise
        to observe that $\FreePower{I}{\rationals}$
        is dense in $\FreePower{I}{\reals}$.
        By a simple cardinality computation
        one has
        $
            \card{\FreePower{I}{\rationals}}
            \leq \dim(\HilbertRaum)
        $.%
        \footnote{%
            Due to the representation of the elements in the free product,
            there is a surjection
            from
                $
                    c_{00}(I,\rationals)
                    = \{
                        f \in \FnRm{I}{\rationals}
                        \mid
                        \supp(f)~\text{finite}
                    \}
                $
            to $\FreePower{I}{\rationals}$.
            If $I$ is finite,
            then
                $
                    \card{\FreePower{I}{\rationals}}
                    \leq \card{c_{00}(I,\rationals)}
                    = \card{\rationals^{I}}
                    = \aleph_{0}
                    \leq \dim(\HilbertRaum)
                $.
            Otherwise
                $c_{00}(I,\rationals)$
            is bijectively equivalent to
                $P \times \rationals$,
            where $P \coloneqq \{F \subseteq I \mid F~\text{finite}\}$.
            Since $\card{P} = \card{I}$,
            one has
                $
                    \card{\FreePower{I}{\rationals}}
                    \leq \card{c_{00}(I,\rationals)}
                    = \card{P \times \rationals}
                    \leq \max\{
                        \card{P},
                        \card{\rationals}
                    \}
                    \leq \max\{\card{I}, \aleph_{0}\}
                    \leq \dim(\HilbertRaum)
                $.
        }
        By $\topSOT$\=/continuity of $\mathcal{U}$
        we can thus replace $\HilbertRaum^{\prime}$
        by
            $H \coloneqq \quer{\mathcal{U}(\FreePower{I}{\rationals})\:r\HilbertRaum}$,
        whilst still ensuring
            $\mathcal{T}(\cdot) = r^{\ast}\:\mathcal{U}(\cdot)\:r$
        on $\FreePower{I}{\realsNonNeg}$.
        Noting that
            $\dim(H) = \dim(\HilbertRaum)$,%
        \footnote{%
            Let $B$ be an ONB for $\HilbertRaum$.
            Then
                $
                    \dim(\HilbertRaum)
                    = \dim(r\HilbertRaum)
                    \leq \dim(H)
                    \leq \card{\FreePower{I}{\rationals} \times B}
                    = \max\{
                        \card{\FreePower{I}{\rationals}},
                        \card{B}
                    \}
                    \leq \dim(\HilbertRaum)
                $.
        }
        we may assume \withoutlog that
            $\HilbertRaum^{\prime} = \HilbertRaum$.

        Let
            $P \subseteq \BoundedOps{\HilbertRaum}$
        be the index set consisting of finite-rank projections on $\HilbertRaum$,
        directly ordered by $p \succeq q$ $:\Leftrightarrow$ $\ran(p) \supseteq \ran(q)$.
        For each $p \in P$,
        since $rp$ is a finite-rank partial-isometry,
        there exists some unitary $w_{p}$
        with $w_{p}p = rp$.
        We now consider the net
            $
                (\mathcal{U}^{(p)} \coloneqq w_{p}^{\ast}\mathcal{U}(\cdot)w_{p})_{p \in P}
                \subseteq \Repr{R}{\HilbertRaum}
            $
        of families of $\topSOT$\=/continuous unitary representations
        of $R$ on $\HilbertRaum$
        and let
            $U_{i}^{(p)} \coloneqq \mathcal{U}^{(p)}\circ\iota_{i}$
        for each $i \in I$, $p \in P$.
        Let now $\xi, \eta \in \HilbertRaum$ be arbitrary.
        Let $p_{0} \in P$
        be the projection onto $\linspann\{\xi,\eta\}$
        and consider an arbitrary $p \in P$ with $p \succeq p_{0}$.
        For each $x \in \FreePower{I}{\realsNonNeg}$
        with (not necessarily minimal) representation
            $x = \prod_{k=1}^{N}\iota_{i_{k}}(x_{k})$
        one computes

        \begin{shorteqnarray}
            \brktLong{
                \Big(
                    \prod_{k=1}^{N}
                        T_{i_{k}}(x_{k})
                    -
                    \prod_{k=1}^{N}
                        U^{(p)}_{i_{k}}(x_{k})
                \Big)
                \xi
            }{
                \eta
            }
                &= &\brkt{
                        (
                            \mathcal{T}(x)
                            -
                            \mathcal{U}^{(p)}(x)
                        )
                        \xi
                    }{
                        \eta
                    }\\
                &= &\brkt{\mathcal{T}(x)\xi}{\eta}
                    - \brkt{w_{p}^{\ast}\mathcal{U}(x)w_{p}\:\xi}{\eta}\\
                &= &\brkt{\mathcal{T}(x)\xi}{\eta}
                    - \brkt{\mathcal{U}(x)w_{p}\:p\xi}{w_{p}\:p\eta}\\
                &&\text{since $\xi,\eta \in \ran(p_{0}) \subseteq \ran(p)$}\\
                &= &\brkt{\mathcal{T}(x)\xi}{\eta}
                    - \brkt{\mathcal{U}(x)\:r\xi}{r\eta}\\
                &&\text{since $w_{p}p = rp$}\\
                &= &\brkt{(\mathcal{T}(x) - r^{\ast}\mathcal{U}(x)r)\xi}{\eta}
                = 0,
        \end{shorteqnarray}

        \continueparagraph
        since $(\HilbertRaum,r,\mathcal{U})$ is a dilation of $\mathcal{T}$.
        Hence the claim approximation holds.
    \end{proof}

This result can be further strengthened if the Hilbert space is separable.

\begin{highlightboxWithBreaks}
\begin{cor}[Residuality of unitary semigroups]
\makelabel{cor:residuality-approximation-free-semigroup:sig:article-free-raj-dahya}
    Let $\HilbertRaum$ be a separable infinite dimensional Hilbert space
    and $I$ a non-empty countable index set.
    Consider the space $\RaumX$ of
        $\topSOT$\=/continuous contractive semigroups
        over $\FreePower{I}{\realsNonNeg}$ on $\HilbertRaum$,
    and viewed with the topology $\toplocWOT$
    with the topology ($\toplocWOT$)
    of uniform weak convergence on compact subsets of $\FreePower{I}{\realsNonNeg}$.%
    \footnoteref{ft:kwot-top:sec:residuality:approx:sig:article-free-raj-dahya}
    Then the subspace
        $\RaumX_{u} \subseteq \RaumX$ of unitary semigroups
    is dense $G_{\delta}$ in $(\RaumX,\toplocWOT)$.
\end{cor}
\end{highlightboxWithBreaks}

\footnotetext[ft:kwot-top:sec:residuality:approx:sig:article-free-raj-dahya]{%
    The $\toplocWOT$\=/topology is given by the convergence condition
    ${
        \mathcal{T}^{(\alpha)}
        \underset{\alpha}{\overset{\tinytoplocWOT}{\longrightarrow}}
        \mathcal{T}
    }$
    if and only if
    ${
        \sup_{x \in K}
            \abs{
            \brkt{
                (
                    \mathcal{T}^{(\alpha)}(x)
                    -
                    \mathcal{T}(x)
                )
                \xi
            }{\eta}}
        \underset{\alpha}{\longrightarrow}
        0
    }$
    for all compact $K \subseteq \FreePower{I}{\realsNonNeg}$
    and all $\xi,\eta\in\HilbertRaum$.
}

    \begin{proof}
        The density part follows immediately from \Cref{cor:approximation-free-semigroup:sig:article-free-raj-dahya},
        noting that the map
            $
                \Repr{\reals}{\HilbertRaum}
                \ni
                \mathcal{U}
                \mapsto
                \mathcal{U}\restr{\FreePower{I}{\realsNonNeg}}
                \in
                \RaumX_{u}
            $
        is well-defined (in fact a bijection).
        It remains to demonstrate that $\RaumX_{u}$
        is $G_{\delta}$ in $(\RaumX,\toplocWOT)$.
        To this end first recall
        that $\OpSpaceU{\HilbertRaum}$
        is (dense) $G_{\delta}$
        in  $(\OpSpaceC{\HilbertRaum},\topWOT)$.%
        \footnote{%
            That is, the set of unitary operators on $\HilbertRaum$
            is a dense $G_{\delta}$-subset
            in the space of contractions under the $\topWOT$\=/topology.
            See \exempli \cite[Theorem~2.2]{Eisner2010typicalContraction}.
        }
        As in the proof of \Cref{cor:approximation-free-semigroup:sig:article-free-raj-dahya},
        one has that
            $D \coloneqq \FreePower{I}{\rationalsNonNeg}$
        is countable and dense
        in $\FreePower{I}{\realsNonNeg}$.
        Since
            $
                p_{x} :
                \RaumX
                \ni
                \mathcal{T}
                \mapsto
                \mathcal{T}(x)
                \in
                \OpSpaceC{\HilbertRaum}
            $
        is clearly $\toplocWOT$\=/$\topWOT$ continuous
        for each $x \in \FreePower{I}{\realsNonNeg}$,
        the set

        \begin{displaymath}
            \RaumX^{\ast}
            \coloneqq
            \{
                \mathcal{T} \in \RaumX
                \mid
                \forall{x \in D:~}
                \mathcal{T}(x) \in \OpSpaceU{\HilbertRaum}
            \}
            = \bigcap_{x \in D}
                p_{x}^{-1}(\OpSpaceU{\HilbertRaum})
        \end{displaymath}

        \continueparagraph
        is thus $G_{\delta}$ in $(\RaumX,\toplocWOT)$.
        So it suffices to prove that $\RaumX_{u} = \RaumX^{\ast}$.
        Clearly, $\RaumX_{u} \subseteq \RaumX^{\ast}$.
        For the other inclusion,
        consider an arbitrary $\mathcal{T} \in \RaumX^{\ast}$.
        Let $\{T_{i}\}_{i \in I}$ be the family of
        contractive $\Cnought$\=/semigroups corresponding to $\mathcal{T}$
        as per \Cref{lemm:correspondence-free-product-continuity:sig:article-free-raj-dahya}.
        Let $i \in I$ be arbitrary.
        By construction of $\RaumX^{\ast}$
        one has that
            $T_{i}(t) = \mathcal{T}(\iota_{i}(t)) \in \OpSpaceU{\HilbertRaum}$
        for each $t \in \rationalsNonNeg$.
        Note further that
            $
                T_{i}(\cdot)^{\ast}
                =
                \{T_{i}(t)^{\ast}\}_{t \in \realsNonNeg}
            $
        is weakly and thus strongly continuous
        (see \exempli
            \cite[Theorem~I.5.8]{EngelNagel2000semigroupTextBook},
            \cite[Theorem~9.3.1 and Theorem~10.2.1--3]{Hillephillips1957faAndSg}%
        ).
        Since
            $T_{i}$, $T_{i}(\cdot)^{\ast}$,
        are strongly continuous
        and unitary valued on the dense subspace
            $\rationalsNonNeg\subseteq\realsNonNeg$,
        one concludes that
            $T_{i}$
        is in fact unitary-valued on all of $\realsNonNeg$.
        Applying \Cref{lemm:correspondence-free-product-continuity:sig:article-free-raj-dahya} again,
        we thus have that $\mathcal{T}$ is unitary,
        \idest $\mathcal{T} \in \RaumX_{u}$.
    \end{proof}

\begin{rem}
    Recall that $\FreePower{I}{\reals}$ is not locally compact if $\card{I} > 1$.
    So it is not apparent whether $(\RaumX,\toplocWOT)$ is a Baire space.
    To prove this, by \cite[Lemma 3.7]{Dahya2022weakproblem}
    it would suffice to prove that the dense subspace
        $(\RaumX_{u},\toplocWOT)$
    itself is a Baire space.
\end{rem}



\subsection[Residual non-commutativity of free dilations]{Residual non-commutativity of free dilations}
\label{sec:residuality:non-commute:sig:article-free-raj-dahya}

\firstparagraph
In both discrete- and continuous-time
the commutativity of free dilations
necessitates the commutativity
of the underling family of contractions \resp contractive semigroups.
We shall show that (and how often) the converse fails.
In the following let $d \in \naturals$ and $\HilbertRaum$ a Hilbert space.


\paragraph{Discrete-time setting:}
Consider the space $X$ of all commuting $d$-tuples
    $\{S_{i}\}_{i=1}^{d}$
of contractions on $\HilbertRaum$
endowed with the \highlightTerm{weak polynomial} topology ($\topPW$).%
\footnote{%
    The $\topPW$\=/topology is given by the convergence condition
    ${
        \{S^{(\alpha)}_{i}\}_{i=1}^{d}
        \underset{\alpha}{\overset{\tinytopPW}{\longrightarrow}}
        \{S_{i}\}_{i=1}^{d}
    }$
    if and only if
    ${
        \prod_{i=1}^{d}(S^{(\alpha)}_{i})^{n_{i}}
        \underset{\alpha}{\overset{\tinytopWOT}{\longrightarrow}}
        \prod_{i=1}^{d}S^{n_{i}}_{i}
    }$
    for all $\mathbf{n} = (n_{i})_{i=1}^{d} \in \naturalsZero^{d}$.
}
By \cite[Remark~1.1]{Dahya2022complmetrproblem} and
\cite[Theorem~4.2]{Dahya2022complmetrproblem},
\cf also \cite[Theorem~4.1]{Eisnermaitrai2013typicalOperators},
$(X,\topPW)$ is a Polish space
if $\HilbertRaum$ is separable.

Say that $\{S_{i}\}_{i=1}^{d} \in X$
has \highlightTerm{strictly non-commuting free dilations}
just in case none of its
(discrete-time)
free unitary dilations
consist of commuting unitaries.
And we let $X^{\ast} \subseteq X$
be the subset of such families.

\begin{prop}
    If $d \leq 2$, then $X^{\ast}=\emptyset$.
    If $d \geq 3$ and $\HilbertRaum$ is infinite dimensional,
    then $X^{\ast} \neq \emptyset$.
    If furthermore $\HilbertRaum$ is separable,
    then $X^{\ast}$ is residual in $(X,\topPW)$,
    \idest
    almost all
    (in the Baire-category sense)
    commuting $d$-tuples of contractions
    have strictly non-commuting free dilations.
\end{prop}

\begin{proof}
    Let $\{S_{i}\}_{i=1}^{d} \in X$ be arbitrary.
    Suppose that at least one free unitary dilation
        $(\HilbertRaum^{\prime},r,\{V_{i}\}_{i=1}^{d})$
        of
        $(\HilbertRaum,\{S_{i}\}_{i=1}^{d})$
    exists,
    for which the unitaries commute.
    By inspecting \eqcref{eq:free-dilation:discrete-time:sig:article-free-raj-dahya},
    the free dilation is clearly a power dilation.
    Conversely, if
        $\{T_{i}\}_{i=1}^{d}$
    admits a power dilation
        $(\HilbertRaum^{\prime},r,\{U_{i}\}_{i=1}^{d})$,
    then by commutativity
    \eqcref{eq:free-dilation:discrete-time:sig:article-free-raj-dahya}
    clearly holds
    and the dilation is also a free unitary dilation.
    Thus $\{S_{i}\}_{i=1}^{d} \in X^{\ast}$
    if and only if $\{S_{i}\}_{i=1}^{d}$
    admits no power dilation.
    By the dilation results of Sz.-Nagy \cite[Theorem~I.4.2]{Nagy1970}
    and And\^{o} \cite{Ando1963pairContractions},
    as well as Parrott's counterexample \cite[\S{}3]{Parrott1970counterExamplesDilation}
    and the results in
    \cite[Corollary~1.7~b)]{Dahya2024interpolation}
    the three claims immediately follow.
\end{proof}



\paragraph{Continuous-time setting:}
Consider the space $X$ of all commuting $d$-tuples
    $\{T_{i}\}_{i=1}^{d}$
of contractive $\Cnought$\=/semigroups on $\HilbertRaum$
endowed with the topology ($\toplocWOT$)
of uniform weak convergence on compact subsets of $\realsNonNeg^{d}$.%
\footnote{%
    This topology is provided by \Cref{defn:k-wot-top:sig:article-free-raj-dahya}
    applied to the semigroups over $\realsNonNeg^{d}$
    corresponding to such commuting families.
}
By \cite[Corollary~4.4]{Dahya2022complmetrproblem},
$(X,\toplocWOT)$ is a Polish space
if $\HilbertRaum$ is separable.

Say that $\{T_{i}\}_{i=1}^{d} \in X$
has \highlightTerm{strictly non-commuting free dilations}
just in case none of its
(continuous-time)
free unitary dilations
consist of commuting unitary representations.
And we let $X^{\ast} \subseteq X$
be the subset of such families.

\begin{prop}
    If $d \leq 2$, then $X^{\ast}=\emptyset$.
    If $d \geq 3$ and $\HilbertRaum$ is infinite dimensional,
    then $X^{\ast} \neq \emptyset$.
    If furthermore $\HilbertRaum$ is separable,
    then $X^{\ast}$ is residual in $(X,\toplocWOT)$,
    \idest almost all
    (in the Baire-category sense)
    commuting $d$-tuples
    of contractive $\Cnought$\=/semigroups
    have strictly non-commuting free dilations.
\end{prop}

\begin{proof}
    Let $\{T_{i}\}_{i=1}^{d} \in X$ be arbitrary.
    Suppose that at least one free unitary dilation
        $(\HilbertRaum^{\prime},r,\{U_{i}\}_{i=1}^{d})$
        of
        $(\HilbertRaum,\{T_{i}\}_{i=1}^{d})$
    exists,
    for which the unitary representations commute.
    By inspecting \eqcref{eq:free-dilation:cts-time:sig:article-free-raj-dahya}
    the free dilation is clearly a simultaneous unitary dilation.
    Conversely, if
        $\{T_{i}\}_{i=1}^{d}$
    admits a simultaneous unitary dilation
        $(\HilbertRaum^{\prime},r,\{U_{i}\}_{i=1}^{d})$,
    then by commutativity
    \eqcref{eq:free-dilation:cts-time:sig:article-free-raj-dahya}
    clearly holds and
    the dilation is also a free unitary dilation.
    Thus $\{T_{i}\}_{i=1}^{d} \in X^{\ast}$
    if and only if $\{T_{i}\}_{i=1}^{d}$
    admits no simultaneous unitary dilation.
    By the dilation results of Sz.-Nagy \cite[Theorem~I.8.1]{Nagy1970}
    and S\l{}oci\'{n}ski \cite{%
        Slocinski1974,%
        Slocinski1982%
    },
    and the counterexamples and results in
    \cite[Theorem~1.5 and Corollary~1.7~b')]{Dahya2024interpolation},
    the three claims immediately follow.
\end{proof}





\section[Applications to time-dependent evolutions]{Applications to time-dependent evolutions}
\label{sec:applications:sig:article-free-raj-dahya}

\firstparagraph
Throughout this section we let $\interval \subseteq \reals$ denote a connected subset,
$\Delta_{\reals} \coloneqq \{(t,s) \in \reals^{2} \mid t \geq s\}$,
and $\Delta \coloneqq \Delta_{\reals} \cap \interval^{2}$.
The \Second free dilation theorem (\Cref{thm:free-dilation:2nd:sig:article-free-raj-dahya})
and the tools we used to develop this
in \S{}\ref{sec:result-concrete:trotter-kato:sig:article-free-raj-dahya}
were concerned with continuously indexed families of semigroups,
\exempli
    $
        \Big\{
            T_{\tau}\coloneqq\{T_{\tau}(t)\}_{t\in\realsNonNeg}
        \Big\}_{\tau \in \interval}
    $,
where each $T_{\tau}(\cdot)$ is a (contractive) $\Cnought$\=/semigroup
over some common Banach space $\BanachRaum$.
As indicated in \S{}\ref{sec:introduction:motivation:sig:article-free-raj-dahya},
a natural context in which such objects are studied are time-dependent systems
in which product expressions of the form

\begin{restoremargins}
\begin{equation}
\label{eq:continuous-products:motivation:sig:article-free-raj-dahya}
    T_{\tau_{n}}(\delta\tau_{n})
    \cdot
    \ldots
    \cdot
        T_{\tau_{2}}(\delta\tau_{2})
    \cdot
        T_{\tau_{1}}(\delta\tau_{1})
\end{equation}
\end{restoremargins}

\continueparagraph
naturally arise, where
$\tau_{0} \leq \tau_{1} \leq \ldots \leq \tau_{n}$
is a partition of some subinterval $[s,\:t]\subseteq\interval$
and each $\delta\tau_{k} = \tau_{k} - \tau_{k-1}$.
A further natural class of product expressions can be found in the quantum setting
via alternations

\begin{restoremargins}
\begin{equation}
\label{eq:continuous-alternating-products:motivation:sig:article-free-raj-dahya}
    P_{\tau_{n}}
    \cdot U(\delta\tau_{n})
    \cdot
    \ldots
    \cdot P_{\tau_{2}}
    \cdot U\:(\delta\tau_{2})
    \cdot P_{\tau_{1}}
    \cdot U(\delta\tau_{1})
\end{equation}
\end{restoremargins}

\continueparagraph
of a continuous-time \usesinglequotes{unitary} evolution $\{U(t)\}_{t \in \realsNonNeg}$
interrupted by \usesinglequotes{measurements} at discrete time points
arising from a parameterised family of
\usesinglequotes{monitoring operators} $\{P_{\tau}\}_{\tau\in\interval}$.
Note that whilst the typical setting of such systems are operations defined on \TextCStarAlgs,
for our purposes we shall work with operators defined on Banach spaces.%
\footnote{%
    Recall that surjective isometries constitute a generalisation
    of unitaries to the Banach space settings.
}
Products of the above kind are studied in
so called \highlightTerm{collision models}
    \cite{%
        RybarFilippovZiman2012Article,%
        HarwoodBrunelliSerafini2023Article%
    }
as well as
\highlightTerm{continuously monitored} quantum systems
    \cite{%
        ScottMilburn2000Article,%
        WisemanMilburn2010Book,%
        Jacobs2014Book%
    },
which trace their origins back to
\highlightTerm{quantum Zeno dynamics} (QZD)
    \cite{%
        BeskowNilsson1966Article,%
        ExnerIchinose2005Article,%
        ExnerIchinose2006Inproceedings,%
        FacchiLigabo2010Article,%
        HermanShaydulinSun2023Article%
    }.
In typical treatments of QZD, the monitoring component involves repeated use of the same measurement apparatus.
In our setting we allow the monitoring operators to vary with time,
which generalise dynamics subject to \highlightTerm{adiabatic change}
(\cf
    \cite{Kitano1997Article}%
).
Yet another use of such alternations
is the application of interleaving gate operations in quantum computing
as a means to mitigate against decoherence of computational states
(see \exempli
    \cite{ZhangSouzaBrandao2014Article}%
).

Observe that the key qualitative difference between
\eqcref{eq:continuous-products:motivation:sig:article-free-raj-dahya}
and
\eqcref{eq:continuous-alternating-products:motivation:sig:article-free-raj-dahya}
is that in the latter the time indices are decoupled from the time durations.
The first goal of this section is to formalise frameworks
to work with both kinds of processes
as well as the problem of obtaining limits
of the above product expressions.
And our main goal is to demonstrate
how the tools of free dilation
can be leveraged to translate
the more classically defined time-dependent systems of evolution
to continuously monitored processes.


\subsection[Background]{Background}
\label{sec:applications:background:sig:article-free-raj-dahya}

\firstparagraph
We first provide a natural context in which the product expressions
in \eqcref{eq:continuous-products:motivation:sig:article-free-raj-dahya} arise.
Kato and later more formally Howland and Evans
(\cf
    \cite{Kato1953Article},
    \cite[\S{}1]{Howland1974Article},
    \cite[Definition~1.4]{Evans1976ArticlePerturb}%
)
introduced semigroup theoretic means
to solve time-dependent PDEs of the form

\begin{shorteqnarray}
    \left\{
    \begin{array}[m]{rcl}
        u^{\prime}(t) &= &A_{t}u(t) + g(t),\quad t \in \interval,\:t \geq s,\\
        u(s) &= &\xi,
    \end{array}
    \right.
\end{shorteqnarray}

\continueparagraph
where $s \in \interval$,
each $A_{\tau}$ is the generator of
some $\Cnought$\=/semigroup
        $T_{\tau} = \{T_{\tau}(t)\}_{t\in\realsNonNeg}$
on a common Banach space $\BanachRaum$,
    $\xi\in\opDomain{A_{s}} \subseteq \BanachRaum$,
and
    ${g : \interval \to \BanachRaum}$
is a forcing term.
Under certain conditions on the family
    $\{A_{\tau}\}_{\tau \in \interval}$,
\usesinglequotes{mild solutions}
can be determined via

\begin{shorteqnarray}
    u(t) = \mathcal{T}(t,s)\xi + \int_{\tau=s}^{t}\mathcal{T}(t,\tau)g(\tau)\:\dee\tau
\end{shorteqnarray}

\continueparagraph
for $(t,s) \in \interval$,
where
    $
        \{\mathcal{T}(t,s)\}_{(t,s) \in \Delta}
        \subseteq\BoundedOps{\BanachRaum}
    $
is an $\topSOT$\=/continuous family
satisfying

\begin{shorteqnarray}
    \mathcal{T}(t,t) = \onematrix
    \quad\text{and}\quad
    \mathcal{T}(t,s)\mathcal{T}(s,r)
        = \mathcal{T}(t,r)
\end{shorteqnarray}

\continueparagraph
for $(t, s), (s, r) \in \Delta$.
Such two-parameter families are referred to as
\highlightTerm{evolution families} or \highlightTerm{propagators}.
We refer the interested reader to
    \cite[Chapter~5]{Pazy1983Book},
    \cite[\S{}3.1]{ChiconeLatushkin1999Book}
for full details.
If the first requirement is discarded,
we shall refer to
    $\{\mathcal{T}(t,s)\}_{(t,s) \in \Delta}$
as a \highlightTerm{pseudo evolution family}.

One may now ask how the evolution family $\mathcal{T}$
can be extracted constructively from the family
    $\{T_{\tau}\}_{\tau\in\interval}$
of semigroups.
To this end,
consider arbitrary $(t,s) \in \Delta$ and
let $s = \tau_{0} \leq \tau_{1} \leq \ldots \leq \tau_{n} = t$
be an arbitrary finite partition of the interval $[s,\:t] \subseteq \interval$.
Letting $\delta\tau_{k} \coloneqq \tau_{k} - \tau_{k-1}$
for each $k\in\{1,2,\ldots,n\}$, one has

\begin{restoremargins}
\begin{equation}
\label{eq:evolution-family:multi-time:sig:article-free-raj-dahya}
    \mathcal{T}(t,s)
        = \mathcal{T}(\tau_{n},\tau_{n-1})
            \cdot
            \ldots
            \cdot
            \mathcal{T}(\tau_{2},\tau_{1})
            \cdot
            \mathcal{T}(\tau_{1},\tau_{0})
        = \revProd_{k=1}^{n}\mathcal{T}(\tau_{k},\tau_{k} - \delta\tau_{k}),
\end{equation}
\end{restoremargins}

\continueparagraph
where $\revProd$ indicates that
the order of multiplication in the product is reversed.
Intuitively, one may attempt to approximate
    $\mathcal{T}(\tau, \tau-\delta)$
by $T_{\tau}(\delta\tau)$
for small values of $\delta\tau$.
Indeed, under certain assumptions on the generators
(see \exempli
    \cite[\S{}3 and Theorem~3]{Faris1967Article},
    \cite[Theorem~5.3.1]{Pazy1983Book}%
), products of the form \eqcref{eq:continuous-products:motivation:sig:article-free-raj-dahya}
converge in a certain sense to the expression in
\eqcref{eq:evolution-family:multi-time:sig:article-free-raj-dahya}
as the partitions become \usesinglequotes{finer}.
We now have sufficient motivation to
set aside both the context of PDEs and any concerns about generators,
and simply study the convergence
of expressions in \eqcref{eq:continuous-products:motivation:sig:article-free-raj-dahya}.



\subsection[Partition systems]{Partition systems}
\label{sec:applications:partitions:sig:article-free-raj-dahya}

\firstparagraph
Before formalising the main problems,
we first need to fix notions of partitions and limits.
Consider the family of finite subsets of $[0,\:1]$

\begin{displaymath}
    \filterunit
    \coloneqq
    \{
        \Xi\subseteq[0,\:1]
        \mid
        N(\Xi) \coloneqq \card{\Xi} - 1 < \infty,
        ~\min(\Xi) = 0,
        ~\max(\Xi) = 1
    \},
\end{displaymath}

\continueparagraph
directly ordered by reverse-inclusion.
We may identify the elements of $\filterunit$ with \highlightTerm{partitions} of $[0,\:1]$.
Given $\Xi \in \filterunit$,
we let $0=\tau^{\Xi}_{0}<\tau^{\Xi}_{1}<\ldots<\tau^{\Xi}_{N(\Xi)}=1$
be the ordered enumeration of elements in $\Xi$
and set $\delta\tau^{\Xi}_{k} \coloneqq \tau^{\Xi}_{k} - \tau^{\Xi}_{k-1}$
for each $k\in\{1,2,\ldots,N(\Xi)\}$.
It is a simple exercise to verify that
    ${
        \delta\Xi
        \coloneqq
            \max_{k}\delta\tau^{\Xi}_{k}
        \longrightarrow 0
    }$
as the partitions are made finer.

Observe for each $\Xi\in\filterunit$ and $(t,s)\in\Delta_{\reals}$
that $s + (t-s)\Xi$ is a partition of $[s,\:t]$ if $t > s$
and $s + (t-s)\Xi = \{s\}$ if $t = s$.
Thus for convenience we may define

\begin{shorteqnarray}
    \Xi^{(t,s)}
        &\coloneqq
        &s + (t-s)\:\Xi,
    \\
    \tau^{\Xi(t,s)}_{j}
        &\coloneqq
        &s + (t-s)\:\tau^{\Xi}_{j},
        ~\text{and}
    \\
    \delta\tau^{\Xi(t,s)}_{k}
        &\coloneqq
        &\tau^{\Xi(t,s)}_{k}
        - \tau^{\Xi(t,s)}_{k-1}
        = (t-s)\:\delta\tau^{\Xi}_{k}
\end{shorteqnarray}

\continueparagraph
for
    $j,k\in\{0,1,\ldots,N(\Xi)\}$
with $k \geq 1$.
Note that
$
    \tau^{\Xi(t,s)}_{0}
    \leq \tau^{\Xi(t,s)}_{1}
    \leq \ldots
    \leq \tau^{\Xi(t,s)}_{N(\Xi)}
$
is a (possibly non-injective) ordered enumeration
of the elements in $\Xi^{(t,s)}$,
whereby the inequalities are strict
(and thus $\card{\Xi^{(t,s)}} = \card{\Xi} = N(\Xi)$)
provided $t > s$.
Otherwise $\Xi^{(t,s)} = \{s\}$,
whence $\card{\Xi^{(t,s)}} = 1 < 2 \leq \card{\Xi} = N(\Xi)$ if $t = s$.

\begin{defn}
    We shall say that a non-empty subset $\mathbf{P} \subseteq \filterunit$
    is a \highlightTerm{self-similar system of partitions}
    if for each $\Xi \in \mathbf{P}$ and $\alpha \in [0,\:1]$
    there exist $\Gamma_{1},\Gamma_{2},\Gamma_{3}\in\mathbf{P}$
    such that

        \begin{restoremargins}
        \begin{equation}
        \label{eq:self-similarity:sig:article-free-raj-dahya}
            \Gamma_{3} = (\alpha \Gamma_{1}) \cup (\alpha + (1-\alpha)\Gamma_{2})
            \supseteq
                \Xi
        \end{equation}
        \end{restoremargins}

    \continueparagraph
    holds.
    We shall refer to this as \highlightTerm{self-similarity}.
\end{defn}

It is a simple exercise to prove that
self-similarity implies that $\mathbf{P}$
is \highlightTerm{cofinal} in $(\filterunit,\supseteq)$,
\idest for each $\Xi \in \filterunit$,
there exists a \usesinglequotes{finer} partition
    $\Xi' \supseteq \Xi$
with $\Xi'\in\mathbf{P}$.
Cofinality in turn entails that $(\mathbf{P},\supseteq)$
itself is a directed index set
and can thus be used as index sets to study limits.
In particular by the afore mentioned scaling properties,
they can be applied to objects defined on arbitrary time intervals.
The following are simple examples of self-similar systems:

\begin{e.g.}[Homogenous $m$-partitions]
\makelabel{e.g.:m-homog-partitions:sig:article-free-raj-dahya}
    For $m \in \naturals$ say that a partition
    $\Xi\in\filterunit$
    is \highlightTerm{$m$-homogenous}
    if $m \mid N(\Xi)$
    and we let

    \begin{displaymath}
        \filterunit^{(m)}
        \coloneqq
        \{
            \Xi\in\filterunit
            \mid
            \Xi~\text{$m$-homogenous}
        \},
    \end{displaymath}

    \continueparagraph
    which is clearly equal to $\filterunit$
    in case $m=1$.
    Equivalently,
    $
        \filterunit^{(m)}
        = \{\Xi^{(m)} \mid \Xi \in \filterunit\}
    $,
    where for $\Xi \in \filterunit$
    we define the partition
        $
            \Xi^{(m)}
            \coloneqq
            \bigcup_{k=1}^{N(\Xi)}
                \{
                    \tau^{\Xi}_{k-1}
                    + \tfrac{r}{m}\delta\tau^{\Xi}_{k} \mid r \in \{0,1,\ldots,m\}
                \}
        $,
    which corresponds to the union of the uniform partitions of
    each subinterval $[\tau^{\Xi}_{k-1},\:\tau^{\Xi}_{k}]$
    into $m$ pieces.
    It is a straightforward exercise to verify
    that $\filterunit^{(m)}$
    exhibits \eqcref{eq:self-similarity:sig:article-free-raj-dahya}
    and thus forms a self-similar system of partitions,
    which we shall refer to as the system of
    \highlightTerm{homogenous $m$-partitions}.
\end{e.g.}



\subsection[Pre-evolutions]{Pre-evolutions}
\label{sec:applications:pre-evolution:sig:article-free-raj-dahya}

\firstparagraph
We are now in a position to formalise the first main problem:

\begin{problem}[Evolution problem]
\makelabel{problem:pre-evolution:sig:article-free-raj-dahya}
    Let $T = \{T_{\tau}\}_{\tau \in \interval}$
    be a family of contractive $\Cnought$\=/semigroups
    on a Banach space $\BanachRaum$.
    Define

        \begin{restoremargins}
        \begin{equation}
        \label{eq:pre-evolution:sig:article-free-raj-dahya}
            \mathcal{T}^{\Xi}(t, s)
                \coloneqq
                    \revProd_{k=1}^{N(\Xi)}
                        T_{\tau^{\Xi(t,s)}_{k}}(\delta \tau^{\Xi(t,s)}_{k})
        \end{equation}
        \end{restoremargins}

    \continueparagraph
    for $(t,s) \in \Delta$, $\Xi \in \filterunit$.
    Let $\mathbf{P} \subseteq \filterunit$ be a self-similar system of partitions.
    Does an operator-valued function
    $\mathcal{T} = \{\mathcal{T}(t,s)\}_{(t,s) \in \Delta}$
    exist such that

        \begin{restoremargins}
        \begin{equation}
        \label{eq:pre-evolution:lim:sig:article-free-raj-dahya}
            \mathcal{T}^{\Xi}
            \underset{\Xi}{\overset{\tinytoplocSOT}{\longrightarrow}}
            \mathcal{T}
        \end{equation}
        \end{restoremargins}

    \continueparagraph
    holds, where the limit is computed over $\Xi \in \mathbf{P}$?
\end{problem}

Given the relation to product expressions \eqcref{eq:evolution-family:multi-time:sig:article-free-raj-dahya}
discussed in \S{}\ref{sec:applications:background:sig:article-free-raj-dahya},
we shall refer to the data
    $(\BanachRaum, \mathbf{P}, T)$
as a \highlightTerm{pre-evolution}.
Concretely, the construction of each
    $\mathcal{T}^{\Xi}(t, s)$
admits the following interpretation:
For each $k$ from $1$ to $N(\Xi)$
the system undergoes continuous-time evolution
according to $T_{\tau^{\Xi(t,s)}_{k}}$ for the duration of
    $[\tau^{\Xi(t,s)}_{k-1},\:\tau^{\Xi(t,s)}_{k})$.
Strictly speaking, we have applied right-adapted definitions.
A natural alternative would be to
apply $T_{\tau^{\Xi(t,s)}_{k-1}}$
for the duration of
    $[\tau^{\Xi(t,s)}_{k-1},\:\tau^{\Xi(t,s)}_{k})$
for each $k \in \{1,2,\ldots,N(\Xi)\}$.
For demonstrative purposes and simplicity
we shall confine ourselves to the variant presented above.

By the following result
(which will be proved in \S{}\ref{sec:applications:results:sig:article-free-raj-dahya}),
solutions to the evolution problem
provide us with general means to construct to evolution families:

\begin{prop}
\makelabel{prop:solution-to-pre-evolution-problem-is-evolution-family:sig:article-free-raj-dahya}
    Suppose $(\BanachRaum, \mathbf{P}, T)$
    is a pre-evolution in which $\{T_{\tau}\}_{\tau \in \interval}$
    is $\toplocSOT$\=/continuous in the index set.
    Suppose further that $\interval$ is compact.
    If the corresponding evolution problem has a positive solution,
    then $\{\mathcal{T}(t,s)\}_{(t,s) \in \Delta}$
    defined in \eqcref{eq:pre-evolution:lim:sig:article-free-raj-dahya}
    constitutes an evolution family.
\end{prop}



\subsection[Continuously monitored processes]{Continuously monitored processes}
\label{sec:applications:monitoring:sig:article-free-raj-dahya}

\firstparagraph
The following class of problems is inspired by physical phenomena
mentioned at the start of this section.

\begin{problem}[Monitoring problem]
\makelabel{problem:monitoring:sig:article-free-raj-dahya}
    Let
        $\mathcal{X} = \{X_{\tau}\}_{\tau \in \interval}$
    be an $\topSOT$\=/continuous
    family of contractions on a Banach space $\BanachRaum$.
    And let $T$ be a contractive $\Cnought$\=/semigroup on $\BanachRaum$.
    Define

        \begin{restoremargins}
        \begin{equation}
        \label{eq:monitoring:sig:article-free-raj-dahya}
            (\mathcal{X} \ltimes T)^{\Xi}(t,s)
                \coloneqq
                    \revProd_{k=1}^{N(\Xi)}
                        X_{\tau^{\Xi}_{k}}
                        \:T(\delta \tau^{\Xi}_{k})
        \end{equation}
        \end{restoremargins}

    \continueparagraph
    for $(t,s) \in \Delta$, $\Xi \in \filterunit$.
    Let $\mathbf{P} \subseteq \filterunit$ be a self-similar system of partitions.
    Does an operator-valued function
        $
            \mathcal{X} \ltimes T
            = \{(\mathcal{X} \ltimes T)(t,s)\}_{(t,s) \in \Delta}
        $
    exist such that

        \begin{restoremargins}
        \begin{equation}
        \label{eq:monitoring:lim:sig:article-free-raj-dahya}
            (\mathcal{X} \ltimes T)^{\Xi}
            \:\underset{\Xi}{\overset{\tinytoplocSOT}{\longrightarrow}}
            \:\mathcal{X} \ltimes T
        \end{equation}
        \end{restoremargins}

    \continueparagraph
    holds, where the limit is computed over $\Xi \in \mathbf{P}$?
\end{problem}

We refer to the data
    $(\BanachRaum, \mathbf{P}, \mathcal{X}, T)$
as a \highlightTerm{continuous monitoring}
or \highlightTerm{continuously monitored process}
and operators in the family $\mathcal{X}$
as \highlightTerm{monitoring operators}.
If $T = U(\cdot)\restr{\realsNonNeg}$
for some $\topSOT$\=/continuous representation via surjective isometries
    $U \in \Repr{\reals}{\BanachRaum}$,
we may refer to the evolution as
    $(\BanachRaum, \mathbf{P}, \mathcal{X}, U)$
instead.

If we inspect the expressions in \eqcref{eq:monitoring:sig:article-free-raj-dahya}
under the special case of $(t,s) \in \Delta$ with $t=s$,
the convergence in \eqcref{eq:monitoring:lim:sig:article-free-raj-dahya}
necessitates the convergence of integer powers
of the monitoring operators
$\{X_{t}^{N(\Xi)}\}_{\Xi\in\mathbf{P}}$
for each $t\in\interval$.
One such way to achieve this is to impose idempotency conditions.
To this end we make use of the following terminology:
For $m\in\naturals$ and an operator $X$ on a Banach space,
say that $X$ is \highlightTerm{$m$-idempotent}, if $X^{m}X = X$.
By induction, $m$-idempotency of $X$
implies that $X^{m + s} = X^{s}$ for all $s \in \naturals$,
and thus
    $X^{km} = X^{m}$
    and
    $X^{km + 1} = X$
for all $k \in \naturals$.
Note that a $1$-idempotent operator
is just an idempotent one, also referred to as a \highlightTerm{projection}.
And simple examples of $m$-idempotent operators
include permutations (or reflections in the case of $m=2$).
Using this terminology we obtain the following:

\begin{prop}
\makelabel{prop:solution-to-quantum-problem-is-evolution-family:sig:article-free-raj-dahya}
    Suppose that \Cref{problem:monitoring:sig:article-free-raj-dahya}
    has a positive solution.
    If each $X_{\tau}$ is $m$-idempotent
    and $\mathbf{P} \subseteq \filterunit^{(m)}$
    for some $m\in\naturals$,%
    \footnote{%
        See \Cref{e.g.:m-homog-partitions:sig:article-free-raj-dahya}.
        Note that families of idempotent operators
        trivially satisfy this requirement for $m=1$.
    }
    then the family
    $\{(\mathcal{X} \ltimes T)(t,s)\}_{(t,s) \in \Delta}$
    defined in \eqcref{eq:monitoring:lim:sig:article-free-raj-dahya}
    constitutes a pseudo evolution family
    with
        $
            (\mathcal{X} \ltimes T)(t,t)
            = X_{t}^{m}
        $
    for $t \in \interval$.
\end{prop}

    \begin{proof}
        Set $\mathcal{T} \coloneqq \mathcal{X} \ltimes T$.
        We first establish the final property:
        Letting $t \in \interval$ and $\Xi \in \mathbf{P} \subseteq \filterunit^{(m)}$,
        one computes
            $
                \mathcal{T}^{\Xi}(t, t)
                = \revProd_{k=1}^{N(\Xi_{t})}
                    X_{\tau^{\Xi(t,t)}_{t}}
                    T(\delta\tau^{\Xi(t,t)}_{k})
                = \revProd_{k=1}^{N(\Xi_{t})}
                    X_{t}
                    T(0)
                = X_{t}^{N(\Xi_{t})}
            $,
        which reduces to $X_{t}^{m}$
        by virtue of $m$-idempotence of $X_{t}$ and since $m \mid N(\Xi)$.
        Taking limits, one trivially obtains $\mathcal{T}(t,t) = X_{t}^{m}$.
        We now prove that
            $\{\mathcal{T}(t,s)\}_{(t,s) \in \Delta}$
        is a pseudo evolution family.

        \paragraph{Algebraic property:}
        Let $(t,s),(s,r)\in\Delta$.
        By the assumed $\toplocSOT$-convergence,
        the strong convergences
            ${\mathcal{T}^{\Xi}(t, s)\underset{\Xi}{\overset{\tinytopSOT}{\longrightarrow}}\mathcal{T}(t,s)}$,
            ${\mathcal{T}^{\Xi}(s, r)\underset{\Xi}{\overset{\tinytopSOT}{\longrightarrow}}\mathcal{T}(s,r)}$,
            and
            ${\mathcal{T}^{\Xi}(t, r)\underset{\Xi}{\overset{\tinytopSOT}{\longrightarrow}}\mathcal{T}(t,r)}$
        hold.
        Since the product expressions in \eqcref{eq:monitoring:sig:article-free-raj-dahya}
        are uniformly bounded and multiplication is $\topSOT$\=/continuous,
        in order to establish that
        $\mathcal{T}(t,r) = \mathcal{T}(t,s)\mathcal{T}(s,r)$,
        it suffices to prove for each $\Gamma_{1},\Gamma_{2},\Gamma_{3}\in\mathbf{P}$
        that some $\Gamma'_{1},\Gamma'_{2},\Gamma'_{3}\in\mathbf{P}$ exist
        such that $\Gamma'_{i} \supseteq \Gamma_{i}$ for each $i \in \{1,2,3\}$
        and

        \begin{restoremargins}
        \begin{equation}
        \label{eq:0:\beweislabel}
            \mathcal{T}^{\Gamma_{3}}(t,r)
            =
                \mathcal{T}^{\Gamma'_{2}}(t,s)
                \mathcal{T}^{\Gamma'_{1}}(s,r)
        \end{equation}
        \end{restoremargins}

        \continueparagraph
        holds.
        Let now $\alpha \in [0,\:1]$ be such,
        that $s - r = \alpha \cdot (t-r)$
        and $t - s = (1-\alpha) \cdot (t-r)$.
        Let
            $
                \Xi
                \coloneqq
                \Gamma_{3} \cup \alpha \Gamma_{1} \cup (\alpha + (1-\alpha)\Gamma_{2})
                \in \filterunit
            $.
        By confinality of $\mathbf{P}$ in $\filterunit$,
        there exists $\Xi' \in \mathbf{P}$ with $\Xi' \supseteq \Xi$.
        And by the self-similarity property \eqcref{eq:self-similarity:sig:article-free-raj-dahya}
        of $\mathbf{P}$,
        there exists
            $\Gamma'_{1},\Gamma'_{2},\Gamma'_{3}\in\mathbf{P}$
        such that

        \begin{shorteqnarray}
            \Gamma'_{3}
            = \alpha \Gamma'_{1} \cup (\alpha + (1-\alpha)\Gamma'_{2})
            \supseteq \Xi'.
        \end{shorteqnarray}

        Since $\Xi' \supseteq \Xi$,
        by construction one has
            $
                \alpha \Gamma'_{1} \cup (\alpha + (1-\alpha)\Gamma'_{2})
                \supseteq
                \alpha \Gamma_{1} \cup (\alpha + (1-\alpha)\Gamma_{2})
            $.
        If $\alpha \neq 0$, this implies $\Gamma'_{1} \supseteq \Gamma_{1}$,
        otherwise we may replace $\Gamma'_{1}$ by $\Gamma_{1}$
        without affecting $\Gamma'_{3}$.
        Either way, we may assume that $\Gamma'_{1} \supseteq \Gamma_{1}$.
        Similarly, we may assume that $\Gamma'_{2} \supseteq \Gamma_{2}$.
        Since
            $\Gamma'_{3} \supseteq \Xi'
            \supseteq \Xi
            \supseteq \Gamma_{3}$,
        all that remains is to demonstrate
        \eqcref{eq:0:\beweislabel}.
        This is a straightforward consequence of
        $\Gamma'_{3} = \alpha \Gamma'_{1} \cup (\alpha + (1-\alpha)\Gamma'_{2})$
        and the construction of $\mathcal{T}^{\Xi}$ for each $\Xi$
        in \eqcref{eq:monitoring:sig:article-free-raj-dahya}.

        \paragraph{Strong\=/continuity:}
        Fix arbitrary
            $(t_{0},s_{0}) \in \Delta$,
            $\xi\in\BanachRaum$,
            and
            $\eps > 0$.
        Let $K \subseteq \Delta$ be a compact neighbourhood of $(t_{0},s_{0})$.
        Since the convergence in \eqcref{eq:pre-evolution:sig:article-free-raj-dahya}
        is uniform on compact subsets of $\Delta$,
        there exists
            $\Xi \in \mathbf{P}$
        such that
            $
                \norm{
                    (
                        \mathcal{T}^{\Xi'}(t, s)
                        -
                        \mathcal{T}(t, s)
                    )
                    \xi
                }
                <
                \eps/4
            $
        for all $(t,s) \in K$
        and all $\Xi' \in \mathbf{P}$
        with $\Xi' \supseteq \Xi$.
        In particular,

            \begin{restoremargins}
            \begin{equation}
            \label{eq:1:\beweislabel}
                \norm{
                    (
                        \mathcal{T}(t, s)
                        -
                        \mathcal{T}(t_{0},s_{0})
                    )
                    \xi
                }
                < 2\cdot\tfrac{\eps}{4}
                    + \norm{
                        (
                            \mathcal{T}^{\Xi}(t, s)
                            -
                            \mathcal{T}^{\Xi}(t_{0}, s_{0})
                        )
                        \xi
                    }
            \end{equation}
            \end{restoremargins}

        \continueparagraph
        for all $(t,s) \in K$.
        By the construction of $\mathcal{T}^{\Xi}$
        in \eqcref{eq:monitoring:sig:article-free-raj-dahya}
        one has

            \begin{restoremargins}
            \begin{equation}
            \label{eq:2:\beweislabel}
            \everymath={\displaystyle}
            \begin{array}[m]{rcl}
                \mathcal{T}^{\Xi}(t, s)
                    &= &\revProd_{k=1}^{N(\Xi)}
                            X_{\tau^{\Xi(t,s)}_{k}}
                            \:
                            T(\delta\tau^{\Xi(t,s)}_{k})\\
                    &= &\revProd_{k=1}^{N(\Xi)}
                            X_{s + (t-s)\tau^{\Xi}_{k}}
                            \:
                            T((t-s)\delta\tau^{\Xi}_{k}),
            \end{array}
            \end{equation}
            \end{restoremargins}

        \continueparagraph
        for each $(t,s) \in K$.
        Since
            ${\interval \ni \tau \mapsto X_{\tau} \in \BoundedOps{\BanachRaum}}$
        and
            ${\realsNonNeg \ni t \mapsto T(t) \in \BoundedOps{\BanachRaum}}$
        are uniformly bounded $\topSOT$\=/continuous maps,
        the product expression in \eqcref{eq:2:\beweislabel} implies that
            ${\Delta \ni (t,s) \mapsto \mathcal{T}^{\Xi}(t,s) \in \BoundedOps{\BanachRaum}}$
        is $\topSOT$\=/continuous.
        By \eqcref{eq:1:\beweislabel} we can thus find a neighbourhood $W \subseteq K$
        of $(t_{0},s_{0})$,
        such that
            $
                \norm{
                    (
                        \mathcal{T}(t, s)
                        -
                        \mathcal{T}(t_{0},s_{0})
                    )
                    \xi
                }
                < \eps
            $
        for $(t, s) \in W$.
        This establishes the $\topSOT$\=/continuity of $\mathcal{T}$.
    \end{proof}

In \Cref{prop:solution-to-quantum-problem-is-evolution-family:sig:article-free-raj-dahya}
we saw the advantage of restricting the monitoring operators $X_{\tau}$
to $m$-idempotent operators.
More generally, we may confine $\mathcal{X}$ and $T$
to certain operator classes
to obtain natural subclasses of \Cref{problem:monitoring:sig:article-free-raj-dahya}.
Say that a set
    $\mathcal{P}$
of bounded (necessarily idempotent) operators
on a Banach space is a \highlightTerm{family of measurements}
if $PQ = Q$ for all $P, Q \in \mathcal{P}$.%
\footnote{%
    The idea here being,
    that once a \usesinglequotes{measurement} is made,
    no further measurement can affect the result
    without first disturbing the state of the system.
}
And say that a family of idempotents $\mathcal{P}$
is \highlightTerm{passive} \wrt to $T$
if
    $\{P\:T(t)\:P\}_{t \in \realsNonNeg}$
satisfies the semigroup law for each $P\in\mathcal{P}$.
Note that if each $X_{\tau} = P$
for a single idempotent $P$ which is passive \wrt $T$,
then the expressions in \eqcref{eq:monitoring:sig:article-free-raj-dahya}
trivially reduce to $
    (\mathcal{X} \ltimes T)^{\Xi}(t,s)
    = P\:T(t-s)\:P
$
for all $\Xi\in\mathbf{P}$, $(t,s)\in\Delta$,
and the monitoring problem is trivially solved.
Hence \Cref{problem:monitoring:sig:article-free-raj-dahya}
only becomes interesting
when considering monitoring operators
subject to temporal change.

\begin{defn}
\makelabel{defn:monitoring-classes:sig:article-free-raj-dahya}
    A continuous monitoring
        $(\BanachRaum, \mathbf{P}, \mathcal{X}, T)$
    shall be called
    a \highlightTerm{continuously monitored quantum process},
    if $T = U(\cdot)\restr{\realsNonNeg}$
    for some $\topSOT$\=/continuous
    representation of $\reals$ on $\BanachRaum$
    via surjective isometries.
    We call a continuously monitored process
    $(\BanachRaum, \mathbf{P}, \mathcal{X}, T)$
    a process continuously monitored via

    \begin{enumerate}[label={\upshape\bfseries \arabic*.}]
    \item\punktlabel{1}
        \highlightTerm{via $m$-idempotents (\resp projections)},
        if $\mathcal{X} = \mathcal{W} = \{W_{\tau}\}_{\tau\in\interval}$
        is a family of $m$-idempotent (\resp idempotent) contractions;

    \item
        \highlightTerm{via measurements},
        if $\mathcal{X} = \mathcal{P} = \{P_{\tau}\}_{\tau\in\interval}$
        is a family of contractive measurements; and

    \item
        \highlightTerm{via passive projections/measurements},
        if $\mathcal{X} = \mathcal{P} = (P_{\tau})_{\tau\in\interval}$
        is a family of contractive projections/measurements
        which is passive
        \wrt $T$.
    \end{enumerate}

    If in \punktcref{1} the $W_{\tau}$ are furthermore surjective isometries,
    we shall speak of \highlightTerm{cycles} (\resp \highlightTerm{reflections} if $m = 2$)
    instead of $m$-idempotents.
\end{defn}



\subsection[Examples]{Examples}
\label{sec:applications:examples:sig:article-free-raj-dahya}

\firstparagraph
Before proceeding with our main result,
we present examples of processes continuously monitored via $m$-idempotents
and consider \Cref{problem:monitoring:sig:article-free-raj-dahya} in each case.

\begin{e.g.}[Chernoff approximations]
\makelabel{e.g.:chernoff:sig:article-free-raj-dahya}
    Consider contractive $\Cnought$\=/semigroups,
        $T_{0}$, $T_{1}$, \ldots, $T_{m-1}$
    on a Banach space $\BanachRaum$
    with generators $A_{0}$, $A_{1}$, \ldots, $A_{m-1}$
    respectively
    for some $m \in \naturals$.
    Let $\mathcal{S}_{m}$ denote the set of permutations
    on $\{0,1,\ldots,m-1\}$
    and define
        $
            Q_{\sigma}(\tau)
            \coloneqq
                T_{\sigma(m-1)}(\tau)
                \cdot
                \ldots
                \cdot
                T_{\sigma(1)}(\tau)
                \cdot
                T_{\sigma(0)}(\tau)
        $
    for $\tau\in\realsNonNeg$
    and $\sigma \in \mathcal{S}_{m}$.
    We now appeal to the Smolyanov--Weizsäcker--Wittich generalisation
    of the \highlightTerm{Chernoff approximation theorem}.
    To this end we assume that the closure
        $A$ of $\sum_{i=0}^{m-1}A_{i}$
    generates a $\Cnought$\=/semigroup $T$ on $\BanachRaum$
    and further that each $Q_{\sigma}$ is \highlightTerm{proper}
    in the sense that

    \begin{shorteqnarray}
        \tfrac{Q_{\sigma}(\tau) - \onematrix}{\tau}
        \eta
        \longrightarrow
        A\eta
    \end{shorteqnarray}

    \continueparagraph
    for all $
        \eta
        \in \{T(a)\xi \mid \xi\in\opDomain{A},~a > 0\}
        \eqqcolon \mathcal{D}_{+}
    $.
    Let $\sigma \in \mathcal{S}_{m}$.
    By \cite[Proposition~3]{SmolyanovWeizsaeckerWittich2000chernoff},
    the convergence

        \begin{restoremargins}
        \begin{equation}
        \label{eq:0:\beweislabel}
            \revProd_{k=1}^{N(\Xi)}
                Q_{\sigma}(\delta\tau^{\Xi(t,s)}_{k})
            \underset{\Xi}{\longrightarrow}
                T(t-s)
        \end{equation}
        \end{restoremargins}

    \continueparagraph
    holds strongly for $(t,s) \in \Delta$,
    where the limit is taken over $\Xi \in \filterunit$.
    In fact,
    one can reason that this convergence
    is uniform on compact subsets of $\Delta$.

    To see this, fix
        $\eps > 0$,
        a non-empty compact subset $K \subseteq \Delta$,
        and
        $T(\alpha)\xi \in \mathcal{D}_{+}$,
    where $\alpha > 0$ and $\xi \in \opDomain{A}$.
    View $\opDomain{A}$ with the graph norm $\norm{\cdot}_{A}$ for $A$,
    which makes it a Banach space as $A$ is closed.
    We now define
    $
        R(\tau)
        \coloneq
            \tfrac{Q_{\sigma}(\tau) - \onematrix}{\tau}
            -
            \tfrac{T(\tau) - \onematrix}{\tau}
    $
    for $\tau > 0$ and $R(0) \coloneqq \zeromatrix$.
    Clearly each $R(\tau)$ constitutes a bounded linear operator
    from $(\opDomain{A},\norm{\cdot}_{A})$ to $\BanachRaum$.
    Moreover, by the properness assumption,
        $R(\cdot)T(\alpha)\xi$
    is continuous,
    making
        $\{R(\tau)T(\alpha)\xi \mid \tau\in[0,\:1]\} \subseteq \BanachRaum$
    a compact and thus norm-bounded set.
    By the uniform boundedness principle,
    it follows that
        $
            \{R(\tau)T(\alpha) \mid \tau\in[0,\:1]\}
            \subseteq
            \BoundedOps{\opDomain{A}}{\BanachRaum}
        $
    is bounded in operator norm by some $C \in (0,\:\infty)$.

    By $\topSOT$\=/continuity of $T$
    and $T$\=/invariance of $\opDomain{A}$,
    the map
        $
            a
            \mapsto
            (T(a)\xi,A\:T(a)\xi)
            = (T(a)\xi,T(a)A\:\xi)
        $
    is continuous.
    So setting
        $
            h_{\max}
            \coloneqq
            \max (\{1\} \cup \{t - s \mid (t, s) \in K\})
            \in (0,\:\infty)
        $,
    this makes
        $
            \{T(a)\xi\mid a\in[\alpha,\:\alpha + h_{\max}]\}
            \subseteq
            \opDomain{A}
        $
    compact \wrt $\norm{\cdot}_{A}$.
    In particular,
    a finite set
        $F \subseteq [\alpha,\:\alpha + h_{\max}]$
    exists,
    such that for each $a \in [\alpha,\:\alpha + h_{\max}]$
    some $a' \in [\alpha,\:\alpha + h_{\max}]$ exists
    with
        $
            \norm{(T(a) - T(a'))\xi}_{A}
            < \tfrac{\eps}{2 h_{\max} C}
        $
    and thus

    \begin{shorteqnarray}
        \norm{R(\tau)T(a)\xi}
            &\leq
                &\norm{R(\tau)T(a')\xi}
                +
                \norm{R(\tau)} \cdot \norm{(T(a) - T(a'))\xi}_{A}
            \\
            &\leq
                &\norm{R(\tau)T(a')\xi}
                +
                C \cdot \tfrac{\eps}{2 h_{\max} C}
    \end{shorteqnarray}

    \continueparagraph
    for $\tau\in[0,\:1]$.
    By virtue of the afore mentioned continuity
    of $R(\cdot)\eta$ for each $\eta \in \mathcal{D}_{+}$,
    one may find an appropriately small value $\tau_{0} \in (0,\:1]$,
    such that

    \begin{restoremargins}
    \begin{equation}
    \label{eq:1:\beweislabel}
        \sup_{a \in [\alpha,\:\alpha + h_{\max}]}
            \norm{R(\tau)T(a)\xi}
            \leq
                \max_{a \in F}\norm{R(\tau)T(a)\xi} + \tfrac{\eps}{2 h_{\max}}
            \leq
                \tfrac{\eps}{h_{\max}}
    \end{equation}
    \end{restoremargins}

    \continueparagraph
    for all $\tau \in [0,\:\tau_{0})$.

    Recalling that
        ${
            \delta\Xi
            = \max_{k} \delta\tau^{\Xi}_{k}
            \longrightarrow 0
        }$
    as the partitions $\Xi \in \filterunit$ get finer,
    we may fix an arbitrary $\Xi_{0} \in \filterunit$ with
        $h_{\max}\delta\Xi_{0} < \tau_{0}$.
    Consider arbitrary
        $\Xi \in \filterunit$ with $\Xi \supseteq \Xi_{0}$
        and $(t,s) \in K$.
    Using telescoping expressions,
    one may compute

        \begin{shorteqnarray}
            &&\normLong{
                \Big(
                    \revProd_{k=1}^{N(\Xi)}
                        Q_{\sigma}(\delta\tau^{\Xi(t,s)}_{k})
                    -
                    T(t-s)
                \Big)
                \:T(\alpha)\xi
            }\\
                &=
                    &\normLong{
                        \sum_{k=1}^{N(\Xi)}
                            \revProd_{i=k+1}^{N(\Xi)}
                            Q_{\sigma}(\delta\tau^{\Xi(t,s)}_{i})
                            \cdot
                            (
                                Q_{\sigma}(\delta\tau^{\Xi(t,s)}_{k})
                                -
                                T(\delta\tau^{\Xi(t,s)}_{k})
                            )
                            \cdot
                            \revProd_{i=1}^{k-1}
                                T(\delta\tau^{\Xi(t,s)}_{i})
                            \:T(\alpha)\xi
                    }
                    \\
                &\leq
                    &\sum_{k=1}^{N(\Xi)}
                        \underbrace{
                            \revProd_{i=k+1}^{N(\Xi)}
                                \norm{Q_{\sigma}(\delta\tau^{\Xi(t,s)}_{i})}
                        }_{\leq 1}
                        \cdot
                        \normLong{
                            \underbrace{
                                (
                                    Q_{\sigma}(\delta\tau^{\Xi(t,s)}_{k})
                                    -
                                    T(\delta\tau^{\Xi(t,s)}_{k})
                                )
                            }_{
                                =
                                \delta\tau^{\Xi(t,s)}_{k}
                                R(\delta\tau^{\Xi(t,s)}_{k})
                            }
                            \:T(\alpha + \sum_{i=1}^{k-1}\delta\tau^{\Xi(t,s)}_{i})
                            \:\xi
                        }
                    \\
                &\overset{(\ast)}{\leq}
                    &h_{\max}
                    \cdot
                    \sup_{\tau \in [0,\:\tau_{0})}
                    \sup_{a \in [\alpha,\:\alpha + h_{\max}]}
                        \norm{R(\tau)\:T(a)\:\xi}
                \eqcrefoverset{eq:1:\beweislabel}{\leq}
                    h_{\max} \cdot \tfrac{\eps}{h_{\max}}
                = \eps,
        \end{shorteqnarray}

    \continueparagraph
    whereby the simplifications in ($\ast$) are obtained by observing that
    $
        \delta\tau^{\Xi(t,s)}_{k}
        = (t-s)\delta\tau^{\Xi}_{k}
        \leq h_{\max} \delta\Xi
        \leq h_{\max} \delta\Xi_{0}
        < \tau_{0}
    $
    and
    $
        \sum_{i=1}^{k-1}\delta\tau^{\Xi(t,s)}_{i}
        \leq t-s
        \leq h_{\max}
    $
    for each $k \in \{1,2,\ldots,N(\Xi)\}$.
    Finally, since the product expressions are contractions,
    the desired uniform strong convergence
    in \eqcref{eq:0:\beweislabel} follows
    by density of $\mathcal{D}_{+}$ in $\BanachRaum$.

    Consider now the Banach space
        $
            \tilde{\BanachRaum}
            \coloneqq
            \complex^{m} \otimes_{\text{alg}} \BanachRaum
        $.
    Let $\tilde{T}$ be the contractive $\Cnought$\=/semigroup
    on $\tilde{\BanachRaum}$
    defined by
        $
            \tilde{T}(\cdot)
            \coloneqq
                \sum_{i=0}^{m-1}
                    \ElementaryMatrix{i}{i} \otimes T_{i}(m\:\cdot)
        $.
    Let
        $
            \mathcal{W}
            \coloneqq
            \{W_{\tau}\}_{\tau\in\interval}
        $
    be the family of isometric surjections
    given by
        $
            W_{\tau}
            \coloneqq
            W
            \coloneqq
                \sum_{i=0}^{m-1}
                    \ElementaryMatrix{(i+1) \mod m}{i}
                    \otimes
                    \onematrix
            \in \BoundedOps{\HilbertRaum}
        $
    for each $\tau \in \realsNonNeg$.
    Observe that $W_{\tau}^{m} = \onematrix$
    and thus $W_{\tau}$ is $m$-idempotent
    for each $\tau \in \interval$.

    Let
        $(t,s) \in \Delta$ and $\Xi \in \filterunit^{(m)}$
    be arbitrary.
    By construction of homogenous $m$-partitions
    (see \Cref{e.g.:m-homog-partitions:sig:article-free-raj-dahya}),
    there exists a partition
        $\Xi_{0} \in \filterunit$
    such that
        $
            \Xi
            = \Xi_{0}^{(m)}
            = \bigcup_{k=1}^{N(\Xi_{0})}
                \{
                    \tau^{\Xi_{0}}_{k-1}
                    +
                    \tfrac{r}{m}
                    \cdot
                    \delta\tau^{\Xi_{0}}_{k}
                    \mid
                    r \in \{0,1,2,\ldots,m\}
                \}
        $.
    In particular,
        $N(\Xi) = m \cdot N(\Xi_{0})$
    and
        $
            \delta\tau^{\Xi(t,s)}_{ml-r}
            = (t-s)\:\delta\tau^{\Xi}_{ml-r}
            = (t-s)\:\delta\tau^{\Xi_{0}}_{l}/m
            = \delta\tau^{\Xi_{0}(t,s)}_{l}/m
        $
    for
        $l \in \{1,2,\ldots,N(\Xi_{0})\}$,
        $r \in \{0,1,\ldots,m-1\}$.
    So

        \begin{shorteqnarray}
            (\mathcal{W} \ltimes \tilde{T})^{\Xi}(t, s)
            &= &\revProd_{k=1}^{N(\Xi)}
                W_{\tau^{\Xi(t,s)}_{k}}
                \tilde{T}(\delta\tau^{\Xi(t,s)}_{k})
                \\
            &= &\revProd_{l=1}^{N(\Xi_{0})}
                \underbrace{
                    \prod_{r=0}^{m-1}
                        W_{\tau^{\Xi(t,s)}_{ml-r}}
                        \:\tilde{T}(\delta \tau^{\Xi(t,s)}_{ml - r})
                }_{
                    = (
                        W \: \tilde{T}(\delta \tau^{\Xi_{0}(t,s)}_{l}/m)
                    )^{m}
                }
                \\
            &= &\revProd_{l=1}^{N(\Xi_{0})}
                \Big(
                    \sum_{i=0}^{m-1}
                        \ElementaryMatrix{(i + 1) \mod m}{i}
                        \otimes
                        T_{i}(m\cdot\delta \tau^{\Xi_{0}(t,s)}_{l}/m)
                \Big)^{m}
                \\
            &=
                &\revProd_{l=1}^{N(\Xi_{0})}
                    \sum_{i = 0}^{m-1}
                        \ElementaryMatrix{i}{i}
                        \otimes
                        \underbrace{
                            \Big(
                                \revProd_{j = 0}^{m-1}
                                    T_{(i + j) \mod m}(\delta \tau^{\Xi_{0}(t,s)}_{k})
                            \Big)
                        }_{
                            = Q_{\sigma_{i}}(\delta \tau^{\Xi_{0}(t,s)}_{l})
                        }
                \\
            &=
                &\sum_{i = 0}^{m-1}
                \revProd_{l=1}^{N(\Xi_{0})}
                    \ElementaryMatrix{i}{i}
                    \otimes
                    Q_{\sigma_{i}}(\delta \tau^{\Xi_{0}(t,s)}_{l}),
                \\
        \end{shorteqnarray}

    \continueparagraph
    where
        $\sigma_{i} \in \mathcal{S}_{m}$
    is the permutation defined by
        $\sigma_{i}(j) \coloneqq (i + j) \mod m$
    for $i,j \in \{0,1,\ldots,m-1\}$.
    Applying the Chernoff approximation \eqcref{eq:0:\beweislabel},
    the above converges strongly to
        $
            \sum_{i = 0}^{m-1}
            \ElementaryMatrix{i}{i} \otimes T(t-s)
        $,
    uniformly for $(t,s)$ on compact subsets of $\Delta$.
    Hence
        $(\tilde{\BanachRaum}, \filterunit^{(m)}, \mathcal{W}, \tilde{T})$
    is a \emph{process continuously monitored via cycles},
    and the monitoring problem in this case
    has a positive solution,
    \viz
        $
            \mathcal{W} \ltimes \tilde{T}
            = \{\onematrix \otimes T(t - s)\}_{(t,s) \in \Delta}
        $.
\end{e.g.}

A sufficient condition to ensure the fulfilment of the properness requirement
in the Smolyanov--Weizsäcker--Wittich result,
is to assume that
    $\opDomain{A} \subseteq \opDomain{A_{i}}$ for each $i\in\{0,1,\ldots,m-1\}$.
To see this, let $\xi \in \opDomain{A}$ and $\sigma \in \mathcal{S}_{m}$.
Setting
    $
        Q_{\sigma,i}(\tau)
        \coloneqq
        \revProd_{j=i+1}^{m-1}
            T_{\sigma(j)}(\tau)
    $
for each $i\in\{0,1,\ldots,m-1\}$,
a simple use of telescoping expressions yields

\begin{shorteqnarray}
    \normLong{
        \Big(
            \tfrac{Q_{\sigma}(\tau) - \onematrix}{\tau}
            -
            A
        \Big)\:\xi
    }
    &=
        &\normLong{
            \sum_{i = 0}^{m-1}
            \Big(
                Q_{\sigma,i}(\tau)
                \cdot
                \tfrac{T_{\sigma(i)}(\tau) - \onematrix}{\tau}
                \xi
                -
                A_{\sigma(i)}
                \xi
            \Big)
        }
        \\
    &\leq
        &\sum_{i = 0}^{m-1}
            \norm{
                Q_{\sigma,i}(\tau)
                -
                \onematrix
            }
            \norm{
                A_{\sigma(i)}
                \xi
            }
        +
            \norm{Q_{\sigma,i}(\tau)}
            \normLong{
                \Big(
                    \tfrac{T_{\sigma(i)}(\tau) - \onematrix}{\tau}
                    -
                    A_{\sigma(i)}
                \Big)
                \xi
            },
\end{shorteqnarray}

\continueparagraph
which converges to $0$ for $\tau \searrow 0$.

\begin{e.g.}[Feynman construction]
\makelabel{e.g.:feynman:sig:article-free-raj-dahya}
    The \highlightTerm{Feynman path integral} famously arises
    as a by-product of the construction
    of a unitary $\Cnought$\=/semigroup $T$
    via Chernoff approximations
    applied to $m = 2$ unitary $\Cnought$\=/semigroups
    $T_{0}$ and $T_{1}$
    on a Hilbert space
        $\HilbertRaum = L^{2}(\reals^{l})$, $l\in\naturals$
    with generators
    $A_{0} = \iunit \kappa \boldsymbol{\Delta}$
    and $A_{1} = -\iunit V$
    respectively,
    where $\kappa > 0$ is a constant,
    $\boldsymbol{\Delta}$ denotes the Laplace operator,
    and ${V : \reals^{l} \to \reals}$
    a Borel measurable function referred to as a \highlightTerm{potential}
    (see \exempli
        \cite[\S{}I.8.13]{Goldstein1985semigroups}%
    ).
    If we consider the special case in which the potential is bounded,
    \idest $V \in L^{\infty}(\reals^{l})$,
    one has that $A \coloneqq A_{0} + A_{1}$
    is a bounded perturbation of a closed operator and thus closed.
    In particular $\opDomain{A} = \opDomain{A_{0}} \subseteq \HilbertRaum = \opDomain{A_{1}}$.
    Hence the above sufficient condition holds
    and we can apply the calculations in \Cref{e.g.:chernoff:sig:article-free-raj-dahya}.
    The semigroup $\tilde{T}$ constructed there
    is clearly a unitary $\Cnought$\=/semigroup
    on the Hilbert space
        $\complex^{2} \otimes \HilbertRaum$
    and thus corresponds to a
    continuous unitary representation
        $\tilde{U} \in \Repr{\reals}{\complex^{2} \otimes \HilbertRaum}$.
    Thus
        $(\tilde{\BanachRaum}, \filterunit^{(2)}, \mathcal{W}, \tilde{U})$
    constitutes a \emph{quantum process continuously monitored via reflections}.
    The positive solution
        $
            \mathcal{W} \ltimes \tilde{U}
            = \{\onematrix \otimes T(t-s)\}_{(t,s) \in \Delta}
        $
    of this problem
    indicates that the solution
    to the Schrödinger equation
    emerges from a process involving
    rapid alternation between unitary evolution
    on the bipartite system
        $L^{2}(\complex^{2} \otimes \reals^{l})$,
    and an involution
    which flips the states on the auxiliary part of the system.
\end{e.g.}

Of greater interest in this section is the case $m=1$,
\idest processes continuously monitored via families of projections.
As motivation we consider the following:

\begin{e.g.}[Wave propagation with absorption]
    Consider the Hilbert space
        $\HilbertRaum \coloneqq L^{2}(\reals^{l})$, $l\in\naturals$.
    Let
        $\{O_{\tau}\}_{\tau\in\interval}$
    be a family of measurable subsets of $\reals^{l}$.
    Assume that this is continuous in the sense
    that the measure theoretic difference between
        $O_{\tau}$ and $O_{\tau_{0}}$
    converges to $0$
    for ${\interval \ni \tau \longrightarrow \tau_{0}}$
    and all $\tau_{0} \in \interval$.
    It is easy to see that the family
    of orthogonal projections
        $\mathcal{P} \coloneqq \{P_{\tau}\}_{\tau\in\interval} \in \BoundedOps{L^{2}(\reals^{l})}$
    defined via
        $
            P_{\tau}\:f
            \coloneqq
            (\einser - \einser_{O_{\tau}})\:f
        $
    for $f \in \HilbertRaum$,
        $\tau\in\interval$,
    is $\topSOT$\=/continuous.
    Further let
        $U \in \Repr{\reals}{\HilbertRaum}$
    be the $\topSOT$\=/unitary representation
    with generator
        $B = \frac{\iunit}{2 m_{0}} \sum_{i=1}^{l}(\tfrac{\da}{\da x_{i}})^{2}$
        (defined on some dense subspace)
    for some constant $m_{0} \in (0,\:\infty)$,
    which defines the propagation of a wave in a vacuum
    and satisfies
        $
            (U(t)\:f)(\mathbf{x})
            = (\tfrac{\iunit 2 \pi t}{m_{0}})^{-n/2}
                \int_{\mathbf{y} \in \reals^{l}}
                    e^{
                        \tfrac{\iunit m_{0}\norm{\mathbf{x} - \mathbf{y}}^{2}}{2 t}
                    }
                    f(\mathbf{y})
                \:\dee\mathbf{y}
        $
    for
        $f \in L^{2}(\reals^{l}) \cap L^{1}(\reals^{l})$,
        $\mathbf{x} \in \reals^{l}$,
        and
        $t \in (0,\:\infty)$
    (see \cite[\S{}I.8.13]{Goldstein1985semigroups}).
    Letting $\mathbf{P}$ be a self-similar system of partitions,
    it follows that
        $(\HilbertRaum, \mathbf{P}, \mathcal{P}, U)$
    constitutes a \emph{quantum process continuously monitored via projections}.
    If this instance of the monitoring problem has a positive solution,
    then the (pseudo) evolution family that arises can be interpreted
    as the propagation of a wave in a vacuum
    containing an absorbing obstacle,
    which itself is subject to temporal variation.
\end{e.g.}



\subsection[Reduction results]{Reduction results}
\label{sec:applications:results:sig:article-free-raj-dahya}

\firstparagraph
We now arrive at the main results of this section.
Our goal is to reduce evolution problems to monitoring problems.
Since pre-evolutions involve (possibly continuously)
indexed families of semigroups,
it seems intuitive that the framework of the \Second free dilation theorem can be applied here.
Indeed, the key ingredient to achieve the following result
is the second diagonalisation in our presentation of the Trotter--Kato theorem
in \S{}\ref{sec:result-concrete:trotter-kato:sig:article-free-raj-dahya}.

\begin{highlightboxWithBreaks}
\begin{lemm}
\makelabel{lemm:reduction:pre-evolution-to-cts-monitored-proc:sig:article-free-raj-dahya}
    Let $(\BanachRaum, \mathbf{P}, T)$ be a pre-evolution.
    Suppose that
        $T = \{T_{\tau}\}_{\tau \in \interval}$
        is $\toplocSOT$\=/continuous
    in the index set
    and that $\interval$ is compact.
    Then a quantum process
        $(\tilde{\BanachRaum}, \mathbf{P}, \mathcal{P}, U)$
    continuously monitored via passive measurements
    exists such that
        the evolution problem
        for $(\BanachRaum, \mathbf{P}, T)$
    reduces to
        the monitoring problem
        for $(\tilde{\BanachRaum}, \mathbf{P}, \mathcal{P}, U)$.
    If both problems are positively solved,
    then the limits satisfy

        \begin{restoremargins}
        \begin{equation}
        \label{eq:dilation:pre-evolution-cts-monitored-proc:sig:article-free-raj-dahya}
        \everymath={\displaystyle}
        \begin{array}[m]{rcl}
            \mathcal{T}(t,s)
                &=
                    &j
                    \:(\mathcal{P} \ltimes U)(t,s)
                    \:r,\\
            (\mathcal{P} \ltimes U)(t,s)
                &= &r\:\mathcal{T}(t,s)\:j_{s}
        \end{array}
        \end{equation}
        \end{restoremargins}

    \continueparagraph
    for all $(t,s) \in \Delta$,
    where
        $r \in \BoundedOps{\BanachRaum}{\tilde{\BanachRaum}}$
        is an isometric embedding
        and $j, j_{s} \in \BoundedOps{\tilde{\BanachRaum}}{\BanachRaum}$
        are surjective isometries
        with $j\:r = j_{s}\:r = \onematrix$
    for $s \in \interval$.
\end{lemm}
\end{highlightboxWithBreaks}

    \begin{proof}
        We first construct a continuously monitored quantum process
            $(\BanachRaum, \mathbf{P}, \mathcal{P}_{\interval}, U_{\interval})$,
        then demonstrate that
        the evolution problem for $(\BanachRaum, \mathbf{P}, T)$
        has a positive solution
        if and only if
        the monitoring problem does.

        \paragraph{Construction of the monitoring problem:}
            Since the family of contractive $\Cnought$\=/semigroups
                $\{T_{\tau}\}_{\tau\in\interval}$
            is assumed to be $\toplocSOT$\=/continuous in its compact index set,
            the \emph{second diagonalisation} (\Cref{prop:second-diagonalisation:sig:article-free-raj-dahya})
            may be applied,
            which yields
                a Banach space $\tilde{\BanachRaum}$,
                an isometric embedding
                    $r \in \BoundedOps{\BanachRaum}{\tilde{\BanachRaum}}$,
                a strongly continuous family
                    $\{j_{\tau}\}_{\tau \in \interval} \in \BoundedOps{\tilde{\BanachRaum}}{\BanachRaum}$
                of surjective isometries,
            and
                an $\topSOT$\=/continuous representation
                $U_{\interval}\in\Repr{\reals}{\tilde{\BanachRaum}}$
                consisting of surjective isometries on $\tilde{\BanachRaum}$,
            such that the dilation in
                \eqcref{eq:second-diagonalisation:dilation:sig:article-free-raj-dahya}
            holds.
            By \Cref{rem:second-diagonalisation:properties:sig:article-free-raj-dahya},
            the contractions in
                $
                    \mathcal{P}_{\interval}
                    \coloneqq
                        \{P_{\tau} \coloneqq r\:j_{\tau}\}_{\tau\in\interval}
                    \subseteq
                        \BoundedOps{\tilde{\BanachRaum}}
                $
            constitute a strongly continuous family of measurements.
            By applying the dilation \eqcref{eq:second-diagonalisation:dilation:sig:article-free-raj-dahya}
            as well as the properties of the embeddings,
            it is a straightforward exercise to see that
                $
                    \{P_{\tau}\:U_{\interval}(t)\:P_{\tau}\}_{t\in\realsNonNeg}
                    = \{r\:T_{\tau}(t)\:j_{\tau}\}_{t\in\realsNonNeg}
                $
            satisfies the semigroup law
            for each $\tau\in\interval$.
            Thus
                $(\tilde{\BanachRaum}, \mathbf{P}, \mathcal{P}_{\interval}, U_{\interval})$
            constitutes a quantum process continuously monitored via passive measurements.

        \paragraph{Basic observations:}
            Fix now some $\hat{\tau} \in \interval$.
            For $\Xi \in \mathbf{P}$ and $(t,s) \in \Delta$ one computes

                \begin{shorteqnarray}
                    \mathcal{T}^{\Xi}(t, s)
                        &\eqcrefoverset{eq:pre-evolution:sig:article-free-raj-dahya}{=}
                            &\revProd_{k=1}^{N(\Xi)}
                                T_{\tau^{\Xi(t,s)}_{k}}(\delta \tau^{\Xi(t,s)}_{k})
                            \\
                        &\eqcrefoverset{eq:second-diagonalisation:dilation:sig:article-free-raj-dahya}{=}
                            &\revProd_{k=1}^{N(\Xi)}
                                j_{\tau^{\Xi(t,s)}_{k}}
                                \:U_{\interval}(\delta \tau^{\Xi(t,s)}_{k})
                                \:r
                            \\
                        &=
                            &\underbrace{
                                j_{\hat{\tau}}
                                \:r
                            }_{=\onematrix}
                            \:\Big(
                                \revProd_{k=1}^{N(\Xi)}
                                    j_{\tau^{\Xi(t,s)}_{k}}
                                    \:U_{\interval}(\delta \tau^{\Xi(t,s)}_{k})
                                    \:r
                            \Big)
                            \\
                        &=
                            &j_{\hat{\tau}}
                            \:\Big(
                                \revProd_{k=1}^{N(\Xi)}
                                    \underbrace{
                                        r
                                        \:j_{\tau^{\Xi(t,s)}_{k}}
                                    }_{
                                        = P_{\tau^{\Xi(t,s)}_{k}}
                                    }
                                    \:U_{\interval}(\delta \tau^{\Xi(t,s)}_{k})
                            \Big)
                            \:r,
                \end{shorteqnarray}

            \continueparagraph
            whence by \eqcref{eq:monitoring:sig:article-free-raj-dahya}

                \begin{restoremargins}
                \begin{equation}
                \label{eq:1:\beweislabel}
                    \mathcal{T}^{\Xi}(t, s)
                        =
                            j_{\hat{\tau}}
                            \:(\mathcal{P}_{\interval} \ltimes U_{\interval})^{\Xi}(t, s)
                            \:r,
                \end{equation}
                \end{restoremargins}

            \continueparagraph
            holds, and thus also

                \begin{restoremargins}
                \begin{equation}
                \label{eq:2:\beweislabel}
                \everymath={\displaystyle}
                \begin{array}[m]{rcl}
                    (\mathcal{P}_{\interval} \ltimes U_{\interval})^{\Xi}(t, s)\:P_{s}
                    &= &P_{\hat{\tau}}\:(\mathcal{P}_{\interval} \ltimes U_{\interval})^{\Xi}(t, s)\:P_{s}\\
                    &= &r\:j_{\hat{\tau}}\:(\mathcal{P}_{\interval} \ltimes U_{\interval})^{\Xi}(t, s)\:r\:j_{s}\\
                    &= &r\:\mathcal{T}^{\Xi}(t, s)\:j_{s},\\
                \end{array}
                \end{equation}
                \end{restoremargins}

            \continueparagraph
            where the first equality holds
            by virtue of the fact that $\mathcal{P}_{\interval}$ is a family of measurements.

        \paragraph{Proof of the reduction:}
            We show that
                the evolution problem for $(\mathbf{P}, T)$ has a positive solution
                if and only if
                the monitoring problem for $(\mathbf{P}, \mathcal{P}_{\interval}, U_{\interval})$ has a positive solution.
            Towards the \usesinglequotes{if}-direction,
            if the $\toplocSOT$\=/limit
                ${
                    \mathcal{P}_{\interval} \ltimes U_{\interval}
                    \coloneqq \lim_{\Xi} (\mathcal{P}_{\interval} \ltimes U_{\interval})^{\Xi}
                }$
            exists,
            then by \eqcref{eq:1:\beweislabel}
            it readily follows that
                $
                    \mathcal{T}^{\Xi}(\cdot, \cdot)
                    = j_{\hat{\tau}}\:(\mathcal{P}_{\interval} \ltimes U_{\interval})^{\Xi}(\cdot, \cdot)\:r
                    \underset{\Xi}{\overset{\tinytoplocSOT}{\longrightarrow}}
                    j_{\hat{\tau}}\:(\mathcal{P}_{\interval} \ltimes U_{\interval})(\cdot, \cdot)\:r
                $.

            Towards the \usesinglequotes{only if}-direction,
            suppose that the uniform strong limit
                ${
                    \mathcal{T}
                    \coloneqq \lim_{\Xi} \mathcal{T}^{\Xi}
                }$
            exists.
            By \eqcref{eq:2:\beweislabel}
            one has that
                $
                    (\mathcal{P}_{\interval} \ltimes U_{\interval})^{\Xi}(t, s)\:P_{s}
                    = r\:\mathcal{T}^{\Xi}(t, s)\:j_{s}
                $
            converges strongly to
                $
                    r\:\mathcal{T}(t,s)\:j_{s}
                $,
            uniformly for $(t,s)$ on compact subsets of $\Delta$.%
            \footnote{%
                For uniform convergence of this product expression,
                we rely on the fact that
                    ${\interval \ni s \mapsto j_{s} \in \BoundedOps{\BanachRaum_{J}}}$
                is $\topSOT$\=/continuous
                and
                    ${\Delta \ni (t,s) \mapsto r\:\mathcal{T}(t,s) \in \BoundedOps{\BanachRaum_{J}}}$
                is uniformly bounded.
            }
            It remains to eliminate the occurrence of $P_{s}$ in this limit.
            Equivalently, it suffices to show that
                ${
                    (\mathcal{P}_{\interval} \ltimes U_{\interval})^{\Xi}(t, s)\:(\onematrix - P_{s})
                    \underset{\Xi}{\overset{\tinytopSOT}{\longrightarrow}}
                    \zeromatrix
                }$
            uniformly for $(t,s)$ on compact subsets of $\Delta$.
            To this end, fix a compact subset $K \subseteq \Delta$.
            Consider arbitrary
            $(t,s) \in K$
            and
            $\Xi \in \mathbf{P}$ with $N(\Xi) \geq 2$.

            Setting
                $\alpha \coloneqq \delta\tau^{\Xi}_{1} \in (0,\:1)$,
                $\Xi_{0} \coloneqq \{0,1\} \in \filterunit$,
                and
                $\Xi^{+} \coloneqq \{\tfrac{\tau^{\Xi}_{k} - \alpha}{1-\alpha}\}_{k=1}^{N(\Xi)} \in \filterunit$,
            one has
                $\Xi = \alpha\Xi_{0} \cup (\alpha + (1-\alpha)\Xi^{+})$.
            Note that $\tau^{\Xi(t,s)}_{1} = s + h\alpha$
            where $h \coloneqq t - s$.
            The product expression in \eqcref{eq:monitoring:sig:article-free-raj-dahya}
            thus reduces to

            \begin{shorteqnarray}
                (\mathcal{P}_{\interval} \ltimes U_{\interval})^{\Xi}(t, s)
                &= &(\mathcal{P}_{\interval} \ltimes U_{\interval})^{\Xi^{+}}(t, \tau^{\Xi(t,s)}_{1})
                        \:P_{\tau^{\Xi(t,s)}_{1}}
                        \:U_{\interval}(\delta \tau^{\Xi(t,s)}_{1})
                    \\
                &= &(\mathcal{P}_{\interval} \ltimes U_{\interval})^{\Xi^{+}}(t, s + h\alpha)
                    \:P_{s + h\alpha}
                    \:U_{\interval}(h\alpha)
            \end{shorteqnarray}

            \continueparagraph
            and since the products in \eqcref{eq:monitoring:sig:article-free-raj-dahya}
            are contractions, one obtains

            \begin{longeqnarray}
                \norm{
                    (\mathcal{P}_{\interval} \ltimes U_{\interval})^{\Xi}(t, s)
                    \:
                    (\onematrix - P_{s})
                    \:\xi
                }
                &\leq &\norm{
                        P_{s + h\alpha}
                        \:U_{\interval}(h\alpha)
                        \:(\onematrix - P_{s})
                        \:\xi
                    }
                \\
                &\leq &\norm{
                        \underbrace{
                            P_{s + h\alpha}
                            \:(\onematrix - P_{s})
                        }_{
                            \overset{(\ast)}{=}
                            P_{s + h\alpha}
                            -
                            P_{s}
                        }
                        \xi
                    }
                    +
                    \norm{
                        P_{s + h\alpha}
                        \:(
                            U_{\interval}(h\alpha)
                            -
                            \onematrix
                        )
                        \:(\onematrix - P_{s})
                        \:\xi
                    }
                \\
                &\leq &\norm{
                        (
                            P_{s + h\alpha}
                            -
                            P_{s}
                        )
                        \:\xi
                    }
                    +
                    \norm{
                        (
                            U_{\interval}(h\alpha)
                            -
                            \onematrix
                        )
                        \:(\onematrix - P_{s})
                        \:\xi
                    }
            \end{longeqnarray}

            \continueparagraph
            for $\xi \in \tilde{\BanachRaum}$,
            where ($\ast$) holds as $\mathcal{P}_{\interval}$ is a family of measurements.
            Setting $h_{\max} \coloneqq \sup_{(t,s) \in K}\abs{t-s} < \infty$,
            the above computation yields

                \begin{shorteqnarray}
                    \sup_{(t,s) \in K}
                        \norm{
                            (\mathcal{P}_{\interval} \ltimes U_{\interval})^{\Xi}(t, s)
                            \:
                            (\onematrix - P_{s})
                            \:\xi
                        }
                    &\leq
                        &\sup_{s \in K}
                        \sup_{h \in [0,\:h_{\max}\delta \tau^{\Xi}_{1}]}
                            \norm{
                                (
                                    P_{s + h}
                                    -
                                    P_{s}
                                )
                                \:\xi
                            }\\
                    &&+
                        \sup_{s \in K}
                        \sup_{h \in [0,\:h_{\max}\delta \tau^{\Xi}_{1}]}
                            \norm{
                                (
                                    U_{\interval}(h)
                                    -
                                    \onematrix
                                )
                                \:(\onematrix - P_{s})
                                \:\xi
                            },
                \end{shorteqnarray}

            \continueparagraph
            which,
            by the $\topSOT$\=/continuity
            and uniform boundedness
            of
                $U_{\interval}$ and $\mathcal{P}_{\interval}$,
            converges to $0$
            as $\Xi\in\mathbf{P}$
            is made finer,
            since
                ${\tau^{\Xi}_{1} \underset{\Xi}{\longrightarrow} 0}$
            and ${\delta \tau^{\Xi}_{1} \underset{\Xi}{\longrightarrow} 0}$.
            The sought after uniform strong convergence
            thus immediately follows.

        \paragraph{Banach space dilation:}
            If the limits in both evolution problems hold,
            then by
                \eqcref{eq:1:\beweislabel}
                and \eqcref{eq:2:\beweislabel}
                as well as the subsequent arguments,
            it follows that
                \eqcref{eq:dilation:pre-evolution-cts-monitored-proc:sig:article-free-raj-dahya}
            holds with
                $j \coloneqq j_{\hat{\tau}}$.
    \end{proof}

\begin{rem}
\makelabel{rem:monitoring-to-pre-evolution:sig:article-free-raj-dahya}
    If the class of evolution problems is widened
    to allow for semigroups $T_{\tau}$ with $T_{\tau}(0)$ being a projection,
    then a counterpart to \Cref{lemm:reduction:pre-evolution-to-cts-monitored-proc:sig:article-free-raj-dahya}
    can be obtained,
    reducing processes continuously monitored via passive measurements to pre-evolutions.
\end{rem}

As an immediate application,
\Cref{lemm:reduction:pre-evolution-to-cts-monitored-proc:sig:article-free-raj-dahya}
yields a quick proof of
\Cref{prop:solution-to-pre-evolution-problem-is-evolution-family:sig:article-free-raj-dahya}:

\def\beweislabel{prop:solution-to-pre-evolution-problem-is-evolution-family:sig:article-free-raj-dahya}
\begin{proof}[of \Cref{prop:solution-to-pre-evolution-problem-is-evolution-family:sig:article-free-raj-dahya}]
    Towards the first algebraic property of an evolution family,
    let $t \in \interval$ be arbitrary.
    One computes
        $
            \mathcal{T}^{\Xi}(t, t)
            = \revProd_{k=1}^{N(\Xi)}
                T_{\tau^{\Xi(t,t)}_{k}}(\delta\tau^{\Xi(t,t)}_{k})
            = \revProd_{k=1}^{N(\Xi)}
                T_{t}(0)
            = \onematrix
        $
    for $\Xi \in \mathbf{P}$,
    whence
        $
            \mathcal{T}(t,t)
            = \lim_{\Xi}\mathcal{T}^{\Xi}(t, t)
            = \onematrix
        $.
    The second algebraic property,
    \viz
        $
            \mathcal{T}(t,r)
            = \mathcal{T}(t,s)
                \:\mathcal{T}(s,r)
        $
        for $(t,s), (s,r) \in \Delta$,
    can be shown exactly as in the proof of
        \Cref{prop:solution-to-quantum-problem-is-evolution-family:sig:article-free-raj-dahya}.

    To establish $\topSOT$\=/continuity, we rely on dilations.
    By assumption,
        $\interval$ is compact,
        $\{T_{\tau}\}_{\tau\in\interval}$ is $\toplocSOT$\=/continuous,
    and
        the evolution problem for $(\BanachRaum, \mathbf{P}, T)$ has a positive solution.
    So by \Cref{lemm:reduction:pre-evolution-to-cts-monitored-proc:sig:article-free-raj-dahya},
    a continuous monitoring
        $(\tilde{\BanachRaum}, \mathbf{P}, \mathcal{P}, U)$
    exists satisfying \eqcref{eq:dilation:pre-evolution-cts-monitored-proc:sig:article-free-raj-dahya}
    and for which the monitoring problem has a positive solution.
    So
        $(\tilde{\BanachRaum}, \mathbf{P}, \mathcal{P}, U)$
    fulfils the conditions of
        \Cref{prop:solution-to-quantum-problem-is-evolution-family:sig:article-free-raj-dahya}
    with $m=1$,
    making $\{(\mathcal{P} \ltimes U)(t,s)\}_{(t,s) \in \Delta}$
    a(n $\topSOT$\=/continuous) pseudo evolution family.
    Via the Banach space dilation in \eqcref{eq:dilation:pre-evolution-cts-monitored-proc:sig:article-free-raj-dahya},
    $\mathcal{T}$ inherits the $\topSOT$\=/continuity of
    $\mathcal{P} \ltimes U$.
\end{proof}

\begin{rem}
    By \Cref{%
        prop:solution-to-pre-evolution-problem-is-evolution-family:sig:article-free-raj-dahya,%
        prop:solution-to-quantum-problem-is-evolution-family:sig:article-free-raj-dahya%
    },
    under natural conditions,
    solutions to the evolution and monitoring problems via $m$-idempotents
    constitute evolution families.
    \Cref{lemm:reduction:pre-evolution-to-cts-monitored-proc:sig:article-free-raj-dahya}
    reveals that any such family
    (1) arising from a $\toplocSOT$\=/continuous pre-evolution
    can be reduced via \emph{Banach space dilations} to
    an evolution family
    (2) arising from a quantum process continuously monitored via passive measurements.
    Continuously monitored processes thus generalise
    a natural class of classically defined time-dependent systems.
    Note that this generalisation might be strict,
    as it is not clear how to emulate
    a process continuously monitored via $m$-idempotents
    using expressions of the form \eqcref{eq:pre-evolution:sig:article-free-raj-dahya}
    when $m \geq 2$
    (\cf the processes for Chernoff approximations
    and the Feynman construction in
    \Cref{%
        e.g.:chernoff:sig:article-free-raj-dahya,%
        e.g.:feynman:sig:article-free-raj-dahya%
    }).
\end{rem}

The means developed to establish the \Second free dilation theorem,
\viz the generalisation of the Trotter--Kato theorem
(see \S{}\ref{sec:result-concrete:trotter-kato:sig:article-free-raj-dahya}),
thus permit the following metaphysical interpretation:
Under modest assumptions,
a fundamental physical phenomenon
can be found to lurk behind any time-dependent system,
\viz unitary evolution continuously monitored by a time-dependent measurement process.






\null


\paragraph{Acknowledgement.}
The author is grateful
to Orr~Shalit
    for helpful suggestions to mend the na\"{i}ve approach
    (\cf \Cref{rem:naive-approach:sig:article-free-raj-dahya}),
    including the use of Sarason's and Cooper's theorems
    \cite{%
        Sarason1965Article,%
        Cooper1947Article%
    };
to Yana~Kinderknecht~(Butko)
    for useful information and literature references
    \cite{%
        SmolyanovWeizsaeckerWittich2000chernoff,%
        SmolyanovVWeizsackerWittich2003Incollection,%
        SmolyanovWeizsaeckerWittich2007chernoff,%
        Butko2020chernoff%
    }
    on the Chernoff approximation;
to Tanja~Eisner
    for her detailed feedback;
and to the referee
    for their careful reading
    and keen suggestions
    which helped to improve this paper.


\bibliographystyle{siam}
\def\bibname{References}
\bgroup

\egroup


\addresseshere
\end{document}
